\documentclass[runningheads]{llncs}

\usepackage[british]{babel}
\usepackage[T1]{fontenc}
\usepackage[utf8]{inputenc}
\usepackage{lmodern}
\usepackage[english]{isodate}

\usepackage[ruled,norelsize]{algorithm2e}

\usepackage{amsmath}
\usepackage{amssymb}
\usepackage{bm}

\DeclareMathOperator{\sign}{sign}
\DeclareMathOperator{\twoscomp}{SI}
\DeclareMathOperator{\unsignedint}{UI}
\DeclareMathOperator{\fl}{fl}
\DeclareMathOperator{\round}{round}
\DeclareMathOperator{\takum}{\tau}
\DeclareMathOperator{\takuminv}{\tau^{\mathrm{inv}}}
\DeclareMathOperator{\truncate}{truncate}

\newcommand{\euler}{\sqrt{\mathrm{e}}}

\usepackage{graphicx}
\usepackage{pgf, tikz}
\usepackage{pgfplots}
\pgfplotsset{compat=1.17}
\usepgfplotslibrary{groupplots}
\usepackage{placeins}

\usepackage{subcaption}
\usepackage{pgfplots}

\usepackage{nameref}
\usepackage[hidelinks, unicode]{hyperref} 

\usepackage[style=numeric]{biblatex}
\usepackage[autostyle=true]{csquotes}
\usepackage{enumitem}
\usepackage{listingsutf8}
\usepackage{orcidlink}
\usepackage{xcolor}
\usepackage{colortbl}
\usepackage{siunitx}
\usepackage{multicol}

\definecolor{sign}{HTML}{b02a2d}
\definecolor{direction}{HTML}{007900}
\definecolor{regime}{HTML}{8c399e}
\definecolor{characteristic}{HTML}{1f5dc2}
\definecolor{mantissa}{HTML}{636363}
\definecolor{error}{HTML}{BD002A}
\definecolor{cellbg}{HTML}{EDEDED}

\definecolor{p-sign}{HTML}{FF5454}
\definecolor{p-regime}{HTML}{CC9966}
\definecolor{p-regime-term}{HTML}{996633}
\definecolor{p-exponent}{HTML}{0080FF}
\definecolor{p-fraction}{HTML}{000000}

\DeclareSIUnit{\parsec}{pc}
\DeclareSIUnit{\electronvolt}{eV}

\makeatletter
\def\lst@makecaption{%
  \def\@captype{table}%
  \@makecaption
}
\makeatother

\begin{document}

\title{Beating Posits at Their Own Game:\\Takum Arithmetic}
\titlerunning{Beating Posits at Their Own Game: Takum Arithmetic}
\author{Laslo Hunhold\,\orcidlink{0000-0001-8059-0298}}
\authorrunning{L. Hunhold}
\institute{%
	Parallel and Distributed Systems Group\\
	University of Cologne, Cologne, Germany\\
	\email{hunhold@uni-koeln.de}
}
\maketitle
\begin{abstract}
Recent evaluations have highlighted the tapered posit number format
as a promising alternative to the uniform precision IEEE 754 floating-point 
numbers, which suffer from various deficiencies. Although the posit encoding 
scheme offers superior coding efficiency at values close to unity, its 
efficiency markedly diminishes with deviation from unity.
This reduction in efficiency leads to suboptimal encodings and a consequent 
diminution in dynamic range, thereby rendering posits suboptimal for
general-purpose computer arithmetic.
\par
This paper introduces and formally proves \enquote{takum} as a novel 
general-purpose
logarithmic tapered-precision number format, synthesising the advantages of 
posits in low-bit
applications with high encoding efficiency for numbers distant from unity.
Takums exhibit an asymptotically constant
dynamic range in terms of bit string length, which is delineated in the paper 
to be suitable for a general-purpose number format. It is demonstrated that 
takums either match or surpass existing alternatives. Moreover, takums address 
several issues previously identified in posits while unveiling novel and 
beneficial arithmetic properties.
\end{abstract}
\section{Introduction}
The fundamental premise of machine number systems is to effectively represent 
numerical values within computers using bit strings. This paper assumes the 
reader's acquaintance with two prominent methodologies: the IEEE 754 
floating-point format (refer to \cite{ieee754-2019}) and posits (refer to 
\cite{posits-standard-2022}). For a comprehensive discussion on both, readers 
are directed to \cite{posits-every_bit_counts-2024}.
The assessment of merits and drawbacks within a number system is contingent 
upon the selected criteria. To lay the groundwork, we will outline a possible
set of essential properties of an ideal number system in the subsequent 
discussion.
\subsection{\textsc{Gustafson} Criteria}\label{sec:gustafson_criteria}
\textsc{Gustafson} outlines ten criteria for assessing a robust number system 
in \cite{posits-every_bit_counts-2024}. We slightly extend their scope and 
designate specific keywords, yielding the subsequent list of properties:
\begin{enumerate}
	\item{
		\emph{Distribution}:
		The distribution of numbers within the system should accurately reflect 
		those used in calculations. Each bit string should be utilised 
		approximately an equal number of times.
	}
	\item{
		\emph{Uniqueness}:
		There should be only one possible bit string for each encoded value.
	}
	\item{
		\emph{Generality}:
		The number system should be defined for bit strings of any length.
	}
	\item{
		\emph{Statelessness}:
		The computation within the number system should be unaffected by any 
		state other than the immediate input data.
	}
	\item{
		\emph{Exactness}:
		Mathematical operations should not introduce errors exceeding the 
		precision of the number system itself. Closure under arithmetic 
		operations is desirable.
	}
	\item{
		\emph{Binary Monotonicity}:
		The mapping of real numbers to bit strings, interpreted as
		two's complement signed integers, should be monotonic.
	}
	\item{
		\emph{Binary Negation}:
		Negating the bit string as a two's complement signed integer
		should negate the represented real number and remain invariant for
		nonnumbers.
	}
	\item{
		\emph{Flexibility}:
		The conversion between bit strings of different lengths should be 
		straightforward.
	}
	\item{
		\emph{$\mathrm{NaR}$ Propagation}:
		If an operation yields a $\mathrm{NaR}$ (not a real), it 
		should be propagated through all subsequent operations.
	}
	\item{
		\emph{Implementation Simplicity}:
		The number system should be both simple and efficient to implement in 
		hardware, considering factors such as transistor count, energy 
		efficiency, latency, and throughput.
	}
\end{enumerate}
The adoption of a two's complement representation is inherently advantageous 
for preserving monotonicity and simplifying negation, chiefly because it 
enables the reutilization of existing components in hardware implementations. 
In this paper, the \textsc{Gustafson} criteria serve as pivotal benchmarks for 
assessing a spectrum of number systems.
\subsection{Dynamic Range Criteria}\label{subsec:dynamic_range_criteria}
Another crucial aspect, particularly related to property 1, is the dynamic 
range of a given number system -- the smallest and largest absolute values it 
can effectively represent.
Irrespective of the chosen dynamic range we propose the following two desirable properties for 
the dynamic range of a number system:
\begin{enumerate}[label=1\alph*)]
	\item{
		The largest and smallest positive, as well as the largest and smallest
		negative representable numbers should be reciprocals of each other.
	}
	\item{
		The dynamic range should be reasonably bounded on both ends as the bit 
		string length approaches infinity.
	}
\end{enumerate}
These properties ensure that increasing the bit-length at a specific point 
enhances precision exclusively, without unnecessarily extending the dynamic 
range. An additional implicitly desirable characteristic is the achievement of 
boundary values even with minimal bit string lengths. However, adhering to 
property 1a presupposes a highly efficient encoding scheme that rapidly 
achieves the maximal dynamic range. Consequently, it allows us, without any 
loss of generality, to focus our discussion primarily on the extent of 
the upper boundary. Property 1 stipulates that the defined boundary should 
reflect the magnitudes commonly encountered in computational processes, whilst 
also accommodating the binary representation of such an upper limit. Should an 
efficient representation for this upper boundary be ascertained, the 
determination of the lower boundary will naturally follow in accordance with 
property 1a.
\par
This aspect can be scrutinised from multiple perspectives, encompassing the 
practical numbers employed in computations, as well as the efficiency of 
encoding and the hardware implementation. It is noteworthy at this juncture 
that for every proposed upper limit, there inevitably exist extreme 
counterexamples. Nonetheless, the principal aim here is to establish a dynamic 
range that exhibits overall compatibility with general-purpose computational 
tasks. While augmenting the dynamic range through an elevation of the upper 
limit is plausible, it invariably entails a compromise on the coding efficiency 
for the encompassed numbers. Moreover, it is imperative to delineate between 
crafting a new format tailored for application-specific representations and 
devising a novel general-purpose arithmetic number format. The aspiration of 
this research is to foster the latter.
\par
All number systems examined in this work, inclusive of the one we introduce, 
adhere to the conventional structure whereby a base is elevated to the power of 
an exponent. With the base being constant, the dynamic range of the format is 
delineated by the maximum and minimum values that the exponent can assume. 
Typically, the exponent is represented as a binary integer, and the base is 
commonly set to $2$. However, within the scope of this paper, we will adopt 
base $\euler \approx 1.649$ for the newly proposed number format and use it as a reference 
base. The rationale behind selecting base $\euler$, in contrast to the 
traditional base $2$, is elaborated upon in 
Section~\ref{subsec:choice_of_base}. Consequently, the discourse 
concerning the maximum integral exponent is framed in the context of base 
$\euler$, rather than $2$.
\par
Certain choices in the binary representation of the maximum exponent offer 
distinct advantages. When the exponent is represented 
as a binary number, an ideal feature is the full utilization -- or saturation -- of 
the exponent's bit string. Such saturation is evident in exponents like 
$2^1-1$, $2^3-1$, $2^7-1$, $2^{15}-1$, $2^{31}-1$, $2^{63}-1$, $2^{127}-1$, 
$2^{255}-1$, $2^{511}-1$, et cetera.
Within the framework of a tapered number format, another beneficial property is 
the efficient encoding of the exponent's bit-length. Therefore, our goal is to 
ensure that the bit-length of the maximum exponent corresponds to a saturated 
integer, adhering to the egxponentially growing sequence given by $k \mapsto 
2^{2^k}-1$, which unfolds as $1, 3, 15, 255, 65535$, and beyond.
\par
Upon examining the integers within this sequence, it becomes immediately 
apparent that the dynamic ranges $\left(\euler^{-1}, \euler^1\right) \approx (0.6, 1.6)$, 
$\left(\euler^{-3}, \euler^{3}\right) \approx (0.2, 4.5)$, and $\left(\euler^{-15}, 
\euler^{15}\right) \approx \left(\num{5.5e-4}, \num{1.8e3}\right)$ are insufficiently expansive 
for general-purpose computational applications. Conversely, the respective
dynamic range $\left(\euler^{-65535}, \euler^{65535}\right) \approx \left(\num{1.8e-14231}, \num{5.6e14230}\right)$ 
is found to be excessively vast, leading to an inefficient allocation of bit 
representations for numerals that are seldom, if ever, employed in 
computational tasks.
The integer $255$ remains as the sole candidate, offering a dynamic range of 
$(\euler^{-255}, \euler^{255}) \approx (\num{4.2e-56}, \num{2.4e55})$. This 
range holds promise for general-purpose arithmetic, potentially 
encompassing the spectrum of numbers frequently used in computations without 
extending into the realm of excessive magnitude.
The subsequent analysis aims to ascertain whether this dynamic range aligns 
with the practical demands of computational applications.
\par
Identifying numbers of significance for computations presents a formidable 
challenge. One might consider analysing experimental data; however, the 
selection of representative datasets is inherently difficult due to its 
subjective nature and the potential for biases specific to certain 
applications. \textsc{Quevedo} historically delineated a desirable dynamic 
range from $10^{-99} \approx \euler^{-658}$ to $10^{99} \approx 
\euler^{658}$ \cite[583]{quevedo-1914}, a decision which, albeit practical for 
maintaining two decimal exponent digits, can be criticised for its pragmatic 
rather than theoretical basis.
\par
An alternative methodology involves examining attributes that characterise the 
known universe. These include its age (\SI{4.35e17}{\second}), diameter 
(\SI{8.7e26}{\metre}), mass (\SI{3e52}{\kilogram}), density 
(\SI{9.9e-27}{\kilogram\per\cubic\metre}), along with the cosmological constant 
(\num{1.1056e-52}) and the \textsc{Hubble} constant (\SI{1e-18}{\per\second}). 
In addition, dimensionless constants from the standard model -- such as the 
electron neutrino mass (\num{9e-30}), representing the smallest value and thus 
establishing a lower bound -- offer benchmarks for consideration. Extreme values, 
including the cosmological constant ($\Lambda$) at 
$\SI{1.1056e-52}{\per\meter\squared}$ and the universe's mass at 
$\SI{1.5e53}{\kilogram}$, alongside the parameters defining the International 
System of Units (SI), provide further context. Tables~\ref{tab:relative_error} 
and \ref{tab:relative_error-large_constants} in Section~\ref{sec:evaluation-relative_error}
offer an illustrative overview of these considerations.
\par
It is also noteworthy that the rescaling of large numbers in computations is a 
common practice, contingent upon the specific application. For instance, 
astronomical distances, typically vast, are often expressed in parsecs 
($\SI{1}{\parsec} = \SI{3.0857e16}{\meter}$), whereas diminutive energy values, 
prevalent in fields such as atomic, nuclear, and particle physics, are commonly 
denoted in electronvolts rather than joules ($\SI{1}{\electronvolt} = 
\SI{1.602176634e-19}{\joule}$). This rescaling extends into everyday contexts 
as well; the Kelvin temperature scale, from which the Celsius scale is derived, 
represents a rescaled unit reflecting the average relative thermal energy of 
particles within a given gas ($\SI{1}{\kelvin} = \SI{1.380649e-23}{\joule}$).
This observation underscores the notion that upper bounds derived from physical 
constants may tend towards the higher end of the spectrum, as various fields 
and applications routinely rescale numerical values to facilitate practical 
manipulations within more manageable ranges.
\par
In conclusion, fortuitously, the integer $255$ not only demonstrates 
advantageous bit string attributes but also aligns within the pre-defined 
bounds for exponents encountered in natural phenomena. Therefore, it is 
judicious to adopt the boundaries $\euler^{-255} 
\approx \num{4.2e-56}$ and $\euler^{255} \approx \num{2.4e55}$ as the dynamic 
range for the discourse within this paper.
\section{IEEE 754 Floating-Point Numbers}
This paper assumes familiarity with the IEEE 754 floating-point format and 
directs readers to \cite{posits-every_bit_counts-2024} for a comprehensive 
introduction, delving into the complexities inherent in the multitude of 
special cases and intricacies intrinsic to the standard. To aid the reader, 
Table~\ref{tab:ieee-parameters} furnishes the fixed number of exponent and 
fraction bits, denoted as $n_e$ and $n_f$ respectively, alongside the dynamic 
range corresponding to each admissible bit string length. Additionally, this 
table encapsulates recent proprietary parameterizations that will serve as 
benchmarks for subsequent comparisons.
\par
\begin{figure}[htbp]
	\begin{center}
		\begin{tikzpicture}
			\draw[<->] (0.0, 0.7) -- (0.4, 0.7) node[above,pos=.5] {sign};
			\draw[<->] (0.4, 0.7) -- (4.5, 0.7) node[above,pos=.5] {exponent};
			\draw[<->] (4.5, 0.7) -- (10.0, 0.7) node[above,pos=.5] {fraction};
	
			\draw (0,  0  ) rectangle (0.4,0.5) node[pos=.5] {};
			\draw (0.4,0  ) rectangle (4.5,0.5) node[pos=.5] {};
			\draw (4.5,0  ) rectangle (10.0,0.5) node[pos=.5] {};
	
			\draw[<->] (0.0, -0.2) -- (0.4, -0.2) node[below,pos=.5] {$1$};
			\draw[<->] (0.4, -0.2) -- (4.5, -0.2) node[below,pos=.5] {$n_e$};
			\draw[<->] (4.5, -0.2) -- (10.0, -0.2) node[below,pos=.5] {$n_f$};
		\end{tikzpicture}
	\end{center}
	\caption{IEEE 754 floating-point number bit string format.}
	\label{fig:ieee-bit_string}
\end{figure}
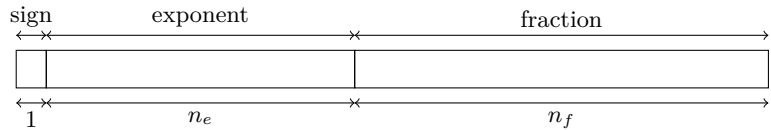
\begin{table}[htbp]
	\caption{Overview of different IEEE 754 floating-point formats, their dynamic
	ranges and the ratio of bit strings that are redundant or exceeding the dynamic range
	 $\euler^{-255} \approx \num{4.2e-56}$ to $\euler^{255} \approx \num{2.4e55}$ (\enquote{waste}).}
	\label{tab:ieee-parameters}
	\small
	\begin{center}
		\bgroup
		\def\arraystretch{1.2}
		\setlength{\tabcolsep}{0.15em}
		\begin{tabular}{| l || l | l | l || l | l | r || l | l |}
			\hline
			name & $n$ & $n_e$ & $n_f$ & smallest & largest & waste/\% & notes\\
			\hline\hline
			\texttt{float8} & 8 & 4 & 3 & $\approx \num{2.0e-3}$ & $\approx \num{2.4e2}$ & 5.08 & \scriptsize not standardised\\
			\hline
			\texttt{float16} & 16 & 5 & 10 & $\approx \num{6.0e-8}$ & $\approx \num{6.6e4}$ & 3.12 &\\
			\hline
			\texttt{bfloat16} & 16 & 8 & 7 & $\approx \num{1.2e-38}$ & $\approx \num{3.4e38}$ & 0.78 & \scriptsize \cite{bfloat16}, no subnormals\\
			\hline
			\texttt{TF32} & 19 & 8 & 10 & $\approx \num{1.2e-38}$ & $\approx \num{3.4e38}$ & 0.78 & \scriptsize \cite{tf32}, no subnormals\\
			\hline
			\texttt{float32} & 32 & 8 & 23 & $\approx \num{1.4e-45}$ & $\approx \num{3.4e38}$ & 0.39 &\\
			\hline
			\texttt{float64} & 64 & 11 & 52 & $\approx \num{4.9e-324}$ & 
			$\approx \num{1.8e308}$ & 82.03 &\\
			\hline
			\texttt{float128} & 128 & 15 & 112 & $\approx \num{6.5e-4966}$ & 
			$\approx \num{1.2e4932}$ & 98.88 &\\
			\hline
			\texttt{float256} & 256 & 19 & 236 & $\approx \num{2.2e-78984}$ & 
			$\approx \num{1.6e78913}$ & 99.93 &\\
			\hline
		\end{tabular}
		\egroup
	\end{center}
\end{table}
The proprietary $\texttt{bfloat16}$ (brain float) and $\texttt{TF32}$ 
(TensorFloat-32) formats were developed as alternatives to the 
$\texttt{float16}$ format, which exhibits insufficient dynamic range due to its 
small exponent. This limitation has become especially prominent in machine 
learning applications in recent years \cite{tf32}. Notably, these formats omit 
subnormal numbers, implying that any number smaller than the smallest normal 
number is simply rounded to zero. The rationale behind this omission is the 
costliness of implementing subnormal numbers, which is deemed unwarranted for 
the minimal gain in dynamic range. However, it's essential to acknowledge that 
this choice deviates from the IEEE 754 floating-point standard.
\par
An additional metric presented in Table~\ref{tab:ieee-parameters} is the ratio 
of \enquote{wasted} bit strings. This metric encompasses not only redundant 
$\mathrm{NaN}$ or unused 
subnormal representations but also all binary representations of values 
exceeding the previously determined target dynamic range from $\euler^{-255} 
\approx \num{4.2e-56}$ to $\euler^{255} \approx \num{2.4e55}$.
It is important to note that the metric of \enquote{waste} employed in this 
paper serves as a quantitative measure to elucidate the inefficiency of a 
number system that offers any excessively high dynamic range. This serves to 
illustrate the tendency for human intuition to underestimate the drawbacks of 
such designs and combinatoric effects. It is not our intention to enshrine the
previously determined dynamic range $\euler^{-255}$ to $\euler^{255}$ as an
immutable benchmark, but rather to present it as one potential reference point.
It holds
\begin{proposition}\label{prop:ieee-excess}
	Assume an IEEE 754 floating-point format with $n_e$ exponent
	bits and $n_f$ fraction bits. The ratio of redundant bit strings and
	bit strings representing numbers exceeding
	$\pm\left(\euler^{-255},\euler^{255}\right)$ is approximately
	\begin{multline}
	\left[1 + (n_e \ge 10) \cdot (2^{n_e} - 734) + (n_e \ge 10 \lor \text{no 
	subnormals}) \right] \cdot 2^{-n_e}-\\
		3 \cdot 2^{-1-n_e-n_f}.
	\end{multline}
\end{proposition}
\begin{proof}
	See Section~\ref{sec:proof-ieee-excess}.
\end{proof}
Please refer to Table~\ref{tab:ieee-parameters} for an overview of values for 
the different IEEE 754 and proprietary floating-point types.
Concerning the earlier mentioned \textsc{Gustafson} criteria, it can be argued 
that IEEE 754 floating-point numbers fail to satisfy any of
them \cite{posits-beating_floating-point-2017}. Notably, \texttt{float16} and
\texttt{float32} exhibit inadequate dynamic range
relative to the desired $\euler^{-255} \approx \num{4.2e-56}$ and $\euler^{255} \approx \num{2.4e55}$. Conversely, \texttt{float64} and larger types exhibit an 
excessive dynamic range, violating property~1. The resultant ratio of wasted
bit strings due to excessive dynamic range, occurring only for $n_e 
\ge 9$, leads to a significant number of unused bit strings in the respective 
IEEE 754 floating-point formats. For instance, double-precision floating-point 
numbers (\texttt{float64}) allocate approximately $82\%$ of their 
available bit strings to representations that are unlikely to be used in 
calculations, violating properties 1a and 1b. This becomes even more
profound for \texttt{float128} and \texttt{float256}, where approximately
$99\%$ and almost $100\%$ are respectively wasted.
\par
Beyond issues of dynamic range, the presence of numerous redundant $\mathrm{NaN}$ 
representations directly contravenes property~2. This issue particularly impacts
smaller IEEE 754 floating-point formats, leading to a significant waste
of bit patterns -- $5.08\%$ and $3.12\%$ in \texttt{float8} and \texttt{float16},
respectively. The failure to satisfy 
properties 6 and 7 introduces a notable overhead in hardware. Concerning 
$\mathrm{NaR}$ propagation, as illustrated by examples such as \texttt{pow(NaN, 0) = 
1} and \texttt{pow(1, NaN) = 1} for \enquote{quiet} $\mathrm{NaN}$'s, there are violations
of property~9. Further discussion on these issues is available in 
\cite{posits-beating_floating-point-2017}.
\section{Posits}
Since their introduction by \textsc{Gustafson} and \textsc{Yonemoto} in 
2017\cite{posits-beating_floating-point-2017}, posits, as an alternative 
number 
system to IEEE 754 floating-point numbers, have undergone extensive analysis 
and evolved into a standard draft \cite{posits-standard-2022}. Posits now
represent the current state of the art for tapered floating-point formats,
which go back to \textsc{Morris} \cite{1971-tapered_floating_point},
and are used in a wide range of fields\cite{johnson-2018,2024-log-posit}.
The standard draft \cite{posits-standard-2022} defines posits as follows:
\begin{definition}[posit encoding]\label{def:posit}
Let $n \in \mathbb{N}$ with $n \ge 5$. Any
$n$-bit MSB$\rightarrow$LSB string of the form
\begin{center}
	\begin{tikzpicture}
		\draw[<->] (0.0, 0.7) -- (0.4, 0.7) node[above,pos=.5] {sign};
		\draw[<->] (0.4, 0.7) -- (4.5, 0.7) node[above,pos=.5] {exponent};
		\draw[<->] (4.5, 0.7) -- (10.0, 0.7) node[above,pos=.5] {fraction};

		\draw (0,  0  ) rectangle (0.4,0.5) node[pos=.5] {\textcolor{p-sign}{$S$}};
		\draw (0.4,0  ) rectangle (3.2,0.5) node[pos=.5] {\textcolor{p-regime}{$R$}};
		\draw (3.2,0  ) rectangle (3.7,0.5) node[pos=.5] {\textcolor{p-regime-term}{$\overline{R_0}$}};
		\draw (3.7,0  ) rectangle (4.5,0.5) node[pos=.5] {\textcolor{p-exponent}{$E$}};
		\draw (4.5,0  ) rectangle (10.0,0.5) node[pos=.5] {$F$};

		\draw[<->] (0.0, -0.2) -- (0.4, -0.2) node[below,pos=.5] {$1$};
		\draw[<->] (0.4, -0.2) -- (3.2, -0.2) node[below,pos=.5] {$k$};
		\draw[<->] (3.2, -0.2) -- (3.7, -0.2) node[below,pos=.5] {$1$};
		\draw[<->] (3.7, -0.2) -- (4.5, -0.2) node[below,pos=.5] {$2$};
		\draw[<->] (4.5, -0.2) -- (10.0, -0.2) node[below,pos=.5] {$p$};
	\end{tikzpicture}
\end{center}
with
\begin{align}
	\textcolor{p-sign}{S} &&\colon \parbox{3.0cm}{sign bit}\\
	\textcolor{p-regime}{R} &:=
			(\textcolor{p-regime}{R}_{k-1},\dots,\textcolor{p-regime}{R}_0)
			&\colon \parbox{3.0cm}{regime bits}\\
	\textcolor{p-regime-term}{\overline{R_0}} &&\colon
			\parbox{3.0cm}{regime termination bit}\\
	r &:= \begin{cases}
			-k & \textcolor{p-regime}{R}_0 = 0\\
			k-1 & \textcolor{p-regime}{R}_0 = 1
		\end{cases}
		&\colon \parbox{3.0cm}{regime}\\
	\textcolor{p-exponent}{E} &:=(\textcolor{p-exponent}{E}_{1},
		\textcolor{p-exponent}{E}_0) &\colon \parbox{3.0cm}{exponent bits}\\
	e &:= 2 \textcolor{p-exponent}{E}_1 + \textcolor{p-exponent}{E}_0
		&\colon \parbox{3.0cm}{exponent} \\
	p &:= n - k - 4 \in \{ 0,\dots,n-5 \} &\colon
		\parbox{3.0cm}{fraction bit count}\label{eq:posit-m}\\
	\textcolor{p-fraction}{F} &:= (\textcolor{p-fraction}{F}_{p-1},\dots,
		\textcolor{p-fraction}{F}_0) \in {\{ 0,1\}}^p &\colon
		\parbox{3.0cm}{fraction bits}\\
	f &:= 2^{-p} \sum_{i=0}^{p-1} \textcolor{p-fraction}{F\!}_i 2^i \in [0,1) &\colon
		\parbox{3.0cm}{fraction}\\
	\hat{e} &:= {(-1)}^{\textcolor{sign}{S}} (4 r + e + \textcolor{sign}{S})
		&\colon \parbox{3.0cm}{\enquote{actual} 
		exponent}\label{eq:posit-exponent}
\end{align}
encodes the posit value
\begin{equation}
	\pi((\textcolor{p-sign}{S},\textcolor{p-regime}{R},\textcolor{p-regime-term}{\overline{R_0}},
		\textcolor{p-exponent}{E},\textcolor{p-fraction}{F})) := \begin{cases}
		\begin{cases}
			0 & \textcolor{p-sign}{S} = 0\\
			\mathrm{NaR} & \textcolor{p-sign}{S} = 1
		\end{cases}
			& \textcolor{p-regime}{R} = \textcolor{p-regime-term}{\overline{R_0}} =
				\textcolor{p-exponent}{E} = \textcolor{p-fraction}{F} = \bm{0} \\
		[(1-3 \textcolor{p-sign}{S}) + f] \cdot
		2^{\hat{e}} & \text{otherwise}.
	\end{cases}\label{eq:def-posit-value}
\end{equation}
with $\pi \colon {\{0,1\}}^n \mapsto \{ 0,\mathrm{NaR} \} \cup
\pm\left[2^{-4n+8},2^{4n-8}\right]$.
Without loss of generality, any bit string shorter than 5 bits is also
included in the definition by assuming the missing bits to be
zero bits (\enquote{ghost bits}). The colour scheme for the different bit
string segments was adopted from the standard \cite{posits-standard-2022}.
\end{definition}
Posits were explicitly designed to satisfy all \textsc{Gustafson} criteria 
(refer to \cite{posits-every_bit_counts-2024}). However, there has been a 
modification in the standard draft \cite{posits-standard-2022} to further accommodate
property~8 after the publication of \cite{posits-beating_floating-point-2017}:
The number of exponent bits was fixed at $2$, a departure from the previous
variation depending on $n$. While a fixed exponent size facilitates easy conversion
between different precisions, it introduces a trade-off between precision and
dynamic range.
\par
Another alteration since \cite{posits-beating_floating-point-2017} is the 
replacement of $\infty$ with $\mathrm{NaR}$. This change was made in favour of 
property~9, enabling the propagation of nonreal values instead of termination, 
albeit at the cost of allowing division by zero and $\infty$. While valid 
arguments support the inclusion of either $\mathrm{NaR}$ or $\infty$, the 
underlying universal wheel algebra \cite{wheels-2004} is defined with both 
elements (a bottom element $\bot := 0/0$ and infinity $\infty$). However, 
including both elements is impractical as it would compromise the symmetry of 
the model. The handling conventions for $\mathrm{NaR}$ are further discussed in
Section~\ref{subsec:nar-convention}.
\par
We will now examine the distribution of numbers within posits' dynamic range, 
which spans from $2^{-4n+8}$ to $2^{4n-8}$ (see 
Table~\ref{tab:posit-parameters} for an overview).
\begin{table}[htbp]
	\caption{Overview of different posit precisions and their dynamic ranges.}
	\label{tab:posit-parameters}
	\small
	\begin{center}
		\bgroup
		\def\arraystretch{1.2}
		\setlength{\tabcolsep}{0.15em}
		\begin{tabular}{| l || l || l | l | r |}
			\hline
			name & $n$ & smallest & largest\\
			\hline\hline
			\texttt{posit8} & 8 & $\approx \num{5.96e-8}$ & $\approx \num{1.68e7}$ \\
			\hline
			\texttt{posit16} & 16 & $\approx \num{1.39e-17}$ & $\approx \num{7.21e16}$ \\
			\hline
			\texttt{posit32} & 32 & $\approx \num{7.52e-37}$ & $\approx \num{1.32e36}$ \\
			\hline
			\texttt{posit64} & 64 & $\approx \num{2.21e-75}$ & $\approx \num{4.52e74}$ \\
			\hline
			\texttt{posit128} & 128 & $\approx \num{1.91e-152}$ & $\approx \num{5.24e151}$ \\
			\hline
			\texttt{posit256} & 256 & $\approx \num{1.42e-306}$ & $\approx \num{7.02e305}$ \\
			\hline
		\end{tabular}
		\egroup
	\end{center}
\end{table}
Of particular interest is the ratio of bit strings that exceed the targeted dynamic range as previously discussed. In Proposition~\ref{prop:ieee-excess}, we discovered that \texttt{float64}, and formats with greater precision, allocate a considerable proportion of bit strings to superfluous numerical representations. In the case of posits, we can show the following
\begin{proposition}\label{prop:posit-excess}
	Let $n \in \mathbb{N}_1$. The ratio of posit bit strings of length $n$ 
	representing numbers exceeding
	$\pm\left(\euler^{-255},\euler^{255} \right)$ is approximately
	\begin{equation}
		\begin{cases}
			0 & n \le 47\\
			\frac{4 \cdot 2^{n-48}}{2^n} & n \ge 48
		\end{cases} =
		\begin{cases}
			0 & n \le 47\\
			2^{-46} & n \ge 48
		\end{cases} \approx
		\begin{cases}
			0 & n \le 47\\
			\num{1.42e-14} & n \ge 48.
		\end{cases}
	\end{equation}
\end{proposition}
\begin{proof}
	See Section~\ref{sec:proof-posit-excess}.
\end{proof}
Remarkably, the ratio of excessive bit patterns in relation to dynamic range remains
consistently small across
all $n$. It is important to note, however, that this argument pertains solely
to the dynamic range and does not assess the efficiency of the posit format
itself.
\par
Overall, alongside meticulous design choices aimed at mitigating the inherent 
redundancy in IEEE 754 
floating-point numbers, the primary quantitative disparity between them and 
posits resides in the 
variable-length exponent of posits. This characteristic affords heightened 
precision for values with 
exponents proximate to zero, which, in practical terms, constitute the numbers 
predominantly employed in 
computational tasks. However, this advantage entails a corresponding trade-off, 
leading to diminished 
precision for values characterized by large exponents. This trade-off is 
further elucidated in 
\cite{posits-beating_floating-point-2017,posits-every_bit_counts-2024}.
\par
We have demonstrated that the variable-length exponent not only enhances the 
precision of posits for 
commonly encountered values but also protects them from excessive bit 
allocation for extremely large 
numbers. This stands in stark contrast to IEEE 754 floating-point numbers, 
which squander a significant 
number of bits (refer to Table~\ref{tab:ieee-parameters}). However, it is 
important to note that the 
dynamic range of posits is relatively limited for bit string lengths most 
relevant below 64.
Furthermore, posits encode exponents of substantial magnitude relatively 
inefficiently due to the 
necessity of lengthy regimes, resulting in a scarcity of bit strings allocated 
for numbers with 
large-magnitude exponents, which puts the results of 
Proposition~\ref{prop:posit-excess} into
perspective.
\par
If one were to optimise the exponent encoding of posits, it would result in
a number format boasting a significantly expanded dynamic range and a greater
abundance of available bit strings for each exponent. 
However, a concomitant increase in the ratio of unused bit strings would ensue.
This phenomenon is exemplified in Section~\ref{sec:takum} and underpins the
adoption of a constrained dynamic range approach 
for the number system delineated in this study.
\section{Takum Arithmetic}\label{sec:takum}
The preceding sections have introduced both IEEE 754 floating-point numbers and 
posits, offering an 
extensive discourse on their respective strengths and weaknesses. Within the 
realm of posits and by 
design, the exponent coding emerges as an area ripe for enhancement, given that 
the sign and fraction bits exhibit maximal information density/entropy in both 
formats.
\par
Regarding posit exponent coding, the utilisation of a prefix code to delineate 
regimes engenders sequences of low-entropy runs (sequences of consecutive ones 
or zeros which can be interpreted to have low information content), presenting 
an opportunity for optimisation. While alternative universal codes have 
been subjected to rigorous scrutiny \cite{lindstrom-2018}, they entail a 
significant overhead, particularly evident for small exponents (see the later
discussion in Section~\ref{sec:evaluation-efficiency}).
\par
A hypothetical strategy for optimizing the posit exponent encoding entails 
defining $k$ as the number of exponent bits minus 2 and supplementing the 
prefix code with a variable bitwise representation of the exponent. An implicit 
most significant bit (MSB) of $1$ is presumed for $k>1$. While this approach 
yields coding efficiency akin to that of posits for diminutive exponents, its 
adoption engenders a considerable expansion in the dynamic range. For instance, 
with a 16-bit configuration, the largest number, 
$\textcolor{p-sign}{0}\textcolor{p-regime}{1\dots1}$, attains $k=15$. 
Consequently, the exponent spans a bit-length of $17$ (inclusive of the 
implicit $1$ bit), manifesting in the binary string $10000000000000000$. This 
corresponds to an excessively large exponent of $65536$.
\par
This expansion results in an abundance of redundant bit strings for numbers 
that significantly exceed the intended dynamic range, thereby rendering such 
bit strings superfluous. Attempts to address this issue by imposing constraints 
on the total length of regime and exponent bits disrupt the symmetry of the 
dynamic range, a characteristic deemed undesirable according to property 1a. 
The pivotal observation is that the coding of exponents must be 
\emph{intrinsically} bounded instead. Consequently, there should come a point 
where additional bits appended to the bit string contribute solely to precision 
rather than dynamic range, as they cease to impact the exponent coding.
\subsection{Definition}
Let us consider a radical question: If we constrain the dynamic range of the 
exponent, is there even a necessity for prefix codes to encode exponents of 
arbitrary magnitude? Previously, we elucidated a target dynamic range spanning 
from $\euler^{-255}$ to $\euler^{255}$, underscoring the significance of the 
binary representation of $255$ as having a length of $8$, a power of $2$. Given 
that the leading bit of any non-zero number is $1$, we can effectively encode 
the bit-length of any number ranging from $1$ to $255$ using merely three 
\enquote{regime} bits (capable of expressing any regime value between $0$ and 
$7$ for the number of \enquote{explicit} bits following the implicit one bit). 
By appending the explicit bits to the regime bits, we achieve a highly 
efficient variable-length representation for numbers larger or equal to one.
\par
To also encode zero, we subtract $1$ from the value, yielding an encoding for 
numbers within the range $0$ to $254$. It may raise the question whether this 
approach compromises our intended dynamic range. However, in reality, it 
enables us to precisely match our target dynamic range from the outset. Since 
the significand of a base-$\euler$ floating-point representation falls within the interval 
$[1,\euler)$, our aim is to confine the exponents within the range $-255$ to $254$, 
corresponding to a dynamic range of $\euler^{-255}$ to $\euler^{255}$.
Let us designate the three regime bits in \textcolor{regime}{mauve} and the 
explicit bits in \textcolor{characteristic}{blue}. The values from 0 to 7 
are encoded as follows: $\textcolor{regime}{000}\textcolor{characteristic}{}$, 
$\textcolor{regime}{001}\textcolor{characteristic}{0}$, 
$\textcolor{regime}{001}\textcolor{characteristic}{1}$, 
$\textcolor{regime}{010}\textcolor{characteristic}{00}$, 
$\textcolor{regime}{010}\textcolor{characteristic}{01}$, 
$\textcolor{regime}{010}\textcolor{characteristic}{10}$, 
$\textcolor{regime}{010}\textcolor{characteristic}{11}$, and 
$\textcolor{regime}{011}\textcolor{characteristic}{000}$. The value $254$ is encoded 
as $\textcolor{regime}{111}\textcolor{characteristic}{1111111}$ (10 bits), 
significantly shorter than the posit encoding requiring $68$ bits.
\par
One advantage of the posits' prefix codes is the ability to store an additional 
bit $\textcolor{p-regime}{R}_0$ of information, depending on whether the regime 
consists of all zeros terminated by a one bit or all ones terminated by a zero 
bit, implying $\textcolor{p-regime}{R}_0=0$ and $\textcolor{p-regime}{R}_0=1$ 
respectively. However, for encoding a complete exponent with our scheme, we 
require an additional \enquote{direction} bit instead to indicate when to apply a bias.
Based on the aforementioned encoding scheme, we propose the following
format:
\begin{definition}[takum encoding]\label{def:takum}
Let $n \in \mathbb{N}$ with $n \ge 12$. Any $n$-bit MSB$\rightarrow$LSB string 
$(\textcolor{sign}{S},\textcolor{direction}{D},\textcolor{regime}{R},
\textcolor{characteristic}{C},\textcolor{mantissa}{M}) \in {\{0,1\}}^n$ of the form
\begin{center}
	\begin{tikzpicture}
		\draw[<->] (0.0, 0.7) -- (0.4, 0.7) node[above,pos=.5] {sign};
		\draw[<->] (0.4, 0.7) -- (4.5, 0.7) node[above,pos=.5] {characteristic};
		\draw[<->] (4.5, 0.7) -- (10.0, 0.7) node[above,pos=.5] {mantissa};

		\draw (0,  0  ) rectangle (0.4,0.5) node[pos=.5] {\textcolor{sign}{S}};
		\draw (0.4,0  ) rectangle (0.8,0.5) node[pos=.5] {\textcolor{direction}{D}};
		\draw (0.8,0  ) rectangle (2.0,0.5) node[pos=.5] {\textcolor{regime}{R}};
		\draw (2.0,0  ) rectangle (4.5,0.5) node[pos=.5] {\textcolor{characteristic}{C}};
		\draw (4.5,0  ) rectangle (10.0,0.5) node[pos=.5] {\textcolor{mantissa}{M}};

		\draw[<->] (0.0, -0.2) -- (0.4, -0.2) node[below,pos=.5] {$1$};
		\draw[<->] (0.4, -0.2) -- (0.8, -0.2) node[below,pos=.5] {$1$};
		\draw[<->] (0.8, -0.2) -- (2.0, -0.2) node[below,pos=.5] {$3$};
		\draw[<->] (2.0, -0.2) -- (4.5, -0.2) node[below,pos=.5] {$r$};
		\draw[<->] (4.5, -0.2) -- (10.0, -0.2) node[below,pos=.5] {$p$};
	\end{tikzpicture}
\end{center}
with
\begin{align}
	\textcolor{sign}{S} &\in \{0,1\} &\colon \parbox{2.9cm}{sign bit}\\
	\textcolor{direction}{D} &\in \{0,1\} &\colon \parbox{2.9cm}{direction 
	bit}\\
	\textcolor{regime}{R} &:=
		(\textcolor{regime}{R}_2,\textcolor{regime}{R}_1,\textcolor{regime}{R}_0)
		\in {\{0,1\}}^3
		&\colon \parbox{2.9cm}{regime bits}\\
	r &:= \begin{cases}
		7 - (4 \textcolor{regime}{R}_2 + 2 \textcolor{regime}{R}_1 +
			\textcolor{regime}{R}_0) & \textcolor{direction}{D} = 0\\
		4 \textcolor{regime}{R}_2 + 2 \textcolor{regime}{R}_1 + \textcolor{regime}{R}_0	&
			\textcolor{direction}{D} = 1
		\end{cases}
		\in \{0,\!\dots\!,\!7\}
		&\colon \parbox{2.9cm}{regime}\label{eq:takum-regime}\\
	\textcolor{characteristic}{C} &:=(\textcolor{characteristic}{C}_{r-1},\dots,
		\textcolor{characteristic}{C}_0) \in {\{0,1\}}^r &\colon 
		\parbox{2.9cm}{characteristic bits}\label{eq:takum-characteristic-bits}\\
	c &:= \begin{cases}
			-2^{r+1} + 1 + \sum_{i=0}^{r-1}
				\textcolor{characteristic}{C}_i 2^{i}& \textcolor{direction}{D} = 0\\
			2^r - 1 + \sum_{i=0}^{r-1}
				\textcolor{characteristic}{C}_i 2^{i}
				& \textcolor{direction}{D} = 1
		\end{cases}
		&\colon \parbox{2.9cm}{characteristic}\label{eq:takum-characteristic}\\
	p &:= n - r - 5 \in \{ n-12,\dots,n-5 \} &\colon
		\parbox{2.9cm}{mantissa bit count}\label{eq:takum-mantissa_bit_count}\\
	\textcolor{mantissa}{M} &:= (\textcolor{mantissa}{M}_{p-1},\dots,
		\textcolor{mantissa}{M}_0) \in {\{0,1\}}^p &\colon
		\parbox{2.9cm}{mantissa bits}\\
	m &:= 2^{-p} \sum_{i=0}^{p-1} \textcolor{mantissa}{M\!}_i 2^i \in [0,1) 
	&\colon
		\parbox{2.9cm}{mantissa}\label{eq:takum-mantissa}\\
	\ell &:= {(-1)}^{\textcolor{sign}{S}}(c + m)
		\in (-255,255)
		&\colon \parbox{2.9cm}{logarithmic value}\label{eq:takum-logarithmic}
\end{align}
encodes the takum value
\begin{equation}\label{eq:takum-value}
	\takum((\textcolor{sign}{S},\textcolor{direction}{D},\textcolor{regime}{R},
	\textcolor{characteristic}{C},\textcolor{mantissa}{M}))
	:= \begin{cases}
		\begin{cases}
			0 & \textcolor{sign}{S} = 0\\
			\mathrm{NaR} & \textcolor{sign}{S} = 1
		\end{cases}
			& \textcolor{direction}{D} = \textcolor{regime}{R} = \textcolor{characteristic}{C} = \textcolor{mantissa}{M} 
			= \bm{0} \\
		{(-1)}^{\textcolor{sign}{S}}
		\euler^{\ell} & \text{otherwise}
	\end{cases}
\end{equation}
with $\takum \colon {\{0,1\}}^n \mapsto \{ 0,\mathrm{NaR} \} \cup
\pm\left(\euler^{-255},\euler^{255}\right)$ and \textsc{Euler}'s number
$\mathrm{e}$.
Without loss of generality, any bit string shorter than 12 bits is also
considered in the definition by assuming the missing bits to be
zero bits (\enquote{ghost bits}).
\end{definition}
The term \enquote{takum} originates from the Icelandic phrase 
\enquote{takmarkað umfang}, translating to \enquote{limited range}. Pronounced 
initially akin to the English term \enquote{tug}, with a shortened \enquote{u} 
sound, the \enquote{g} is articulated as a \enquote{k}. The \enquote{-um} 
follows swiftly after the \enquote{k}, pronounced akin to \enquote{um} in the 
English term \enquote{umlaut}.
\par
This format specification initiates a shift in nomenclature, substituting the 
exponent bits $\textcolor{p-exponent}{E}$ and exponent $e$ with characteristic bits 
$\textcolor{characteristic}{C}$ and characteristic $c$.
This modification originates from the intrinsic logarithmic characteristics of takum arithmetic, as elaborated upon in Section~\ref{subsec:lns}. Such a foundation more aptly embodies the substitution of the traditional integral exponent with a logarithmic value, $\ell$, which consists of a characteristic, $c$ -- the former representing the integral portion, and the latter, a mantissa, $m$, delineating the fractional portion. In comparison to the posit 
definition (refer to Definition~\ref{def:posit}), this adjustment notably 
simplifies (\ref{eq:def-posit-value}).
\par
Additionally, we introduce a takum colour scheme, prioritising uniformity in 
both lightness and chroma within the perceptually uniform OKLCH colour
space \cite{2023-colour}. Detailed colour 
definitions are delineated in Table~\ref{tab:colour_scheme}.
\begin{table}
	\caption{Overview of the takum arithmetic colour scheme.}
	\label{tab:colour_scheme}
	\small
	\begin{center}
		\bgroup
		\def\arraystretch{1.2}
		\setlength{\tabcolsep}{0.15em}
		\begin{tabular}{| l || l || l | l | l |}
			\hline
			colour & identifier & OKLCH & CIELab & HEX (sRGB)\\
			\hline\hline
			\cellcolor{sign} & sign & $(50\%, 0.17, 25)$ &
				$(40.28, 53.78, 33.31)$ & \texttt{\#B02A2D}\\
			\hline
			\cellcolor{direction} & direction & $(50\%, 0.17, 142.5)$ & 	
				$(43.91, -45.71, 46.68)$ & \texttt{\#007900}\\
			\hline
			\cellcolor{regime} & regime & $(50\%, 0.17, 320)$ &
				$(39.29, 46.60, -39.19)$ & \texttt{\#8C399E}\\
			\hline
			\cellcolor{characteristic} & characteristic & $(50\%, 0.17, 260)$ &
				$(40.34, 10.54, -59.37)$ & \texttt{\#1F5DC2}\\
			\hline
			\cellcolor{mantissa} & mantissa & $(50\%, 0.00, 0)$ &
				$(42.00, 0.00, 0.00)$ & \texttt{\#636363}\\
			\hline
		\end{tabular}
		\egroup
	\end{center}
\end{table}
\par
We observe that the additional information previously represented by 
$\textcolor{p-regime}{R}_0$ is now integrated into the \enquote{direction bit} 
$\textcolor{direction}{D}$.
The designation has been chosen due to the bit's function in indicating whether the logarithmic value $\ell$ increases ($\textcolor{direction}{D}=1$) or decreases ($\textcolor{direction}{D}=0$) with the incrementation of the takum bit string. In other words, the bit signifies whether the growth of the logarithmic value is aligned with the growth of the number, which can also be referred to as both pointing in the same direction.
It is noteworthy that, in general, the total length of the characteristic bit string segment ($1+3+r$) 
never exceeds 11 bits, and that we precisely align with the targeted dynamic 
range of $\euler^{-255}$ to $\euler^{255}$.
For a succinct elucidation of the takum encoding scheme, please refer to 
Table~\ref{tab:examples}, which presents a small selection of examples.
\par
\begin{table}
	\caption{Examples illustrating the takum encoding scheme.}
	\label{tab:examples}
	\begin{center}
		\bgroup
		\def\arraystretch{1.2}
		\setlength{\tabcolsep}{0.15em}
		\begin{tabular}{| l || l | l | l | l || l |}
			\hline
			bits & $r$ & $c$ & $m$ & $\ell$ & $t$\\
			\hline\hline
			$\textcolor{sign}{0}
				\textcolor{direction}{1}
				\textcolor{regime}{}
				\textcolor{characteristic}{}
				\textcolor{mantissa}{}$ &
				$0$ &
				$0$ &
				$0$ &
				$+(0 + 0) = 0$ &
				$+\euler^{0} = 1$\\\hline
			$\textcolor{sign}{0}
				\textcolor{direction}{1}
				\textcolor{regime}{000}
				\textcolor{characteristic}{}
				\textcolor{mantissa}{001}$ &
				$0$ &
				$0$ &
				$0.125$ &
				$+(0+0.125)=0.125$ &
				$+\euler^{0.125} \approx \num{1.1}$\\\hline
			$\textcolor{sign}{1}
				\textcolor{direction}{1}
				\textcolor{regime}{}
				\textcolor{characteristic}{}
				\textcolor{mantissa}{}$ &
				$0$ &
				$0$ &
				$0$ &
				$-(0 + 0) = 0$ &
				$-\euler^{0} = -1$\\\hline
			$\textcolor{sign}{1}
				\textcolor{direction}{1}
				\textcolor{regime}{000}
				\textcolor{characteristic}{}
				\textcolor{mantissa}{001}$ &
				$0$ &
				$0$ &
				$0.125$ &
				$-(0+0.125)=-0.125$ &
				$-\euler^{-0.125} \approx \num{-0.9}$\\\hline
			$\textcolor{sign}{0}
				\textcolor{direction}{1}
				\textcolor{regime}{001}
				\textcolor{characteristic}{}
				\textcolor{mantissa}{}$ &
				$1$ &
				$1$ &
				$0$ &
				$+(1+0)=1$ &
				$+\euler$\\\hline
			$\textcolor{sign}{0}
				\textcolor{direction}{0}
				\textcolor{regime}{1}
				\textcolor{characteristic}{}
				\textcolor{mantissa}{}$ &
				$3$ &
				$-15$ &
				$0$ &
				$+(-15 + 0) = -15$ &
				$+\euler^{-15} \approx \num{5.5e-4}$\\\hline
			$\textcolor{sign}{0}
				\textcolor{direction}{1}
				\textcolor{regime}{000}
				\textcolor{characteristic}{}
				\textcolor{mantissa}{1}$ &
				$0$ &
				$0$ &
				$0.5$ &
				$+(0 + 0.5) = 0.5$ &
				$+\euler^{0.5} \approx 1.3$\\\hline
			$\textcolor{sign}{1}
				\textcolor{direction}{0}
				\textcolor{regime}{111}
				\textcolor{characteristic}{}
				\textcolor{mantissa}{1}$ &
				$0$ &
				$-1$ &
				$0.5$ &
				$-(-1+0.5)=0.5$ &
				$-\euler^{0.5} = -1.3$\\\hline
			$\textcolor{sign}{1}
				\textcolor{direction}{0}
				\textcolor{regime}{010}
				\textcolor{characteristic}{11111}
				\textcolor{mantissa}{1}$ &
				$5$ &
				$-32$ &
				$0.5$ &
				$-(-32 + 0.5) = 31.5$ &
				$-\euler^{31.5} \approx \num{-6.9e6}$\\\hline
			$\textcolor{sign}{1}
				\textcolor{direction}{0}
				\textcolor{regime}{011}
				\textcolor{characteristic}{0000}
				\textcolor{mantissa}{00}$ &
				$4$ &
				$-31$ &
				$0$ &
				$-(-31+0) = 31$ &
				$-\euler^{31} \approx \num{-5.4e6}$\\\hline
			$\textcolor{sign}{1}
				\textcolor{direction}{0}
				\textcolor{regime}{011}
				\textcolor{characteristic}{0000}
				\textcolor{mantissa}{01}$ &
				$4$ &
				$-31$ &
				$0.25$ &
				$-(-31+0.25) = 30.75$ &
				$-\euler^{30.75} \approx \num{-4.8e6}$\\\hline
			$\textcolor{sign}{1}
				\textcolor{direction}{0}
				\textcolor{regime}{000}
				\textcolor{characteristic}{0000000}
				\textcolor{mantissa}{1}$ &
				$7$ &
				$-255$ &
				$0.5$ &
				$-(-255+0.5) = 254.5$ &
				$-\euler^{254.5} \approx \num{-1.8e55}$\\\hline
			$\textcolor{sign}{1}
				\textcolor{direction}{1}
				\textcolor{regime}{111}
				\textcolor{characteristic}{1111111}
				\textcolor{mantissa}{1}$ &
				$7$ &
				$254$ &
				$0.5$ &
				$-(254+0.5)=-254.5$ &
				$-\euler^{-254.5} \! \approx \! 
					-5.4\!\times\! 10^{-56}$\\\hline
			$\textcolor{sign}{0}
				\textcolor{direction}{0}
				\textcolor{regime}{000}
				\textcolor{characteristic}{0000000}
				\textcolor{mantissa}{1}$ &
				$7$ &
				$-255$ &
				$0.5$ &
				$+(-255+0.5)=-254.5$ &
				$+\euler^{-254.5} \approx \num{5.4e-56}$\\\hline
			$\textcolor{sign}{0}
				\textcolor{direction}{1}
				\textcolor{regime}{111}
				\textcolor{characteristic}{1111111}
				\textcolor{mantissa}{1}$ &
				$7$ &
				$254$ &
				$0.5$ &
				$+(254+0.5)=254.5$ &
				$+\euler^{254.5} \approx \num{1.8e55}$\\\hline
		\end{tabular}
		\egroup
	\end{center}
\end{table}
\subsection{Rounding}
In terms of rounding, we adhere to the posit standard 
\cite[Section~4.1]{posits-standard-2022} by employing \enquote{saturation 
arithmetic}. This approach ensures that there is no under- or overflow for 
numbers outside the dynamic range; instead, they are clamped to the smallest or 
largest representable number, respectively. Saturation arithmetic finds 
justification in the fact that the error introduced by saturation is 
consistently smaller than the potentially infinite error resulting from 
overflow, and it eliminates scenarios where a number of infinitesimally small 
magnitude vanishes.
\begin{algorithm}[htb]
	\caption{Takum rounding algorithm yielding $\round_n(x)$ for a number
		$x \in \mathbb{R} \cup \{ \mathrm{NaR} \}$
		to $n \in \mathbb{N}_2$ bits. The lossless takum encoding function
		$\takuminv$ is defined in Proposition~\ref{prop:takum-encoding},
		the $\truncate_i$ function zero-extends or strips off LSBs until the
		bit string has the desired length $i \in \mathbb{N}_0$.}
	\label{alg:rounding}
	\begin{multicols}{2}
		\DontPrintSemicolon
		\SetKwInOut{Input}{input}
		\SetKwInOut{Output}{output}
		\SetNoFillComment
		\Input{%
			$x \in \mathbb{R} \cup \{ \mathrm{NaR} \}$: input\\
			$n \in \mathbb{N}_1$: bit count
		}
		\Output{%
			$\round_n(x) \in \{ 0,\mathrm{NaR} \} \cup 
					\pm\left(\euler^{-255},\euler^{255}\right)$
		}
		\BlankLine
		\tcc{saturate excessive numbers}
		\uIf{$x \in \left(-\infty,-\euler^{255}\right]$}{
			$T \gets \begin{cases}
			(1,1) & n = 2\\
			(1,\bm{0},1) & n \ge 3
			\end{cases} \in {\{0,1\}}^{n}$
		}
		\uElseIf{$x \in \left[-\euler^{-255},0\right)$}{
			$T \gets \bm{1} \in {\{0,1\}}^{n}$
		}
		\uElseIf{$x \in \left(0,\euler^{-255}\right]$}{
			$T \gets (\bm{0},1) \in {\{0,1\}}^{n}$
		}
		\uElseIf{$x \in \left[\euler^{255}, \infty\right)$}{
			$T \gets (0,\bm{1}) \in {\{0,1\}}^{n}$
		}
		\columnbreak
		\Else{
			$T \gets \truncate_{n+1}(\takuminv(x))$\;
			\tcc{round}
			\eIf{$T_0 = 0$}{
				$T \gets \truncate_{n}(T)$\;
			}{
				$T \gets \truncate_{n}(T) + 1$\;
			}
			\tcc{saturate over-/underflows}
			\uIf{$T = \bm{0} \land x \neq 0$}{
				$T \gets T + \sign(x)$\;
			}
			\ElseIf{$T = (1,\bm{0}) \land x \neq \mathrm{NaR}$}{
				$T \gets T - \sign(x)$\;
			}
		}
		$\round_n(x) \gets \takum(T)$\;
	\end{multicols}
\end{algorithm}
\par
The significantly expanded dynamic range offered by takums confers a notable 
advantage in that saturation cases occur much less frequently compared to 
posits. Refer to Algorithm~\ref{alg:rounding} for the rounding procedure, which 
initially clamps values outside the dynamic range. For values falling within 
the dynamic range, the procedure first converts them to $n+1$ bit truncated 
takums and then performs rounding based on the least significant bit (LSB). 
Under- and overflows are subsequently corrected to $0$ and $\mathrm{NaR}$, 
respectively, as the final step.
\par
As evident, rounding takes place within the logarithmic domain, where
the logarithmic value $\ell$ is rounded -- a conventional practice in
logarithmic number systems. The rounding midpoint between two numbers
corresponds to their geometric mean rather than their arithmetic
mean. Notably, this method of rounding remains consistent with
roundings in the non-mantissa bits. In contrast, posits, with their
linear significand, inherently employ two types of roundings: geometric
mean rounding in the non-fraction bits and arithmetic mean rounding in the
fraction bits.
\par
While this impedes the formal analysis of posits in low-precision
applications, it presents an opportunity for future research to
develop a comprehensive theory of takum rounding encompassing both
mantissa and non-mantissa bits, extending beyond the scope of
Proposition~\ref{prop:takum-precision}.
Such an endeavour holds considerable significance for low-precision
applications, particularly those where representations predominantly
feature zero mantissa bits and rounding predominantly affects the
non-mantissa bits.
\subsection{Logarithmic Significand}\label{subsec:lns}
Besides the encoding scheme, takums also diverge from posits due 
to their logarithmic significand. This choice stems from promising outcomes
observed in the application of logarithmic significands to posits and the promising
qualities of logarithmic number systems in general \cite{2016-cnn-lns,2024-log-posit}.
Although the result provided in 
Definition~\ref{def:takum} appears straightforward, it conceals the underlying 
derivation. This section aims to elucidate this derivation and expound upon the 
advantages of a logarithmic number system. As formalised in 
\cite[(1)]{lindstrom-2018} and extended here for base $\euler$, any real number
$x \in \mathbb{R} \setminus \{ 0 \}$ can be represented as a floating-point number
\begin{equation}
	(-1)^s \times \euler^h \times \sigma(f),
\end{equation}
where $s \in \{0,1\}$ denotes the sign, $h \in \mathbb{Z}$ signifies 
the exponent, $f \in [0,1)$ denotes the fraction, and the mapping $\sigma \colon [0,1) 
\mapsto [1,\euler)$ represents the significand. The linear significand $\sigma(f) = 1+ (\euler - 1) f$ 
is the conventional choice (for base $2$ it would be the function $f \mapsto 1+f$).
While this representation may seem unconventional, it is imperative to
recognise that the base $2$ for floating-point arithmetic is, in theory, not immutable.
We possess the liberty to conceptualise alternative representations, temporarily
setting aside considerations pertaining to implementation efficiency.
\par
As an alternative to the linear significand, a base-$\euler$ logarithmic number system utilises the logarithmic significand $\sigma(f) = \euler^f$, facilitating the representation of $x$ as $(-1)^s \times \euler^{h+f} =: (-1)^s \times \euler^\ell$, where $\ell \in \mathbb{R}$ signifies the logarithmic value of $x$. There exists considerable confusion surrounding the terminology -- \enquote{exponent}, \enquote{fraction}, \enquote{characteristic}, and \enquote{mantissa} -- within the context of floating-point numbers. The terms \enquote{characteristic} and \enquote{mantissa} originate from logarithm tables, delineating the integral and fractional components of a logarithm, respectively. In 1946, \textsc{Burks} et al. (as published in \cite{1982-burks-mantissa}) employed this nomenclature to denote the exponent and fraction of a floating-point number. Nevertheless, given that floating-point numbers do not conform strictly to logarithmic principles, this terminology is considered inaccurate \cite{2002-kahan-names}, a stance reflected in the absence of these terms in the current IEEE 754 standard \cite{ieee754-2019}. For the logarithmic significand, it is appropriate to revert the terminology, naming the exponent as \enquote{characteristic} $c := h$ and the fraction as \enquote{mantissa} $m := f$, a convention also closely adopted in the definition of takums. A comparative illustration of both significands is presented in Figure~\ref{fig:lns-plot}.
\par
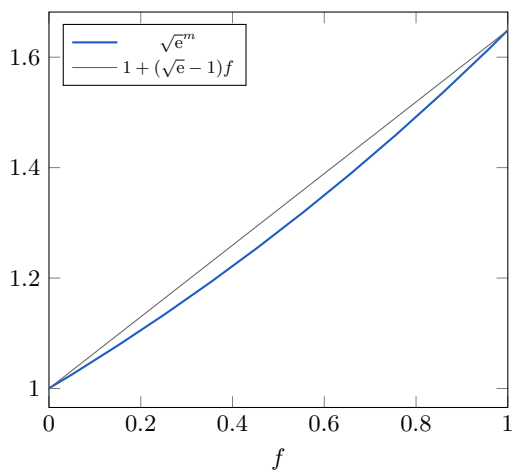
\begin{figure}[htbp]
	\begin{center}
		\begin{tikzpicture}
			\begin{axis}[
				xmin=0,
				xmax=1,
				xlabel=$f$,
				scale only axis,
				width=\textwidth/2,
				legend style={nodes={scale=0.75, transform shape}},
				legend style={at={(0.03,0.97)},anchor=north west}
			]
				\addplot[characteristic,samples=100,thick] {sqrt(exp(1))^x};
				\addlegendentry{$\euler^m$}
				\addplot[mantissa,samples=100] {1+ (sqrt(exp(1))-1) * x};
				\addlegendentry{$1+(\euler - 1)f$}
			\end{axis}
		\end{tikzpicture}
	\end{center}
	\caption{A comparison of the linear and logarithmic significands.}
	\label{fig:lns-plot}
\end{figure}
Arithmetic operations such as multiplication, division, inversion, square root, 
and squaring on numbers of this form become remarkably straightforward. Given 
$x,\tilde{x} \in \mathbb{R} \setminus \{ 0 \}$, where $x={(-1)}^s \euler^{\ell}$ and 
$\tilde{x}={(-1)}^{\tilde{s}} \euler^{\tilde{\ell}}$, with $s,\tilde{s} \in \{0,1\}$ 
and $\ell,\tilde{\ell} \in \mathbb{R}$, it is observed that:
\begin{align}
	x \cdot \tilde{x} &= {(-1)}^s \euler^{\ell} \cdot {(-1)}^{\tilde{s}}
		\euler^{\tilde{\ell}} = {(-1)}^{s+\tilde{s}}
		\euler^{\ell+\tilde{\ell}},\\
	x \div \tilde{x} &= {(-1)}^s \euler^{\ell} \div {(-1)}^{\tilde{s}}
		\euler^{\tilde{\ell}} = {(-1)}^{s-\tilde{s}}
		\euler^{\ell-\tilde{\ell}},\\
	x^{-1} &= {\left( {(-1)}^s \euler^{\ell} \right)}^{-1} =
		{(-1)}^s \euler^{-\ell},\\
	\sqrt{|x|} &= \sqrt{\euler^\ell} = \euler^{\frac{\ell}{2}},\\
	x^2 &= {\left( {(-1)}^s \euler^{\ell} \right)}^2 = \euler^{2\ell}.
\end{align}
If $\ell$ and $\tilde{\ell}$ are stored as fixed-point numbers, these operations 
reduce to fixed-point additions, subtractions, negations, and bit-shifts, all 
of which are highly efficient. The primary challenge lies in additions and 
subtractions: Without loss of generality, assuming  $s = \tilde{s} = 0$, $\ell > 
\tilde{\ell}$ and $q := \ell - \tilde{\ell} > 0$
\begin{align}
	\log_{\euler}\!\left(x + \tilde{x}\right) &= \log_{\euler}\!\left(\euler^\ell +
		\euler^{\tilde{\ell}}\right)\\
		&= \log_{\euler}\!\left(\euler^\ell \left(1 + \euler^{\tilde{\ell} - \ell}\right)\right)\\
		&=\log_{\euler}\!\left(\euler^\ell \left(1 + \euler^{-q}\right)\right)\\
		&= \ell + \log_{\euler}\!\left(1 + \euler^{-q}\right)
\end{align}
and
\begin{align}
	\log_{\euler}\!\left(x - \tilde{x}\right) &= \log_{\euler}\!\left(\euler^\ell -
		\euler^{\tilde{\ell}}\right)\\
	&= \log_{\euler}\!\left(\euler^\ell \left(1 - \euler^{\tilde{\ell} - \ell}\right)\right)\\
	&= \log_{\euler}\!\left(\euler^\ell \left(1 - \euler^{-q}\right)\right)\\
	&= \ell + \log_{\euler}\!\left(1 - \euler^{-q}\right).
\end{align}
In essence, to determine the exponent resulting from the addition or 
subtraction of $x$ and $\tilde{x}$, one adds $\log_{\euler}\!\left(1 + \euler^{-q}\right)$ 
or $\log_{\euler}\!\left(1 - \euler^{-q}\right)$ to $x$ respectively. 
These logarithmic calculations are commonly referred to as \textsc{Gauss}ian
logarithms and can be defined as
\begin{align}\label{eq:gaussian_logarithms}
	\Phi_b^+(q) &:= \log_b(1 + b^{-q}),\\
	\Phi_b^-(q) &:= \log_b(1 - b^{-q})
\end{align}
for $q > 0$ and a general base $b > 1$.
Efficient hardware implementations, with up to 32 bits precision, have been 
demonstrated using lookup tables and interpolation \cite{2000-lns-elm, 2016-lns-compete-fp, 2020-lns-cotransform}.
These implementations offer comparable, if not superior, latency to addition and subtraction using linear 
significands, alongside significantly enhanced overall arithmetic performance and reduced 
power consumption.
For an exploration of the handling of even greater levels of precision, where tables
don't scale well anymore, the reader is directed to Section~\ref{subsec:choice_of_base}.
This section further elucidates why this is also the rationale behind our selection of the base $\euler$.
\par
A significant advantage of logarithmic number systems overall lies in the singular
focus required within FPU design, 
as opposed to linear significands which necessitate dedicated logic design and 
optimisation for operations like multiplication, division, inversion, square 
roots, and powers. Posits employing logarithmic significands have previously been 
explored and found to be well-suited for neural networks \cite{johnson-2018}. 
Notably, posits and takums with logarithmic significands offer the advantageous 
feature of perfect invertibility for every number, aligning closely with the 
original concept of unums \cite{2015-the_end_of_error}. Furthermore, this 
characteristic simplifies hardware implementations as inversion can be 
accomplished through a bitwise operation (see 
Proposition~\ref{prop:takum-inversion}).
\par
All of these considerations prompted the default definition of takums with a 
logarithmic significand and the format is so tailored for being a logarithmic
number system that defining it with a linear significand makes little sense,
mostly because of the base $\euler$ that cannot properly translate
into bit-shifts for arithmetic. If we do it anyway,
(\ref{eq:takum-logarithmic}) and (\ref{eq:takum-value}) would 
change to
\begin{equation}
	\overline{\ell} = {(-1)}^{\textcolor{sign}{S}} (c + \textcolor{sign}{S})
	\label{eq:takum-exponent-linear}
\end{equation}
and
\begin{multline}
	\overline{t}((\textcolor{sign}{S},\textcolor{direction}{D},\textcolor{regime}{R},
	\textcolor{characteristic}{C},\textcolor{mantissa}{M}))
	:=\\\begin{cases}
		\begin{cases}
			0 & \textcolor{sign}{S} = 0\\
			\mathrm{NaR} & \textcolor{sign}{S} = 1
		\end{cases}
			& \textcolor{direction}{D} = \textcolor{regime}{R} = \textcolor{characteristic}{C} = \textcolor{mantissa}{M}
			= \bm{0} \\
		[(1 - (1 + \euler) \textcolor{sign}{S}) + (\euler - 1) f] \cdot
		\euler^{\overline{\ell}} & \text{otherwise.}
	\end{cases}
	\label{eq:takum-value-linear}
\end{multline}
We observe that the fraction $f \in [0,1)$ is linearly mapped to a 
value in the interval $[1,\euler)$ when $\textcolor{sign}{S}=0$, and to the interval 
$[-\euler,-1)$ when $\textcolor{sign}{S}=1$, through the mapping
\begin{equation}
	f \mapsto \begin{cases}
		1+ (\euler - 1) f & \textcolor{sign}{S}=0\\
		-\euler+ (\euler - 1)f & \textcolor{sign}{S}=1
	\end{cases} = (1 - (1 + \euler) \textcolor{sign}{S}) + (\euler - 1) f.
\end{equation}
We can insert the logarithmic significand
\begin{equation}
	m \mapsto \begin{cases}
		\euler^m & \textcolor{sign}{S}=0\\
		-\euler^{1-m} & \textcolor{sign}{S}=1
	\end{cases} = {(-\euler)}^{\textcolor{sign}{S}}
		\euler^{{(-1)}^{\textcolor{sign}{S}} m}
\end{equation}
and (\ref{eq:takum-exponent-linear}) into
(\ref{eq:takum-value-linear}) to derive (\ref{eq:takum-value}) as follows
\begin{align}
	{(-\euler)}^{\textcolor{sign}{S}} \euler^{{(-1)}^{\textcolor{sign}{S}} m}
		\euler^{\overline{\ell}} &=
		{(-\euler)}^{\textcolor{sign}{S}} \euler^{{(-1)}^{\textcolor{sign}{S}} m}
		\euler^{{(-1)}^{\textcolor{sign}{S}} (c + \textcolor{sign}{S})}\\
	&= {(-\euler)}^{\textcolor{sign}{S}} \euler^{{(-1)}^{\textcolor{sign}{S}}
		(c + m + \textcolor{sign}{S})}\\
	&= {(-1)}^{\textcolor{sign}{S}} \euler^{\textcolor{sign}{S}} 
		\euler^{{(-1)}^{\textcolor{sign}{S}} (c + m + \textcolor{sign}{S})}\\
	&= {(-1)}^{\textcolor{sign}{S}}
		\euler^{{(-1)}^{\textcolor{sign}{S}} (c + m + \textcolor{sign}{S}) +
		\textcolor{sign}{S}}\\
	&= {(-1)}^{\textcolor{sign}{S}}
		\euler^{{(-1)}^{\textcolor{sign}{S}} (c + m)}\\
	&= (-1)^{\textcolor{sign}{S}} \euler^\ell,
\end{align}
with $\ell$ as in (\ref{eq:takum-logarithmic}).
As evident from the $\textcolor{sign}{S}$-addition present in the exponent 
formula (\ref{eq:takum-exponent-linear}), which is analogous to the original 
posit exponent (refer to (\ref{eq:posit-exponent}) in 
Definition~\ref{def:posit}), a notable cancellation occurs, resulting in a 
remarkably straightforward expression for the takum value. This observation 
also lends support to the adoption of a logarithmic significand, as it not only 
simplifies the mathematical formulation but also facilitates subsequent formal 
analyses, as demonstrated in Section~\ref{subsec:formal_analysis}.
\subsection{Choice of Base $\euler$}\label{subsec:choice_of_base}
Another crucial aspect warranting discussion is the selection of
$\euler$ as the base of the takum logarithmic number system, namely in the
expression ${(-1)}^{\textcolor{sign}{S}} \euler^\ell$ in 
(\ref{eq:takum-value}), where much more common alternatives such as $\sqrt{2}$,
$2$ or $\sqrt{2^3}$ could have been chosen.
\par
While floating-point numbers with linear significands necessitate a base of $2$ to 
facilitate significand shifts for arithmetic, no such restriction applies to 
logarithmic number systems, as there exists no implicit dependency on a 
base-$2$ exponent for arithmetic operations. The sole immediate consequence of 
moving from base $2$ to base $\euler$ is a reduction of the dynamic range by a 
factor of $\log_2(\euler) \approx 0.72$. It prompts scrutiny into how deeply 
ingrained the choice of base $2$ is in logarithmic number systems, 
primarily due to the prevailing \enquote{binarity} of computer architectures. 
In an alternate reality, where computers were constructed using ternary or 
alternative logic, discussions might revolve around the significance of base 
$3$ or other bases for logarithmic number systems.
\par
There has been recent pioneering work on exploring alternative bases for
logarithmic number systems, particularly for 
tailored low precision applications, in \cite{2021-lns-beyond_base_2}.
However, it was only focused on tailoring dynamic range and expressive power,
not arithmetic advantages. Takum represents the inaugural implementation of a
pure base-$\euler$ logarithmic number system with an explicit focus
on improved arithmetic properties. The base $\euler$ was chosen over
$\mathrm{e}$ because the latter would yield an excessive dynamic range
of $\pm(\mathrm{e}^{-255},\mathrm{e}^{255}) \approx \pm(\num{1.8e-111},\num{5.6e110})$.
Any arithmetic advantages provided by base $\mathrm{e}$ are however still
accessible as a change from $\euler$ to $\mathrm{e}$ and back merely
constitutes a left and right shift of the fixed-point logarithmic value $\ell$
respectively.
\par
Regardless of the underlying logic employed in a computing system, which is 
inherently of human design, the base $\mathrm{e}$ holds unparalleled significance 
due to its profound integration within mathematics and the natural sciences. 
Particularly, the pivotal role of the exponential function in mathematics, its 
utility in modelling growth, decay, and dynamic systems, alongside its 
widespread application across diverse fields including electromagnetism, 
quantum mechanics, thermodynamics, optics, acoustics, biological systems, and 
environmental science, among others, underscores the unique role of the 
base $\mathrm{e}$ in comparison to other bases.
While $\mathrm{e}$ is an irrational number, it is worth noting that adopting a 
rational base such as $2$ with non-integral exponents mostly yields irrational 
numbers as well. The principal advantage of employing a base of $2$ lies in its 
capacity to accurately represent powers of two. Nonetheless, one must consider 
whether prioritising exact integral representations should be a fundamental 
design objective for a non-integral number system, or if it entails trade-offs.
One might argue that an integer is effectively represented within a specific number system if a round-trip conversion—comprising conversion to this number system and back, followed by rounding to its original number of significant digits—results in the retrieval of the original integer.
\par
The intricacies inherent to the \textsc{Gaussian} logarithm remain ostensibly unaffected by alterations in the base when utilising lookup tables with interpolation, as such modifications merely involve a rescaling of the stored constants. However, it's noted that lookup tables with interpolation do not scale well to higher precisions beyond $32$ bits \cite{2016-lns-compete-fp}.
The base $\mathrm{e}$ has been explored to pave the way for novel avenues in arithmetic \cite{2020-lns-elma}, employing a mixed-base methodology (with a base $2$ exponent and base $\mathrm{e}$ significand) alongside efficient $\ln()$/$\exp()$ evaluation algorithms sourced from \cite{1985-exp-ln-evaluation}. These approaches are inapplicable to base $2$ as the change of basis entails the multiplication or division of the result by $\ln(2)$, which is not power-efficient. In contrast, transitioning from the base $\mathrm{e}$ to the base $\euler$ can be achieved with a simple bit shift of the stored logarithmic value, as $\ln\!\left(\euler\right) = 0.5$.
Given the constrained dynamic range characteristic of takums, adaptations are feasible for the entire logarithmic value $\ell$, rather than confining oneself solely to a base $\euler$ significand. The adaptation of the algorithms presented in \cite{1985-exp-ln-evaluation} will be addressed in future works.
\par
Another primary advantage of utilising a base-$\euler$ logarithmic number system resides in the formulation of Gaussian logarithms (see (\ref{eq:gaussian_logarithms})). Assuming that evaluating $\Phi_{\mathrm{e}}^{\pm}$ is efficient, we observe that the expression
\begin{equation}
    \Phi_{\euler}^{\pm} (q) = \log_{\euler}\!\left(1 \pm \euler^{-q}\right) = 2 \Phi_{\mathrm{e}}^{\pm} \! \left(\frac{q}{2}\right)
\end{equation}
involves minimal overhead, requiring just two bit shifts to transition between $\Phi_{\mathrm{e}}^{\pm}$ and $\Phi_{\euler}^{\pm}$. Conversely, in the conventional form with base $2$, denoted as $\Phi_2^{\pm}$, we have
\begin{equation}
    \Phi_2^{\pm} (q) = \log_2\!\left(1 \pm 2^{-q}\right) = \frac{\Phi_{\mathrm{e}}^\pm(q \ln(2))}{\ln(2)},
\end{equation}
where rescaling with $\log(2)$ proves to be costly. This incurs overhead in the evaluation of hyperbolic and other elementary functions, often necessitating the application of $\ln(2)$ for argument or result rescaling -- a practice commonly encountered due to the prevalent role of the natural base $\mathrm{e}$ in mathematics. Such considerations hold particular significance for evaluating activation functions in deep learning contexts \cite{2018-lns-deep_learning}, especially when combined with efficient evaluations of $\ln()$ and $\exp()$, which warrant further exploration.
\par
As a side note, it is worth highlighting the remarkable proximity between $\sqrt{e} \approx 1.65$ and the golden ratio $\frac{1+\sqrt{5}}{2} \approx 1.62$.
\subsection{Formal Analysis}\label{subsec:formal_analysis}
This work places particular emphasis on the formal verification of the proposed 
takum format, acknowledging the inherent superiority of formal analysis over manual mechanical 
verification. Formal verification not only facilitates the development of novel 
proof techniques, which were indispensable in substantiating certain 
propositions delineated within this section, but also contributes to 
maintaining rigor and precision in the analysis. In the interest of 
readability, most proofs are relegated to the appendix. The 
exposition commences with a straightforward proof, affirming the non-redundancy 
of number encodings in the takum format:
\begin{proposition}[Takum Uniqueness]\label{prop:takum-uniqueness}
	Let $n \in \mathbb{N}_1$ and
	$B,\tilde{B} \in 
	{\{0,1\}}^n$ as in Definition~\ref{def:takum}. It holds
	\begin{equation}
		\takum(B) = \takum\!\left(\tilde{B}\right) \Rightarrow
		B = \tilde{B},
	\end{equation}
	which means that $\takum$ is an injective function.
\end{proposition}
\begin{proof}
	See Section~\ref{sec:proof-takum-uniqueness}.
\end{proof}
Before proceeding to establish the next property, we first introduce some 
essential definitions:
\begin{definition}[Unsigned Integer Mapping]\label{def:unsigned_integer_mapping}
	Let $n \in \mathbb{N}_1$ and an $n$-bit
	string $B = (B_{n-1},\dots,B_0) \in {\{ 0,1 \}}^n$.
	The unsigned integer mapping $\unsignedint \colon {\{ 0,1 \}}^n \mapsto
	\left\{ 0,\dots, 2^{n}-1\right\}$
	is defined as
	\begin{equation}
		\unsignedint(B) = \sum_{i=0}^{n-1} B_i 2^i.
	\end{equation}
\end{definition}
\begin{definition}[Bit String Incrementation]
	Let $n \in \mathbb{N}_1$ and an $n$-bit
	string $B = (B_{n-1},\dots,B_0) \in {\{ 0,1 \}}^n$. The
	bit string incrementation is defined as
	\begin{equation}
		B + 1 := \begin{cases}
			\unsignedint^{\mathrm{inv}}(\unsignedint(B) + 1)
				& \unsignedint(B) \in \left\{ 0,\dots,2^{n}-2 \right\}\\
			\bm{0} & \unsignedint(B) = 2^{n}-1.
		\end{cases}
	\end{equation}
\end{definition}
Necessary for the subsequent proofs is the following elementary property of 
unsigned integers:
\begin{lemma}[Unsigned Integer Negation]\label{lem:unsigned_integer_negation}
	Let $n \in \mathbb{N}_0$ and an $n$-bit
	string $B = (B_{n-1},\dots,B_0) \in {\{ 0,1 \}}^n$. It holds
	\begin{equation}\label{eq:lemma-unsigned_integer_negation}
		\unsignedint\!\left(\overline{B}\right) = 
		2^n - 1 - \unsignedint(B).
	\end{equation}
\end{lemma}
\begin{proof}
	By inserting Definition~\ref{def:unsigned_integer_mapping} and
	adding the right sum on both sides, we can see that
	(\ref{eq:lemma-unsigned_integer_negation}) is equivalent to
	\begin{equation}
		2^n-1 = \sum_{i=0}^{n-1} \left( \overline{B_i} + B_i \right) 2^i =
		\sum_{i=0}^{n-1} 2^i = 2^n - 1,
	\end{equation}
	which proves it.\qed
\end{proof}
After the introduction of unsigned integers, we proceed to define two's 
complement signed integers:
\begin{definition}[Two's Complement Signed Integer Mapping]
	Let $n \in \mathbb{N}_1$ and an $n$-bit
	string $B = (B_{n-1},\dots,B_0) \in {\{ 0,1 \}}^n$.
	The two's complement signed integer mapping
	$\twoscomp \colon {\{ 0,1 \}}^n \mapsto
	\left\{ -2^{n-1},\dots, 2^{n-1}-1\right\}$
	is defined as
	\begin{equation}
		\twoscomp(B) = -B_{n-1} 2^{n-1} + \sum_{i=0}^{n-2} B_i 2^i.
	\end{equation}
\end{definition}
\begin{lemma}[Two's Complement Signed Integer Monotonicity]
	Let $n \in \mathbb{N}_1$ and an $n$-bit
	string $B = (B_{n-1},\dots,B_0) \in {\{ 0,1 \}}^n$ with $B \neq 
	(0,1,\dots,1)$. It holds
	\begin{equation}
		\twoscomp(B + 1) = \twoscomp(B) + 1.
	\end{equation}
\end{lemma}
\begin{proof}
	For $(B_{n-2},\dots,B_0) \neq \bm{1}$, the bit string does not overflow 
	upon incrementation. Similarly, the sum $\sum_{i=0}^{n-2} B_i 2^i$ does not 
	overflow, thus validating the equation $\twoscomp(B + 1) = \twoscomp(B) + 
	1$. The condition $B = (0,1,\dots,1)$ has been explicitly excluded, leaving 
	only the scenario where $B= (1,1,\dots,1) = \bm{1}$.
	Consequently, it holds that 
	\begin{equation}
		0 = \left(-2^{n-1} + 2^{n-1} - 1\right) + 1 = \twoscomp(\bm{1}) + 1 = 
		\twoscomp(\bm{1}+1) = \twoscomp(\bm{0}) = 0,
	\end{equation}
	as was to be shown.\qed
\end{proof}
This monotonicity yields an order on ${\{ 0,1 \}}^n$ induced by
the image of the signed integer mapping $\twoscomp$, specifically the set
$\left(\left\{ -2^{n-1},\dots, 2^{n-1}-1\right\},\le\right)$,
which forms a subset of $(\mathbb{Z},\le)$, as follows:
\begin{definition}[Two's Complement Signed Integer Partial Order]
	Let $n \in \mathbb{N}_1$ and $n$-bit
	strings $B, \tilde{B} \in {\{ 0,1 \}}^n$. The
	\emph{two's complement signed partial order} is defined as
	\begin{equation}
		B \preceq \tilde{B} \;:\longleftrightarrow\;
		\twoscomp(B) \le \twoscomp\!\left(\tilde{B}\right),
	\end{equation}
	yielding a partially ordered set $({\{ 0,1 \}}^n, \preceq)$.
\end{definition}
With this apparatus established, we can formally articulate the concept of 
monotonicity concerning takums:
\begin{proposition}[Takum Monotonicity]\label{prop:takum-monotonicity}
	Let $n \in \mathbb{N}_1$ and $n$-bit
	strings $B, \tilde{B} \in {\{ 0,1 \}}^n \setminus \{ (1,0,\dots,0) \}$.
	It holds
	\begin{equation}
		B \preceq \tilde{B} \;\longrightarrow\; \takum(B) \le 
		\takum\!\left(\tilde{B}\right).
	\end{equation}
\end{proposition}
\begin{proof}
	See Section~\ref{sec:proof-takum-monotonicity}.
\end{proof}
Having established the fundamental properties of takums, the focus now shifts 
towards demonstrating three essential bitwise operations: negation, inversion, 
and the combined operation of negation and inversion, applied to a given takum.
Let us begin by examining the process of signed integer negation:
\begin{proposition}[Two's Complement Signed Integer Negation]
	Let $n \in \mathbb{N}_1$ and an $n$-bit string $B = (B_{n-1},\dots,B_0)
	\in {\{0,1\}}^n$ with $B \neq (1,0,\dots,0)$. It holds
	\begin{equation}
		\twoscomp(\overline{B} + 1) = -\twoscomp(B).
	\end{equation}
\end{proposition}
\begin{proof}
	Given $B \neq (1,0,\dots,0)$ we know that $\overline{B}+1$
	never carries into the inverted MSB $\overline{B_{n-1}}$.
	For $B_{n-1} = 0$ it holds with Lemma~\ref{lem:unsigned_integer_negation}
	\begin{multline}
		\twoscomp(\overline{B} + 1) = -2^{n-1} + 2^{n-1} -1 -
			\sum_{i=0}^{n-2} B_i 2^i + 1 =\\
			-\left( 0 \cdot 2^{n-1} + \sum_{i=0}^{n-2} B_i 2^i \right) =
			-\twoscomp(B).
	\end{multline}
	Likewise for $B_{n-1} = 1$ we can deduce
	\begin{multline}
		\twoscomp(\overline{B} + 1) = -0 \cdot 2^{n-1} +
			2^{n-1} -1 - \sum_{i=0}^{n-2} B_i 2^i + 1 =\\
			-\left( -1 \cdot 2^{n-1} + \sum_{i=0}^{n-2} B_i 2^i \right) =
			-\twoscomp(B).
	\end{multline}
	This was to be shown.\qed
\end{proof}
\begin{proposition}[Takum Negation]\label{prop:takum-negation}
	Let $n \in \mathbb{N}_1$ and
	$(\textcolor{sign}{S},\textcolor{direction}{D},\textcolor{regime}{R},
	\textcolor{characteristic}{C},\textcolor{mantissa}{M})
	\in {\{0,1\}}^n$ as in Definition~\ref{def:takum}. It holds
	\begin{equation}
		\takum\!\left(\left(\overline{\textcolor{sign}{S}},\overline{\textcolor{direction}{D}},
				\overline{\textcolor{regime}{R}},
				\overline{\textcolor{characteristic}{C}},
				\overline{\textcolor{mantissa}{M}}\right)\!+\!1\right) = \begin{cases}
			-\takum((\textcolor{sign}{S},\textcolor{direction}{D},\textcolor{regime}{R},
						\textcolor{characteristic}{C},\textcolor{mantissa}{M})) &
				\takum((\textcolor{sign}{S},\textcolor{direction}{D},
				\textcolor{regime}{R},
				\textcolor{characteristic}{C},\textcolor{mantissa}{M})) \neq
				\mathrm{NaR}\\
			\mathrm{NaR} & \takum((\textcolor{sign}{S},\textcolor{direction}{D},
				\textcolor{regime}{R},
				\textcolor{characteristic}{C},\textcolor{mantissa}{M})) =
				\mathrm{NaR}.
		\end{cases}
	\end{equation}
\end{proposition}
\begin{proof}
	See Section~\ref{sec:proof-takum-negation}.
\end{proof}
Most notably, this bitwise operation is tantamount to negating a two's 
complement integer. Subsequently, the ensuing observation serves as a minor yet 
pivotal simplification in elucidating the proof pertaining to the third bitwise 
operation under scrutiny.
\begin{lemma}[Takum Inversion-Negation]\label{lem:takum-inversion-negation}
	Let $n \in \mathbb{N}_1$ and
	$(\textcolor{sign}{S},\textcolor{direction}{D},\textcolor{regime}{R},
	\textcolor{characteristic}{C},\textcolor{mantissa}{M})
	\in {\{0,1\}}^n$ as in Definition~\ref{def:takum} with
	$\takum((\textcolor{sign}{S},\textcolor{direction}{D}, \textcolor{regime}{R},
	\textcolor{characteristic}{C},\textcolor{mantissa}{M})) \neq \mathrm{NaR}$. It holds
	\begin{equation}
		\takum\!\left(\left(\overline{\textcolor{sign}{S}},\textcolor{direction}{D},
		\textcolor{regime}{R},
		\textcolor{characteristic}{C},
		\textcolor{mantissa}{M}\right)\right) = \begin{cases}
			-\frac{1}{\takum((\textcolor{sign}{S},\textcolor{direction}{D},\textcolor{regime}{R},
						\textcolor{characteristic}{C},\textcolor{mantissa}{M}))} &
				\takum((\textcolor{sign}{S},\textcolor{direction}{D},
				\textcolor{regime}{R},
				\textcolor{characteristic}{C},\textcolor{mantissa}{M})) \neq
				0\\
			\mathrm{NaR} & \takum((\textcolor{sign}{S},\textcolor{direction}{D},
				\textcolor{regime}{R},
				\textcolor{characteristic}{C},\textcolor{mantissa}{M})) = 0.
		\end{cases}	
	\end{equation}
\end{lemma}
\begin{proof}
	See Section~\ref{sec:proof-takum-inversion-negation}.
\end{proof}
The final bitwise operation results in the inversion of a given takum. This 
aspect is particularly noteworthy when juxtaposed with posits, as the 
capability to perform this operation is unique to takums. This distinction 
arises from the intrinsic nature of takums as a logarithmic number system, 
wherein each numerical value possesses a perfect reciprocal.
\begin{proposition}[Takum Inversion]\label{prop:takum-inversion}
	Let $n \in \mathbb{N}_1$ and
	$(\textcolor{sign}{S},\textcolor{direction}{D},\textcolor{regime}{R},
	\textcolor{characteristic}{C},\textcolor{mantissa}{M})
	\in {\{0,1\}}^n$ as in Definition~\ref{def:takum} with
	$\takum((\textcolor{sign}{S},\textcolor{direction}{D}, \textcolor{regime}{R},
	\textcolor{characteristic}{C},\textcolor{mantissa}{M})) \neq \mathrm{NaR}$. It holds
	\begin{equation}
		\takum\!\left(\left(\textcolor{sign}{S},\overline{\textcolor{direction}{D}},
		\overline{\textcolor{regime}{R}},
		\overline{\textcolor{characteristic}{C}},
		\overline{\textcolor{mantissa}{M}}\right) + 1\right) = \begin{cases}
			\frac{1}{\takum((\textcolor{sign}{S},\textcolor{direction}{D},\textcolor{regime}{R},
						\textcolor{characteristic}{C},\textcolor{mantissa}{M}))} &
				\takum((\textcolor{sign}{S},\textcolor{direction}{D},
				\textcolor{regime}{R},
				\textcolor{characteristic}{C},\textcolor{mantissa}{M})) \neq
				0\\
			\mathrm{NaR} & \takum((\textcolor{sign}{S},\textcolor{direction}{D},
				\textcolor{regime}{R},
				\textcolor{characteristic}{C},\textcolor{mantissa}{M})) = 0.
		\end{cases}
	\end{equation}
\end{proposition}
\begin{proof}
	See Section~\ref{sec:proof-takum-inversion}.
\end{proof}
For this property alone, it is prudent to embrace a logarithmic significand, thus 
achieving a symmetrical treatment of negation and inversion across the entirety 
of the numerical system.
\par
Whilst the decoding process of a takum remains straightforward, encoding a 
floating-point value as a takum warrants further deliberation. We propose the 
following proposition for the lossless encoding and decoding of a given floating-point 
number:
\begin{proposition}[Takum Floating-Point Encoding]\label{prop:takum-encoding}
	Let
	\begin{equation}
		x := {(-1)}^s\cdot (1+f) \cdot 2^h \in
		\left(-\euler^{255},-\euler^{-255}\right) \cup \{ 0 \} \cup \left(\euler^{-255},\euler^{255}\right)
	\end{equation}
	be a floating-point number with sign $s \in \{0,1\}$, fraction $f \in 
	[0,1)$ and exponent $h \in \mathbb{Z}$ with
	\begin{equation}\label{eq:takum-encoding-exponent-condition}
		h \in \left(
			\frac{-127.5-\ln(1+f)}{\ln(2)},
			\frac{127.5-\ln(1+f)}{\ln(2)}
		\right) \supset \{-184,183\}.
	\end{equation}
	Using the notation from Definition~\ref{def:takum} it holds for
	$\takum((\textcolor{sign}{S},\textcolor{direction}{D},\textcolor{regime}{R},
	\textcolor{characteristic}{C},\textcolor{mantissa}{M})) = x$:
	\begin{align}
		\textcolor{sign}{S} &= s,\\
		\ell &= 2 \left( h \ln(2) + \ln(1+f) \right) \in (-255,255),\\
		c &= \left\lfloor {(-1)}^{\textcolor{sign}{S}} \ell \right\rfloor
			\in \{ -255,\dots,254\},\\
		\textcolor{direction}{D} &= c \ge 0,\\
		r &= \begin{cases}
			{\lfloor \log_2(-c) \rfloor} & \textcolor{direction}{D} = 
			0\\
			{\lfloor \log_2(c+1) \rfloor} & \textcolor{direction}{D} = 1,
		\end{cases}\\
		\textcolor{regime}{R} &= \begin{cases}
			7-r & \textcolor{direction}{D} = 0\\
			r & \textcolor{direction}{D} = 1,
		\end{cases}\\
		\textcolor{characteristic}{C} &= \begin{cases}
			c + 2^{r+1} - 1 & \textcolor{direction}{D} = 0\\
			c - 2^r + 1 & \textcolor{direction}{D} = 1,
		\end{cases}\\
		m &= {(-1)}^{\textcolor{sign}{S}} \ell - c \in [0,1),\\
		p &= \inf_{i \in \mathbb{N}_0} \!\left(2^i m \in \mathbb{N}_0\right) \in
			\mathbb{N}_0 \cup \{ \infty \},\\
		\textcolor{mantissa}{M} &= 2^p m
			\in {\{0,1\}}^p.
	\end{align}
	We define $\takuminv \colon \{ 0,\mathrm{NaR} \} \cup
			\pm\left(\euler^{-255},\euler^{255}\right) \mapsto {\{0,1\}}^{5+r+p}$ as
	\begin{equation}
		\takuminv(x) := \begin{cases}
			(\textcolor{sign}{0},\textcolor{direction}{0},\textcolor{regime}{\bm{0}},
			\textcolor{characteristic}{\bm{0}},\textcolor{mantissa}{\bm{0}}) & x = 0\\
			(\textcolor{sign}{1},\textcolor{direction}{0},\textcolor{regime}{\bm{0}},
			\textcolor{characteristic}{\bm{0}},\textcolor{mantissa}{\bm{0}}) & x = \mathrm{NaR}\\
			(\textcolor{sign}{S},\textcolor{direction}{D},\textcolor{regime}{R},
			\textcolor{characteristic}{C},\textcolor{mantissa}{M}) & \text{otherwise}
		\end{cases}
	\end{equation}
\end{proposition}
\begin{proof}
	See Section~\ref{sec:proof-takum-encoding}.
\end{proof}
Algorithms~\ref{alg:encoding} and \ref{alg:decoding} delineate the encoding and decoding process for a non-zero 
floating-point number, adhering to the insights expounded in 
Proposition~\ref{prop:takum-encoding}. It serves to illustrate that this 
procedure exhibits no greater complexity than that encountered in the encoding 
and decoding of (logarithmic) posits, which also require at least one evaluation of a logarithm and exponent respectively. Furthermore, there is no difference in computational complexity
between using base $2$ or $\euler$.
\begin{algorithm}[htb]
	\caption{Takum encoding algorithm for a floating-point number of the form 
		${(-1)}^s\cdot (1+f) \cdot 2^h \in
		\pm\left(\euler^{-255},\euler^{255}\right)$ based on
		Proposition~\ref{prop:takum-encoding}.}
	\label{alg:encoding}
	\begin{multicols}{2}
		\DontPrintSemicolon
		\SetKwInOut{Input}{input}
		\SetKwInOut{Output}{output}
		\Input{%
			$s \in \{0,1\}$: sign\\
			$f \in [0,1)$: fraction\\
			$h \in \mathbb{Z}$ as in
				(\ref{eq:takum-encoding-exponent-condition}): exponent\\
		}
		\Output{%
			$\textcolor{sign}{S}$ : sign bit\\
			$\textcolor{direction}{D}$ : direction bit\\
			$\textcolor{regime}{R}$ : regime bits\\
			$r$ : regime value\\
			$\textcolor{characteristic}{C}$ : characteristic bits\\
			$p \in \mathbb{N}_0 \cup \{ \infty \}$ :\\\quad mantissa bit count\\
			$M := (M_{p-1},\dots,M_0)$ :\\\quad mantissa bits
		}
		\BlankLine
		$\textcolor{sign}{S} \gets s$\;
		$\ell \gets 2 \left( h \ln(2) + \ln(1+f) \right)$\;
		\columnbreak
		$c \gets \left\lfloor
			{(-1)}^{\textcolor{sign}{S}} \ell
		\right\rfloor$\;
		\eIf{$c \ge 0$}{
			$\textcolor{direction}{D} \gets 1$\;
			$r \gets \lfloor \log_2(c+1) \rfloor$\;
			$\textcolor{regime}{R} \leftarrow r$\;
			$\textcolor{characteristic}{C} \gets a - 2^r + 1$\;
		}{
			$\textcolor{direction}{D} \gets 0$\;
			$r \gets \lfloor \log_2(-c) \rfloor$\;
			$\textcolor{regime}{R} \leftarrow 7-r$\;
			$\textcolor{characteristic}{C} \gets a + 2^{r+1} - 1$\;
		}
		$m \gets {(-1)}^{\textcolor{sign}{S}} \ell - c$\;
		$p \gets \inf_{i \in \mathbb{N}_0} \!\left(2^i m \in \mathbb{N}_0\right)$\;
		$\textcolor{mantissa}{M} \gets 2^p m \in {\{ 0,1 \}}^p$\;
	\end{multicols}
\end{algorithm}
\begin{algorithm}[htbp]
	\caption{Floating-point encoding algorithm of a takum of the
	form ${(-1)}^{\textcolor{sign}{S}} \euler^{\ell} \in \pm \left(\euler^{-255},\euler^{255}\right)$
	to ${(-1)}^s\cdot (1+f) \cdot 2^h \in \pm\left(\euler^{-255},\euler^{255}\right)$.}
	\label{alg:decoding}
	\begin{multicols}{2}
		\DontPrintSemicolon
		\SetKwInOut{Input}{input}
		\SetKwInOut{Output}{output}
		\Input{%
			$\textcolor{sign}{S} \in \{0,1\}$: sign\\
			$\ell \in (-255,255)$:\\\quad logarithmic value
		}
		\Output{%
			$s \in \{0,1\}$: sign\\
			$f \in [0,1)$: fraction\\
			$h \in \mathbb{Z}$: exponent
		}
		\BlankLine
		\columnbreak
		$s \gets \textcolor{sign}{S}$\;
		$t \gets \frac{1}{2} \log_2\!(\mathrm{e}) \ell$\;
		$h \gets \lfloor t \rfloor$\;
		$f \gets 2^{t - h} - 1$\;
	\end{multicols}
\end{algorithm}
\par
Another crucial aspect for formal analysis is error analysis. A key quantity in 
numerical analysis is the constant upper bound,
\begin{equation}\label{eq:ieee-epsilon}
	\left| \frac{x - \fl(x)}{x} \right| \le 2^{-n_f-1}
	:= \varepsilon(n_f),
\end{equation}
of the relative approximation error of a number $x \in \mathbb{R}$ 
within the normal range of an IEEE 754 floating-point format with $n_f$ 
fraction bits (refer to Table~\ref{tab:ieee-parameters}). Here, $\fl$ denotes 
the rounding operation. A similar bound applies to posits with an upper limit 
of $\varepsilon(p)$, where $p \in \{ 0,\dots,n-5 \}$ (as seen in 
(\ref{eq:posit-m})) represents the number of fraction bits, which varies due to 
the tapered exponent. Since $p$ is unbounded from below, the relative 
approximation error can potentially be up to $50\%$ irrespective
of $n$. This characteristic poses 
challenges in applying standard numerical analysis techniques to posits in 
general. Furthermore, there lacks a theoretical framework to comprehend the 
comparatively higher precision of tapered floating-point formats exhibited by 
numbers near unity, confining posits to empirical performance enhancements 
without formal guarantees. In contrast, takums enforce a lower limit on the 
number of available mantissa bits, possibly enabling the application of 
standard numerical analysis techniques until a theoretical framework for 
tapered floating-point arithmetic is developed. This will be investigated as 
follows.
\par
For $n \geq 12$ and a given $X \in \{0,1\}^n$ with $x = \takum(X) \notin 
\{ 0, \mathrm{NaR} \}$, it is possible to determine the precise count of 
available mantissa bits $p$ (refer to (\ref{eq:takum-mantissa_bit_count})) 
solely based on the represented takum value:
\begin{proposition}[Takum Mantissa Bit Count]\label{prop:takum-mantissa_bit_count}
	Let $n \in \mathbb{N}$ with $n \ge 12$ and $X \in {\{0,1\}}^n$ with
	$x := \takum(X) \notin \{ 0, \mathrm{NaR} \}$ and mantissa bit
	count $p$ as in Definition~\ref{def:takum}. It holds
	\begin{multline}
		p = n - 5 -
		\left\lfloor \log_2\!\left(\left|
			\lfloor 2\ln(|x|) \sign(x) \rfloor +
			(\ln(|x|) \sign(x) \ge 0)
		\right|\right) \right\rfloor
		\in\\ \{ n-12,\dots,n-5 \}.
	\end{multline}
\end{proposition}
\begin{proof}
	See Section~\ref{sec:proof-takum-mantissa_bit_count}.
\end{proof}
While this outcome may initially appear to offer limited utility, it can be 
extended to encompass all numbers within the dynamic range, including those 
that cannot be directly represented. Such an extension permits the 
determination of the guaranteed number of mantissa bits for any given number 
prior to the commencement of rounding operations.
\begin{proposition}[Takum Mantissa Bit Count Lower Bound]
	\label{prop:takum-mantissa_bit_count-lower_bound}
	Let\\$x \in \pm\left(\euler^{-255},\euler^{255}\right)$
	and $n \in \mathbb{N}$ with $n \ge 12$. It holds
	for $X \in {\{ 0,1 \}}^n$ with $\round_n(x) = \takum(X)$
	and mantissa bit count $p$ as in Definition~\ref{def:takum}
	\begin{multline}
		p \ge n - 6 -
		\left\lfloor \log_2\!\left(\left|
			\lfloor 2\ln(|x|) \sign(x) \rfloor +
			(\ln(|x|) \sign(x) \ge 0)
		\right|\right) \right\rfloor
		\in\\ \{ n-13,\dots,n-6 \}.
	\end{multline}
\end{proposition}
\begin{proof}
	See Section~\ref{sec:proof-takum-mantissa_bit_count-lower_bound}.
\end{proof}
This finding holds significant utility as it furnishes us with the capability 
to gauge the precision of any given number when expressed as a takum even prior 
to rounding. Such an outcome assumes particular importance, unlike in the realm 
of uniform precision arithmetic such as IEEE 754 floating-point numbers, where 
the relative precision remains consistent across all normal numbers.
Upon scrutinising the proof, it becomes evident that alterations in the 
mantissa bit count occur only exceptionally rarely during rounding processes. 
It may be feasible in subsequent research endeavours to establish an even more 
robust lower bound for the mantissa bit count.
\par
Having established both the mantissa bit count and its lower bound for 
arbitrary numbers, we are now poised to delve into an examination of the 
relative approximation error inherent in takum arithmetic. Specifically, we 
shall derive an upper bound for the relative approximation error contingent 
upon the mantissa bit count $p$:
\begin{proposition}[Takum Machine Precision]\label{prop:takum-precision}
	Let $x \in \pm\left(\euler^{-255},\euler^{255}\right)$,
	$n \ge 12$ and $X \in {\{0,1\}}^n$ with
	$\round_n(x) = \takum(X)$. The bit string $X$ has the mantissa bit count
	$p \in \{ n-12,\dots,n-5 \}$ as in
	Definition~\ref{def:takum}. It holds for the relative approximation error
	\begin{equation}
		\left| \frac{x - \round_n(x)}{x} \right| \le \euler^{2^{-p-1}} -1 =:
		\lambda(p).
	\end{equation}
	The upper bound $\lambda(p)$ satisfies
	\begin{equation}
		\lambda(p) < \frac{2}{3} \varepsilon(p) < \varepsilon(p).
	\end{equation}
\end{proposition}
\begin{proof}
	See Section~\ref{sec:proof-takum-precision}.
\end{proof}
Because takums constitute a logarithmic number system, the upper bound assumes 
a distinct format compared to IEEE 754 floating-point numbers and posits. 
However, the most noteworthy disparity, particularly when contrasted with 
posits, lies in the fact that the mantissa bit count, denoted as $p$, is 
constrained to a minimum of $n-12$. Consequently, the relative approximation 
error of takums is bounded above by $\lambda(n-12) < \frac{2}{3} \varepsilon(n-12)$ for $n \geq 12$, a value 
readily applicable in standard numerical analysis theory. Despite its modest 
quality, this upper bound establishes a nexus with the extensive body of 
literature surrounding IEEE 754 floating-point numbers.
Moreover, in the context of \texttt{float64}, which boasts $n_f = 52$ fraction 
bits, \texttt{takum64} assures the same minimal mantissa bit count of $n - 12 = 
52$. Given $\lambda(52) < \frac{2}{3} \varepsilon(52)$ it can 
be deduced that \texttt{takum64} presents only at most two-thirds of the relative approximation error of \texttt{float64}, within the dynamic range of 
\texttt{takum64}. This observation in turn means that \texttt{takum64} can be 
presumed to possess at least the same (uniform) machine precision to 
\texttt{float64}, thereby enabling the direct application of all results
pertinent to double-precision IEEE 754 floating-point numbers that depend
upon machine precision.
\par
Future research endeavours shall explore novel methodologies for analysing the 
tapered precision inherent in posits and takums, aiming to comprehensively 
elucidate and formalise the advantages associated with tapered precision 
numerical formats. This nascent area of inquiry may be designated as
\enquote{tapered precision numerical analysis}, diverging from the prevailing 
paradigm of \enquote{uniform precision numerical analysis}. Nonetheless, until 
such investigations yield substantive results, takums offer a distinct 
advantage over posits by enabling the application of theoretical frameworks 
predicated on the assumption of a constant relative approximation error.
\subsection{$\mathrm{NaR}$ Convention}\label{subsec:nar-convention}
Conveniently omitted in Proposition~\ref{prop:takum-monotonicity} is the role 
of $\mathrm{NaR}$ in the ordering of takums. This is due to the discretionary 
nature of $\mathrm{NaR}$ handling. The IEEE 754 standard encompasses various 
forms of $\mathrm{NaN}$ and ultimately elected to stipulate $\mathrm{NaN} \neq 
\mathrm{NaN}$ universally, primarily because its initial specification did not 
mandate implementations to furnish a mechanism for discerning if a given 
floating-point number is a $\mathrm{NaN}$. Instead, users were offered a 
recourse through the distinctive property $x \neq x \rightarrow x = 
\mathrm{NaN}$ to identify $\mathrm{NaN}$s \cite[8]{1997-kahan-ieee}. Analogously,
all comparisons involving $\mathrm{NaN}$s yield \texttt{false}. Blindly espousing
this convention for takums sans introspection would be imprudent.
\par
In the 2019 revision of the IEEE 754 standard, a total-ordering predicate was 
introduced (see \cite[§5.10]{ieee754-2019}), which subtly alters the treatment 
of $\mathrm{NaN}$. Notably, in addition to other modifications irrelevant to 
takums, the equality $\mathrm{NaN} = \mathrm{NaN}$ is upheld, while 
$-\mathrm{NaN}$ is deemed smaller than the smallest representable number, and 
$\mathrm{NaN}$ is considered larger than the largest representable number. 
These adjustments suggest that the original handling of $\mathrm{NaN}$ may not 
be optimal. Consequently, we propose the following convention for managing 
$\mathrm{NaR}$s within the context of takums:
\begin{definition}[$\mathrm{NaR}$ Total-Ordering Convention]
	Let $n \in \mathbb{N}_1$. $\mathrm{NaR}$ is defined according to the
	total-ordering convention if and only if
	\begin{equation}
		\left(\mathrm{NaR} = \mathrm{NaR}\right) \land
		\left(\forall_{\mathrm{NaR} \neq x \in \takum\left( {\{0,1\}}^n \right)} \colon
			\mathrm{NaR} < x \right)
	\end{equation}
	hold.
\end{definition}
While there is no universally optimal method for defining $\mathrm{NaR}$ 
handling, establishing it in a manner that ensures the takum number system 
maintains a total order is deemed reasonable. Such an approach guarantees 
that $\mathrm{NaR} = \mathrm{NaR}$ remains valid and that $\mathrm{NaR}$ is 
deemed smaller than the smallest representable number. This alignment is in 
harmony with the convention regarding $\mathrm{NaR}$ in posits 
\cite{posits-standard-2022} and the total-ordering predicate
in IEEE 754-2019\cite[§5.10]{ieee754-2019}.
\par
This $\mathrm{NaR}$ convention proves to be judicious for hardware 
implementations, owing to the fact that the bit representation of 
$\mathrm{NaR}$, $(1,0,\dots,0)$, coincides with that of the smallest two's 
complement signed integer. Consequently, no special case for comparisons is 
necessitated, and takums can be compared akin to two's complement signed 
integers.
It is also a pragmatic decision to designate $\mathrm{NaR}$ as the smallest 
representable number for practical applications. This choice finds resonance in 
commonplace approximation loops, typified by constructs such as \texttt{while 
(residual < bound)}, where a predetermined bound triggers termination upon 
encountering $\mathrm{NaR}$ as the residual value, which is desirable.
\subsection{Linear Takums}
While takums are categorised within the ambit of logarithmic number systems, 
boasting numerous advantageous properties and the innovative application of a base $\euler$ approach 
which heralds the promise of more efficient hardware implementations for applications demanding higher precision, 
it is pertinent to acknowledge that the realm of logarithmic number systems remains comparatively nascent. 
This domain embodies a paradigm-shifting perspective that necessitates a fundamental reevaluation of the 
conventional tenets underpinning floating-point arithmetic.
Amidst this discourse, the utility of logarithmic significands in comparison to their linear counterparts remains an open question, poised at the brink of being either a revolutionary advancement or a notable misstep in the quest for optimizing arithmetic computation.
\par
Given that the logarithmic significand is not the sole distinguishing aspect of takums, and considering that the primary feature, namely the efficient characteristic encoding, likely holds benefits even in a traditional floating-point context, we will now define takums with a linear base-$2$ significand for applications that favour a floating-point representation.
As detailed in Section~\ref{subsec:lns}, the terms \enquote{characteristic} and \enquote{mantissa} are deemed inappropriate for non-logarithmic number systems. Consequently, the bit string segments and variables are renamed accordingly
to refer to \enquote{exponent} and \enquote{fraction}, respectively. However, the definition of the characteristic is retained as it facilitates the articulation of an exponent that precisely aligns with the exponent of the floating-point representation.
\begin{definition}[linear takum encoding]\label{def:linear_takum}
Let $n \in \mathbb{N}$ with $n \ge 12$. Any $n$-bit MSB$\rightarrow$LSB string 
$(\textcolor{sign}{S},\textcolor{direction}{D},\textcolor{regime}{R},
\textcolor{characteristic}{C},\textcolor{mantissa}{F}) \in {\{0,1\}}^n$ of the form
\begin{center}
	\begin{tikzpicture}
		\draw[<->] (0.0, 0.7) -- (0.4, 0.7) node[above,pos=.5] {sign};
		\draw[<->] (0.4, 0.7) -- (4.5, 0.7) node[above,pos=.5] {exponent};
		\draw[<->] (4.5, 0.7) -- (10.0, 0.7) node[above,pos=.5] {fraction};

		\draw (0,  0  ) rectangle (0.4,0.5) node[pos=.5] {\textcolor{sign}{S}};
		\draw (0.4,0  ) rectangle (0.8,0.5) node[pos=.5] {\textcolor{direction}{D}};
		\draw (0.8,0  ) rectangle (2.0,0.5) node[pos=.5] {\textcolor{regime}{R}};
		\draw (2.0,0  ) rectangle (4.5,0.5) node[pos=.5] {\textcolor{characteristic}{C}};
		\draw (4.5,0  ) rectangle (10.0,0.5) node[pos=.5] {\textcolor{mantissa}{F}};

		\draw[<->] (0.0, -0.2) -- (0.4, -0.2) node[below,pos=.5] {$1$};
		\draw[<->] (0.4, -0.2) -- (0.8, -0.2) node[below,pos=.5] {$1$};
		\draw[<->] (0.8, -0.2) -- (2.0, -0.2) node[below,pos=.5] {$3$};
		\draw[<->] (2.0, -0.2) -- (4.5, -0.2) node[below,pos=.5] {$r$};
		\draw[<->] (4.5, -0.2) -- (10.0, -0.2) node[below,pos=.5] {$p$};
	\end{tikzpicture}
\end{center}
with
\begin{align}
	\textcolor{sign}{S} &\in \{0,1\} &\colon \parbox{2.9cm}{sign bit}\\
	\textcolor{direction}{D} &\in \{0,1\} &\colon \parbox{2.9cm}{direction 
	bit}\\
	\textcolor{regime}{R} &:=
		(\textcolor{regime}{R}_2,\textcolor{regime}{R}_1,\textcolor{regime}{R}_0)
		\in {\{0,1\}}^3
		&\colon \parbox{2.9cm}{regime bits}\\
	r &:= \begin{cases}
		7 - (4 \textcolor{regime}{R}_2 + 2 \textcolor{regime}{R}_1 +
			\textcolor{regime}{R}_0) & \textcolor{direction}{D} = 0\\
		4 \textcolor{regime}{R}_2 + 2 \textcolor{regime}{R}_1 + \textcolor{regime}{R}_0	&
			\textcolor{direction}{D} = 1
		\end{cases}
		\in \{0,\!\dots\!,\!7\}
		&\colon \parbox{2.9cm}{regime}\\
	\textcolor{characteristic}{C} &:=(\textcolor{characteristic}{C}_{r-1},\dots,
		\textcolor{characteristic}{C}_0) \in {\{0,1\}}^r &\colon 
		\parbox{2.9cm}{characteristic bits}\\
	c &:= \begin{cases}
			-2^{r+1} + 1 + \sum_{i=0}^{r-1}
				\textcolor{characteristic}{C}_i 2^{i}& \textcolor{direction}{D} = 0\\
			2^r - 1 + \sum_{i=0}^{r-1}
				\textcolor{characteristic}{C}_i 2^{i}
				& \textcolor{direction}{D} = 1
		\end{cases}
		&\colon \parbox{2.9cm}{characteristic}\\
	p &:= n - r - 5 \in \{ n-12,\dots,n-5 \} &\colon
		\parbox{2.9cm}{fraction bit count}\\
	\textcolor{mantissa}{F} &:= (\textcolor{mantissa}{F}_{p-1},\dots,
		\textcolor{mantissa}{F}_0) \in {\{0,1\}}^p &\colon
		\parbox{2.9cm}{fraction bits}\\
	f &:= 2^{-p} \sum_{i=0}^{p-1} \textcolor{mantissa}{F\!}_i 2^i \in [0,1) 
	&\colon
		\parbox{2.9cm}{fraction}\\
	e &:= {(-1)}^{\textcolor{sign}{S}} (c + \textcolor{sign}{S})
		\in \{-255,\dots,254\}
		&\colon \parbox{2.9cm}{exponent}
\end{align}
encodes the linear takum value
{
	\begin{equation}
		\overline{\takum}((\textcolor{sign}{S},\textcolor{direction}{D},\textcolor{regime}{R},
		\textcolor{characteristic}{C},\textcolor{mantissa}{F}))
		:= \begin{cases}
			\begin{cases}
				0 & \textcolor{sign}{S} = 0\\
				\mathrm{NaR} & \textcolor{sign}{S} = 1
			\end{cases}
				& \textcolor{direction}{D} = \textcolor{regime}{R} = \textcolor{characteristic}{C} = \textcolor{mantissa}{F} 
				= \bm{0} \\
			[(1 - 3 \textcolor{sign}{S}) + f] \cdot 2^{e} & \text{otherwise}
		\end{cases}
	\end{equation}
	with $\overline{\takum} \colon {\{0,1\}}^n \mapsto \{ 0,\mathrm{NaR} \} \cup
		\pm\left(2^{-255},2^{255}\right)$.
	Without loss of generality, any bit string shorter than 12 bits is also
	considered in the definition by assuming the missing bits to be
	zero bits (\enquote{ghost bits}).
}
%
%
\end{definition}
Implementers are at liberty to adopt either variant for takums, albeit the logarithmic significand in Definition~\ref{def:takum} is designated as the standard. It is incumbent upon implementations to explicitly specify whether they support \enquote{linear takums} or \enquote{logarithmic takums}. In the absence of such clarification, the logarithmic variant is to be presumed.
\par
The linear takums more closely adhere to the original definition of posits outlined in Definition~\ref{def:posit} than the logarithmic takums. However, the posit definition exhibits a minor deficiency in failing to define a variable for the actual floating-point exponent. To address this, we introduce the term \enquote{characteristic} and designate variable $c$ to represent the \enquote{pre-exponent}, while defining the exponent $e$ to denote the actual floating-point exponent.
\par
While this document delineates the concept of linear takums, they will not be encompassed within the ensuing evaluation. Nonetheless, it is imperative to acknowledge that the cornerstone findings of the formal analysis -- specifically Propositions~\ref{prop:takum-uniqueness} (uniqueness), \ref{prop:takum-monotonicity} (monotonicity, thus also the $\mathrm{NaR}$ convention elaborated in Section~\ref{subsec:nar-convention}), and
\ref{prop:takum-negation} (negation) -- remain applicable regardless of the significand being linear or logarithmic. Linear takums have the relative
approximation error $\varepsilon(p)$ and we can define a linear takum
rounding function $\overline{\round}_n(x)$ as
in Algorithm~\ref{alg:rounding} by adapting the bounds to base $2$.
Algorithms~\ref{alg:linear-encoding} and \ref{alg:linear-decoding} delineate the procedures for converting a floating-point number to a linear takum and vice versa, respectively. While these algorithms are presented without formal proof, their derivation follows directly from the definitions and results from equating ${(-1)}^s(1+g)2^h$ and $[(1-3 \textcolor{sign}{S}) + f] \cdot 2^e$. This equivalence necessitates the borrowing or lending of one factor of $2$ in cases where $g$ or $f$ equals zero.
\par
\begin{algorithm}[htbp]
	\caption{Linear takum encoding algorithm of a floating-point number of the form 
		${(-1)}^s\cdot (1+g) \cdot 2^h \in \pm\left(2^{-255},2^{255}\right)$.}
	\label{alg:linear-encoding}
	\begin{multicols}{2}
		\DontPrintSemicolon
		\SetKwInOut{Input}{input}
		\SetKwInOut{Output}{output}
		\Input{%
			$s \in \{0,1\}$: sign\\
			$g \in [0,1)$: fraction\\
			$h \in \{-255,254\}$: exponent\\
			$h=-255 \rightarrow g \neq 0$
		}
		\Output{%
			$\textcolor{sign}{S}$ : sign bit\\
			$\textcolor{direction}{D}$ : direction bit\\
			$\textcolor{regime}{R}$ : regime bits\\
			$\textcolor{characteristic}{C}$ : amplitude bits\\
			$\textcolor{mantissa}{F} :=
			(\textcolor{mantissa}{F}_{m-1},\dots,\textcolor{mantissa}{F}_0)$ :\\\quad fraction bits
		}
		\BlankLine
		$\textcolor{sign}{S} \leftarrow s$\;
		\eIf{$\textcolor{sign}{S} = 0$}{
			$c \leftarrow h$\;
			$f \leftarrow g$\;
		}{
			\eIf{$g=0$}{
				$c \leftarrow -h$\;
				$f \leftarrow 0$\;
			}{
				$c \leftarrow -h-1$\;
				$f \leftarrow 1-g$\;
			}
		}
		\columnbreak
		$\textcolor{direction}{D} \leftarrow c \ge 0$\;
		\eIf{$\textcolor{direction}{D} = 0$}{
			$r \leftarrow \lfloor \log_2(-c) \rfloor$\;
			$\textcolor{regime}{R} \leftarrow 7-r$\;
			$\textcolor{characteristic}{C} \leftarrow c + 2^{r+1} - 1$\;
		}{
			$r \leftarrow \lfloor \log_2(c+1) \rfloor$\;
			$\textcolor{regime}{R} \leftarrow r$\;
			$\textcolor{characteristic}{C} \leftarrow c - 2^r + 1$\;
		}
		$p \gets \inf_{i \in \mathbb{N}_0} \!\left(2^i f \in \mathbb{N}_0\right)$\;
		$\textcolor{mantissa}{F} \leftarrow 2^p f$\;
	\end{multicols}
\end{algorithm}
\begin{algorithm}[htbp]
	\caption{Floating-point encoding algorithm of a linear takum of the
	form $[(1 - 3 \textcolor{sign}{S}) + f] \cdot 2^{e} \in \pm (2^{-255},2^{255})$
	to ${(-1)}^s\cdot (1+g) \cdot 2^h \in \pm\left(2^{-255},2^{255}\right)$.}
	\label{alg:linear-decoding}
	\begin{multicols}{2}
		\DontPrintSemicolon
		\SetKwInOut{Input}{input}
		\SetKwInOut{Output}{output}
		\Input{%
			$\textcolor{sign}{S} \in \{0,1\}$: sign\\
			$f \in [0,1)$: fraction\\
			$e \in \{-255,254\}$: exponent\\
			$e=-255 \rightarrow f \neq 0$
		}
		\Output{%
			$s \in \{0,1\}$: sign\\
			$g \in [0,1)$: fraction\\
			$h \in \{-255,254\}$: exponent\\
			$h=-255 \rightarrow g \neq 0$
		}
		\BlankLine
		\columnbreak
		$s \leftarrow \textcolor{sign}{S}$\;
		\eIf{$s = 0$}{
			$h \leftarrow e$\;
			$g \leftarrow f$\;
		}{
			\eIf{$f=0$}{
				$h \leftarrow e+1$\;
				$g \leftarrow 0$\;
			}{
				$h \leftarrow e$\;
				$g \leftarrow 1-f$\;
			}
		}
	\end{multicols}
\end{algorithm}
Upon examination, it becomes apparent that linear takums do not significantly lag behind logarithmic takums in terms of analytical outcomes. In the event that logarithmic number systems fail to gain traction, linear takums present themselves as a feasible alternative. However, there exist two distinct aspects where linear takums exhibit shortcomings: firstly, they lack a straightforward bitwise inversion mechanism, as demonstrated in Proposition~\ref{prop:takum-inversion}; secondly, their machine precision is at least two-thirds inferior to that of logarithmic takums, as indicated in Proposition~\ref{prop:takum-precision}. Notably, the former deficiency, which implies a mathematical symmetry between negation and inversion, holds particular appeal for its elegance.
\section{Evaluation}
In the ensuing analysis, we will appraise takums across various dimensions to 
assess their efficacy both as a numerical system for general-purpose arithmetic, 
juxtaposed against IEEE 754 floating-point numbers, and as a tapered 
floating-point format, juxtaposed against posits.
\subsection{\textsc{Gustafson} Criteria}
We commence by examining our adherence to the \textsc{Gustafson} criteria as 
delineated in Section~\ref{sec:gustafson_criteria}. Given our deliberate design 
of takums to conform to the dynamic range criteria, outlined in 
Section~\ref{subsec:dynamic_range_criteria}, there is no necessity to 
scrutinise them further.
\par
We fulfil \emph{Property~1 (distribution)} by establishing a rational dynamic 
range of numbers, detailed in Section~\ref{subsec:dynamic_range_criteria}, 
while the tapered format of takums aptly encompasses a greater proportion of 
numbers proximal to 1, aligning with typical computational requirements. 
\emph{Property~2 (uniqueness)} has been substantiated, as evidenced by 
Proposition~\ref{prop:takum-uniqueness}. \emph{Property~3 (generality)} is 
inherent in our construction, with no imposed restrictions on the bit string 
length $n$. By construction the satisfaction of 
\emph{Property~4 (statelessness)} is also assured. \emph{Property~5 
(exactness)}, though intricately linked with implementation specifics, finds 
support in the encouraging outcomes detailed in 
Section~\ref{sec:evaluation-closure}.
\par
\emph{Property~6 (binary monotonicity)} is formally established in 
Proposition~\ref{prop:takum-monotonicity}, mirroring the proof for 
\emph{Property~7 (binary negation)} as elucidated in 
Proposition~\ref{prop:takum-negation}. Additionally, we affirm the property of 
binary inversion via Proposition~\ref{prop:takum-inversion}. Regarding 
\emph{Property~8 (flexibility)}, we contend that our approach potentially 
exceeds posits' capabilities, as not only can bit strings of varying lengths be 
effortlessly converted to other lengths, but uniform decoding logic can be 
applied across all variants.
\emph{Property~9 ($\mathrm{NaR}$ propagation)} pertains to implementation 
intricacies, while the discussion surrounding \emph{Property~10 (implementation 
simplicity)} is reserved for Section~\ref{sec:evaluation-hardware}, albeit 
preliminary satisfaction is inferred. In sum, our adherence to the 
\textsc{Gustafson} criteria is comprehensive.
\subsection{Hardware Implementation}
\label{sec:evaluation-hardware}
Despite the ostensibly more intricate mathematical definition of takums 
compared to posits (refer to Definitions~\ref{def:takum} and \ref{def:posit}), 
the former's hardware implementation is considerably simpler. With merely 8 
regime states, a lookup-table necessitates only 3 entries per state for 
comprehensive format parsing: a bit mask for direct characteristic bits extraction 
(12 bits per entry, with the implicit 1 specified in the mask and a consistent 5-bit left shift per entry) and the offset $5+r$ denoting the 
initiation of fraction bits (4 bits per entry), if not directly computed. This culminates 
in a modest LUT size of 16 bytes. Given the characteristic bits' confinement within the 
initial 12 bits, logic application is also only necessary in this restricted 
domain. Moreover, the identical exponent parsing logic and lookup-tables can be 
universally employed across all types, facilitating encoding processes (refer 
to Algorithm~\ref{alg:encoding}), which have been demonstrated to be 
straightforward, no more intricate than posit encoding.
\par
In contrast, the exponent in posit format may span the entire posit width, 
necessitating logic attachment to all input bits for exponent parsing. 
Determining $k$ mandates a resource-intensive bit-counting methodology, thereby 
also impeding posit software implementations. If regime detection were to be 
executed through lookup-tables, the table size would grow
with increasing bit string lengths.
\par
Logarithmic fractions and their hardware complexity have been extensively 
studied \cite{2000-lns-elm,2016-lns-compete-fp,2020-lns-cotransform}. Recent 
advancements and hardware implementations demonstrate that addition and 
subtraction can be performed with comparable or even lower latency than 
floating-point operations using linear significands. For all other arithmetic 
operations, logarithmic significands exhibit significantly reduced latency and power 
consumption \cite{2016-lns-compete-fp,2020-lns-elma}. From a practical standpoint, when 
implementing an Arithmetic Processing Unit (APU) for handling logarithmic 
fractions, focus solely on addition and subtraction is necessary, eliminating the 
need for implementing and optimizing complex operations such as multiplication, 
division, square root extraction, and squaring, as these operations can be 
reduced to fixed-point additions, subtractions, and shifts. This aspect is 
further underscored by formally proven bitwise operations enabling rapid 
negation and inversion of a takum, thus obviating the necessity for dedicated 
inverter logic. For scenarios requiring higher bit counts where LUT based approaches become impractical, existing methodologies, such as those discussed in \cite{2020-lns-elma}, can be adapted. This adaptation leverages the novel foundational principle that takums are based on the basis $\euler$.
\par
Additionally, by adhering to the $\mathrm{NaR}$ convention as elucidated in 
Section~\ref{subsec:nar-convention}, the dedicated handling of $\mathrm{NaR}$ values 
becomes superfluous without compromising mathematical integrity. Instead, 
takums can be seamlessly type-cast to two's complement signed integers for 
comparison across all bit representations, including instances featuring 
$\mathrm{NaR}$.
\subsection{Coding Efficiency}\label{sec:evaluation-efficiency}
In considering coding efficiency, our investigation focuses on the bit count 
necessary to encode a given positive number, without loss of generality. In 
the case of takums, this pertains to encoding the characteristic $c$, the integral component of the logarithmic value $\ell$. A thorough comparative analysis of takums 
vis-à-vis other encoding schemes is presented in 
Figure~\ref{fig:coding_efficiency}.
The efficacy of coding is evaluated by examining the intrinsic encoding capabilities of the schemes when applied to numerals. This methodology deliberately avoids direct comparisons based on exponents, as such an approach would disproportionately benefit formats other than takums. This discrepancy arises because exponentiation to the base $\euler$ increases at a more gradual pace compared to exponentiation to the base $2$.
\begin{figure}[htbp]
	\begin{center}
		\begin{tikzpicture}
		\begin{groupplot}[
				group style={
					group size=2 by 1,
					horizontal sep=0pt,
					xlabels at=edge bottom,
					ylabels at=edge left,
				},
				use fpu=false,
				xlabel={\hspace{2cm}value},
				ylabel={coded length/bit},
				scale only axis,
				height=\textwidth/2,
				width=\textwidth-3cm,
					grid=major,
					major grid style={line width=.2pt,draw=gray!30},
		        xmin=-3, xmax=142,
		        xtick={0,20,...,260},
		        ytick={0,5,...,142},
		        domain=0:140,
		        samples=141,
		        ymin=1, ymax=22,
		        extra x ticks={16,30,254},
		        extra x tick style={
		        	ticklabel pos=upper,
		        },
		        extra y ticks={8,11},
				minor x tick style={draw=none},
				extra y tick style={
					major grid style={draw=gray!30, line width=.2pt},
					ticklabel pos=right,
				},
				extra y tick labels={},
				legend style={nodes={scale=0.75, transform shape}},
				legend style={at={(0.03,0.97)},anchor=north west},
				legend style={cells={align=left}},
		]
			\nextgroupplot[
				axis y line*=left,
				axis y line*=right,
				axis y line*=none,
			  x axis line style={-}, 
			  width=\textwidth-3.5cm, 
			]
			\addplot[characteristic,thick,restrict x to domain=0:140] table [x=exp, 
			y={takum}, col sep=comma]
				{out/cost.csv};
			\addlegendentry{takum}
			\addplot[sign,restrict x to domain=0:140] table [x=exp, y={posit(eS=2)}, col sep=comma]
				{out/cost.csv};
			\addlegendentry{posit}
			\addplot[mark=none, densely dashed, black, restrict x to domain=0:127]{8};
			\addlegendentry{\texttt{float32}/\\\texttt{bfloat16}}
			\addplot[mark=none,semithick, densely dotted, black, restrict x to domain=0:140]{11};
			\addlegendentry{\texttt{float64}}
			\addplot[direction,densely dashdotted, restrict x to domain=0:140] table [x=exp, y={elias-gamma}, col sep=comma]
				{out/cost.csv};
			\addlegendentry{\textsc{Elias}-$\gamma$}
			\addplot[regime,densely dashdotdotted,restrict x to domain=0:140] table [x=exp, y={elias-delta(0)}, col sep=comma]
				{out/cost.csv};
			\addlegendentry{\textsc{Elias}-$\delta(0)$}
			\nextgroupplot[
				axis y line*=right,
				xlabel=\empty,
				ylabel=\empty,
				xmin=230, xmax=258,
				ymin=1, ymax=22,
				domain=240:254,
				width=2cm,
		        axis x discontinuity=crunch,
		        ytick={0,5,...,100},
		        extra y tick labels={8,11},
		        yticklabel=\empty,
			]
			\addplot[characteristic,thick,restrict x to domain=240:254] table [x=exp, 
			y={takum}, col sep=comma]
				{out/cost.csv};
			\addplot[direction,densely dashdotted, restrict x to domain=240:254] table [x=exp, 
			y={elias-gamma}, col sep=comma]
				{out/cost.csv};
			\addplot[regime,densely dashdotdotted, restrict x to domain=240:254] table [x=exp, 
			y={elias-delta(0)}, col sep=comma]
				{out/cost.csv};
			\addplot[mark=none, black, semithick, densely dotted]{11};
		\end{groupplot}
		\end{tikzpicture}
	\end{center}
	\caption{Bit requirements for encoding a given 
	value with different methods.
	The \textsc{Elias} codes are both defined in \cite{lindstrom-2018}.}
	\label{fig:coding_efficiency}
\end{figure}
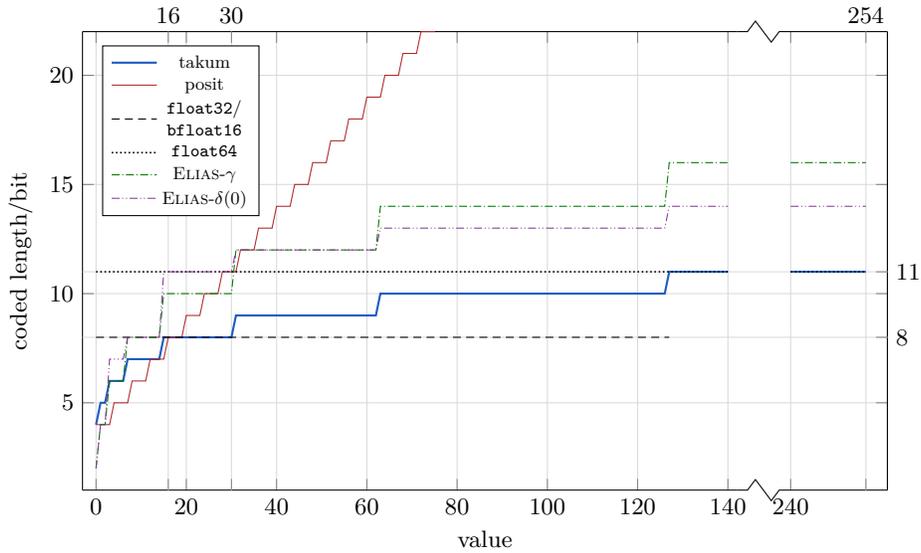
When interpreting the results, it is imperative to discern between the 
performance concerning small and large values. Depending on the application, 
the significance of one over the other varies. In the realm of small values
($0-15$), posits maintain an overall superiority, surpassing all alternative 
methods. Particularly, the \textsc{Elias} codes exhibit considerable initial 
overhead. Takums rank second, closely matching the performance of posits 
initially, albeit with a slight degradation of around $1$ bit overall, ultimately 
outperforming \textsc{Elias} codes beyond the number $2$.
\par
Conversely, in the context of higher values ($>15$), takums demonstrate 
superior coding efficiency, notably \enquote{overtaking} posits at number
$16$ 
with an additional 8 bits, comparable to the \texttt{bfloat16}/\texttt{float32} 
encoding cost. Takum encoding remains at 8 bits until number 30 (where posits 
already require 11 bits), never exceeding 11 bits (comparable to 
\texttt{float64}) until the highest value 254, where it also demonstrates to
surpass the \textsc{Elias} codes, which necessitate 14 and 16 bits respectively.
Posits, as previously noted, exhibit subpar performance for large numbers,
a contributing factor to their unsuitability for general-purpose arithmetic.
\par
The \textsc{Elias} codes,
previously identified as possible alternatives to the posit encoding 
scheme\cite{lindstrom-2018}, scale much better, but only manage to 
\enquote{overtake} posits at numbers $31$ with $12$ bits, which is 1 bit more 
than \texttt{float64}, and are significantly less efficient than the takum
encoding overall. Thus even though the \textsc{Elias}
codes exhibit better asymptotic behaviour than posits, their initial and general
prefix-code overhead is too high. Overall when compared to IEEE 754
floating-point numbers, it is evident that takums are at least as efficient as
\texttt{float32}/\texttt{bfloat16} for numbers up to $30$ and outperform or at least match \texttt{float64} up to the maximum value $254$. It shall be
noted here that \texttt{float32}/\texttt{bfloat16}, based on the previous
discussion in this paper, do not offer a suitable dynamic range for general
purpose arithmetic.
\par
These findings underscore that takums strike a fine balance between small and 
large values, rendering them apt for general-purpose arithmetic, while also 
incorporating the advantages of tapered floating-point arithmetic previously 
associated with posits \cite{posits-beating_floating-point-2017}.
\FloatBarrier
\subsection{Dynamic Range}
As illustrated in Figure~\ref{fig:dynamic_range}, takums exhibit a consistent 
dynamic range across various bit-lengths, efficiently achieving this range even 
with a limited number of bits. In contrast, posits only achieve a comparable 
dynamic range with more than 47 bits, while IEEE 754 floating-point numbers
demonstrate inadequate dynamic range for 32 bits and fewer, and excessive dynamic
range for 64 bits and beyond. Notably, the proprietary formats \texttt{bfloat16} and 
\texttt{TF32} display insufficient dynamic range to serve as general-purpose 
arithmetic formats.
\par
Figure~\ref{fig:dynamic_range} further demonstrates that subnormal numbers 
insignificantly extend the dynamic range of IEEE 754 floating-point numbers, 
which elucidates why \texttt{bfloat16} and \texttt{TF32} have omitted them to 
reduce overhead and questions their overall utility.
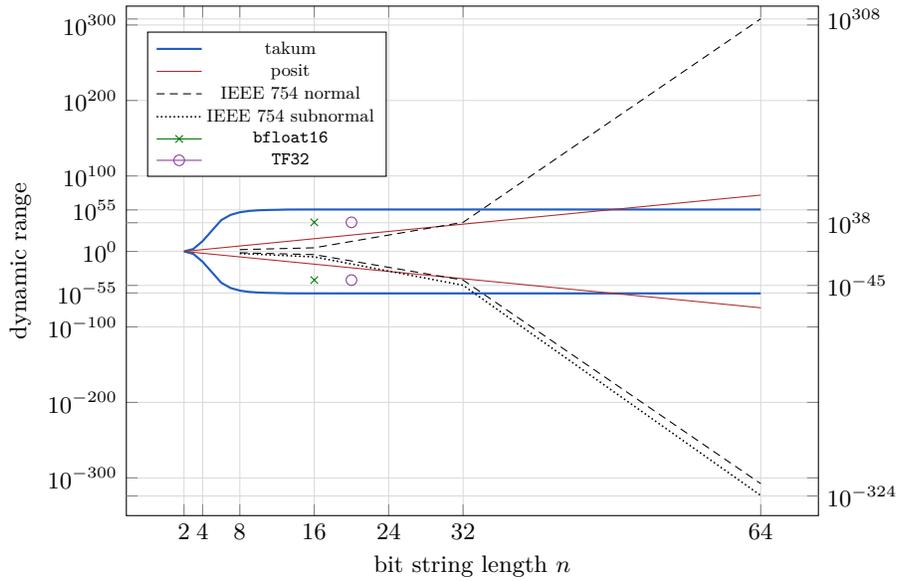
\begin{figure}[htbp]
	\begin{center}
		\begin{tikzpicture}
			\begin{axis}[
				scale only axis,
				height=\textwidth/1.8,
				width=\textwidth-3cm,
				ymin=10^-350,
				ymax=10^325,
				xlabel={bit string length $n$},
				ylabel={dynamic range},
				ymode=log,
				xtick={2,4,8,16,24,32,64},
				xminorticks=true,
				yminorticks=true,
				ytick={10^-400,10^-300,10^-200,10^-100,10^-55, 1,10^55,10^100,10^200,10^300,10^400},
				extra y ticks={10^-324,10^-45,10^38,10^308},
				extra y tick style={%
					grid=major,
					ticklabel pos=right,
				},
				grid=both,
				minor y tick num=9,
				grid style={line width=.1pt, draw=gray!10},
				major grid style={line width=.2pt,draw=gray!30},
				legend style={nodes={scale=0.75, transform shape}},
				legend style={at={(0.03,0.95)},anchor=north west}
			]
				\addplot[characteristic,thick] table [x=n, y=takum-min, col 
				sep=comma] {out/dynamic_range.csv};
				\addplot[characteristic,thick,forget plot] table [x=n, y=takum-max, col 
				sep=comma] {out/dynamic_range.csv};
				\addlegendentry{takum};
				\addplot[sign] table [x=n, y=posit2-min, col sep=comma] {out/dynamic_range.csv};
				\addplot[sign,forget plot] table [x=n, y=posit2-max, col sep=comma] {out/dynamic_range.csv};
				\addlegendentry{posit};
				\addplot[densely dashed] table [x=n, y=ieee-normal-min, col sep=comma] {out/dynamic_range.csv};
				\addlegendentry{IEEE 754 normal};
				\addplot[semithick,densely dotted] table [x=n, y=ieee-subnormal-min, col sep=comma] {out/dynamic_range.csv};
				\addlegendentry{IEEE 754 subnormal};
				\addplot[forget plot,densely dashed] table [x=n, y=ieee-max, col sep=comma] {out/dynamic_range.csv};
				\addplot[direction,mark=x,forget plot] coordinates {(16,1.175494351e-38)};
				\addplot[direction,mark=x] coordinates {(16,3.38953139e38)};
				\addlegendentry{\texttt{bfloat16}};
				\addplot[regime,mark=o,forget plot] coordinates {(20,1.175494351e-38)};
				\addplot[regime,mark=o] coordinates {(20,3.38953139e38)};
				\addlegendentry{\texttt{TF32}}
			\end{axis}
		\end{tikzpicture}
	\end{center}
	\caption{Dynamic range comparison between various number formats relative
		to the bit string length $n$.}
	\label{fig:dynamic_range}
\end{figure}
\subsection{Absolute and Relative Approximation Error}
\label{sec:evaluation-relative_error}
In this section, we explore both the absolute and relative approximation errors 
of takums compared to other numerical formats. Drawing insights from the 
discussion in Section~\ref{subsec:dynamic_range_criteria} and inspired by 
\cite{posits-good-bad-ugly-2019}, we conduct benchmark tests on takums, posits, 
and IEEE 754 floating-point formats across various bit lengths using a set of 
physical constants. The first group comprises the six constants defining the 
International System of Units (SI), namely the \textsc{Planck} constant 
$\mathrm{h}$, the \textsc{Boltzmann} constant $\mathrm{k}$, the elementary 
charge $\mathrm{e}$ (not to be confused with \textsc{Euler}'s number),
the speed of light $\mathrm{c}$, the caesium standard 
$\mathrm{\Delta\nu}$, and the \textsc{Avogadro} constant $\mathrm{N_A}$ (refer 
to Table~\ref{tab:relative_error}). The second group encompasses two constants 
of vastly differing magnitudes: the cosmological constant $\mathrm{\Lambda}$ 
and the mass of the universe $\mathrm{M}$ (refer to 
Table~\ref{tab:relative_error-large_constants}). All values are rounded to the 
same number of significant digits as the ground truth values and are presented 
without units for conciseness.
\par
It should be noted that these benchmarks are not intended to be the forte of 
tapered formats. On the contrary, 
\cite{posits-good-bad-ugly-2019} introduced this approach to highlight the 
shortcomings of posits when dealing with numbers of extremely small or large 
magnitudes. This is where uniform precision formats, such as 
IEEE 754, hold an advantage. Therefore, the aspiration with tapered precision 
formats like posits and takums is, at the very least, to achieve 
performance comparable to that of uniform precision formats with an equivalent 
number of bits.
\par
\begin{table}[htbp]
	\caption{
		Comparison of representations of the SI-defining constants
		in various formats.
	}
	\footnotesize
	\begin{center}
	\bgroup
	\def\arraystretch{1.2}
	\setlength{\tabcolsep}{0.15em}
	\begin{tabular}{| c || p{3.2cm} | p{3.2cm} | p{3.2cm} |}
		\hline
		name & \textsc{Planck} constant & \textsc{Boltzmann} constant &
			elementary charge\\
		\hline
		symbol & $\mathrm{h}$ & $\mathrm{k}$ & $\mathrm{e}$\\
		\hline
		value & $6.62607015 \times 10^{-34}$ &
			$1.380649 \times 10^{-23}$ &
			$1.602176634 \times 10^{-19}$\\
		\hline\hline
		\texttt{float8} &
			$\textcolor{error}{0}$ &
			$\textcolor{error}{0}$ &
			$\textcolor{error}{0}$ \\
		\hline
		\texttt{posit8} &
			$\textcolor{error}{5.96046448 \times 10^{-8}}$ &
			$\textcolor{error}{5.960464 \times 10^{-8}}$ &
			$\textcolor{error}{5.960464478 \times 10^{-8}}$ \\
 		\hline
 		\cellcolor{cellbg}\texttt{takum8} &
 			\cellcolor{cellbg}$\textcolor{error}{2.97569687 \times 10^{-35}}$ &
 			\cellcolor{cellbg}$\textcolor{error}{4.303623} \times 10^{-23}$ &
 			\cellcolor{cellbg}$1.\textcolor{error}{282891824} \times 10^{-19}$ \\
  		\hline\hline
		\texttt{float16} &
 			$\textcolor{error}{0}$ &
 			$\textcolor{error}{0}$ &
 			$\textcolor{error}{0}$ \\
		\hline
		\texttt{bfloat16} &
 			$6.62\textcolor{error}{038418} \times 10^{-34}$ &
 			$1.38\textcolor{error}{5528} \times 10^{-23}$ &
 			$1.60\textcolor{error}{0892270} \times 10^{-19}$ \\
		\hline
		\texttt{posit16} &
 			$\textcolor{error}{1.38777878 \times 10^{-17}}$ &
 			$\textcolor{error}{1.387779 \times 10^{-17}}$ &
 			$\textcolor{error}{1.387778781 \times 10^{-17}}$ \\
 		\hline
 		\cellcolor{cellbg}\texttt{takum16} &
  			\cellcolor{cellbg}$6.\textcolor{error}{56428218} \times 10^{-34}$ &
  			\cellcolor{cellbg}$1.3\textcolor{error}{75520} \times 10^{-23}$ &
  			\cellcolor{cellbg}$1.\textcolor{error}{596584671} \times 10^{-19}$ \\
  		\hline\hline
  		\texttt{TF32} &
  			$6.62\textcolor{error}{790735} \times 10^{-34}$ &
  			$1.380\textcolor{error}{358} \times 10^{-23}$ &
  			$1.60\textcolor{error}{1951062} \times 10^{-19}$ \\
  		\hline
   		\texttt{posit19} &
   			$\textcolor{error}{3.38813179 \times 10^{-21}}$ &
   			$\textcolor{error}{3.388132 \times 10^{-21}}$ &
   			$\textcolor{error}{2.168404345} \times 10^{-19}$ \\
   		\hline
  		\cellcolor{cellbg}\texttt{takum19} &
  			\cellcolor{cellbg}$6.6\textcolor{error}{1576649} \times 10^{-34}$ &
  			\cellcolor{cellbg}$1.380\textcolor{error}{904} \times 10^{-23}$ &
  			\cellcolor{cellbg}$1.602\textcolor{error}{833526} \times 10^{-19}$ \\
  		\hline\hline
		\texttt{float32} &
 			$6.6260701\textcolor{error}{8} \times 10^{-34}$ &
 			$1.380649 \times 10^{-23}$ &
 			$1.602176\textcolor{error}{598} \times 10^{-19}$ \\
		\hline
		\texttt{posit32} &
 			$\textcolor{error}{7.70371978} \times 10^{-34}$ &
 			$1.380\textcolor{error}{358} \times 10^{-23}$ &
 			$1.602\textcolor{error}{215759} \times 10^{-19}$ \\
 		\hline
 		\cellcolor{cellbg}\texttt{takum32} &
  			\cellcolor{cellbg}$6.62607\textcolor{error}{126} \times 10^{-34}$ &
  			\cellcolor{cellbg}$1.380649 \times 10^{-23}$ &
  			\cellcolor{cellbg}$1.602176\textcolor{error}{753} \times 10^{-19}$ \\
  		\hline
	\end{tabular}
	\\[0.5cm]
	\begin{tabular}{| c || p{3.2cm} | p{3.2cm} | p{3.2cm} |}
		\hline
		name & speed of light & caesium standard & \textsc{Avogadro} constant\\
		\hline
		symbol & $\mathrm{c}$ & $\mathrm{\Delta\nu}$ & $\mathrm{N_A}$ \\
		\hline
		value & $2.99792458 \times 10^8$ &
			$9.192631770 \times 10^9$ &
			$6.02214076 \times 10^{23}$\\
		\hline\hline
		\texttt{float8} &
			$\textcolor{error}{\infty}$ &
			$\textcolor{error}{\infty}$ &
			$\textcolor{error}{\infty}$ \\
		\hline
		\texttt{posit8} &
			$\textcolor{error}{1.67772160 \times 10^{7}}$ &
			$\textcolor{error}{1.677721600 \times 10^{7}}$ &
			$\textcolor{error}{1.67772160 \times 10^{7}}$ \\
 		\hline
 		\cellcolor{cellbg}\texttt{takum8} &
 			\cellcolor{cellbg}$2.9\textcolor{error}{4267566} \times 10^{8}$ &
 			\cellcolor{cellbg}$\textcolor{error}{1.606646472 \times 10^{10}}$ &
 			\cellcolor{cellbg}$\textcolor{error}{1.26865561 \times 10^{24}}$ \\
  		\hline\hline
		\texttt{float16} &
 			$\textcolor{error}{\infty}$ &
 			$\textcolor{error}{\infty}$ &
 			$\textcolor{error}{\infty}$ \\
		\hline
		\texttt{bfloat16} &
 			$2.99\textcolor{error}{892736} \times 10^{8}$ &
 			$9.19\textcolor{error}{3914368} \times 10^{9}$ &
 			$6.02\textcolor{error}{101727} \times 10^{23}$ \\
		\hline
		\texttt{posit16} &
 			$\textcolor{error}{3.01989888} \times 10^{8}$ &
 			$9.\textcolor{error}{663676416} \times 10^{9}$ &
 			$\textcolor{error}{7.20575940 \times 10^{16}}$ \\
 		\hline
 		\cellcolor{cellbg}\texttt{takum16} &
  			\cellcolor{cellbg}$2.9\textcolor{error}{8901606} \times 10^{8}$ &
  			\cellcolor{cellbg}$9.\textcolor{error}{226194467} \times 10^{9}$ &
  			\cellcolor{cellbg}$\textcolor{error}{5.99270479} \times 10^{23}$ \\
  		\hline\hline
  		\texttt{TF32} &
  			$2.99\textcolor{error}{892736} \times 10^{8}$ &
  			$9.19\textcolor{error}{3914368} \times 10^{9}$ &
  			$6.02\textcolor{error}{101727} \times 10^{23}$ \\
  		\hline
   		\texttt{posit19} &
   			$2.99\textcolor{error}{892736} \times 10^{8}$ &
   			$9.1\textcolor{error}{26805504} \times 10^{9}$ &
   			$\textcolor{error}{2.95147905 \times 10^{20}}$ \\
   		\hline
  		\cellcolor{cellbg}\texttt{takum19} &
  			\cellcolor{cellbg}$2.997\textcolor{error}{78578} \times 10^{8}$ &
  			\cellcolor{cellbg}$9.19\textcolor{error}{0224944} \times 10^{9}$ &
  			\cellcolor{cellbg}$6.02\textcolor{error}{792137} \times 10^{23}$ \\
  		\hline\hline
		\texttt{float32} &
 			$2.997924\textcolor{error}{48} \times 10^{8}$ &
 			$9.192631\textcolor{error}{296} \times 10^{9}$ &
 			$6.022140\textcolor{error}{64} \times 10^{23}$ \\
		\hline
		\texttt{posit32} &
 			$2.99792\textcolor{error}{384} \times 10^{8}$ &
 			$9.19263\textcolor{error}{6416} \times 10^{9}$ &
 			$6.02\textcolor{error}{101727} \times 10^{23}$ \\
 		\hline
 		\cellcolor{cellbg}\texttt{takum32} &
  			\cellcolor{cellbg}$2.997924\textcolor{error}{44} \times 10^{8}$ &
  			\cellcolor{cellbg}$9.19263\textcolor{error}{2204} \times 10^{9}$ &
  			\cellcolor{cellbg}$6.022140\textcolor{error}{98} \times 10^{23}$ \\
  		\hline
	\end{tabular}
	\egroup
	\end{center}
	\label{tab:relative_error}
\end{table}
In the first group (Table~\ref{tab:relative_error}), it is evident that only 
\texttt{takum8} consistently matches the magnitudes of the constants among all 
8-bit types. While \texttt{float8} either underflows or overflows, 
\texttt{posit8}'s employment of saturation arithmetic unsuccessfully strives to provide an 
answer that is as accurate as feasible. With 16 bits, \texttt{takum16} 
demonstrates comparable to but slightly worse performance than \texttt{bfloat16}, despite the former 
possessing a roughly $50\%$ larger dynamic range. Conversely, \texttt{posit16} displays inferior accuracy 
overall and suffers from inadequate dynamic range for most constants. 
Across 19 and 32 bits, takums and IEEE 754 floating-point numbers are
on par in performance, while posits consistently exhibit diminished 
accuracy.
\begin{table}[htbp]
	\caption{
		Comparison of representations of
		two very large physical constants in various 
		formats.
	}
	\footnotesize
	\begin{center}
	\bgroup
	\def\arraystretch{1.2}
	\setlength{\tabcolsep}{0.15em}
	\begin{tabular}{| c || p{3.2cm} | p{3.2cm} |}
		\hline
		name & cosmological constant & mass of the universe \\
		\hline
		symbol & $\mathrm{\Lambda}$ & $\mathrm{M}$ \\
		\hline
		value & $1.1056 \times 10^{-52}$ &
			$1.5 \times 10^{53}$\\
		\hline\hline
		\texttt{float8} &
			$\textcolor{error}{0}$ &
			$\textcolor{error}{\infty}$ \\
		\hline
		\texttt{posit8} &
			$\textcolor{error}{5.9605 \times 10^{-8}}$ &
			$\textcolor{error}{1.7 \times 10^{7}}$ \\
 		\hline
 		\cellcolor{cellbg}\texttt{takum8} &
 			\cellcolor{cellbg}$1.\textcolor{error}{2642} \times 10^{-52}$ &
 			\cellcolor{cellbg}$\textcolor{error}{7.9 \times 10^{51}}$ \\
  		\hline\hline
		\texttt{float16} &
 			$\textcolor{error}{0}$ &
 			$\textcolor{error}{\infty}$ \\
		\hline
		\texttt{bfloat16} &
 			$\textcolor{error}{0}$ &
 			$\textcolor{error}{\infty}$ \\
		\hline
		\texttt{posit16} &
 			$\textcolor{error}{1.3878 \times 10^{-17}}$ &
 			$\textcolor{error}{7.2 \times 10^{16}}$ \\
 		\hline
 		\cellcolor{cellbg}\texttt{takum16} &
  			\cellcolor{cellbg}$1.1\textcolor{error}{156} \times 10^{-52}$ &
  			\cellcolor{cellbg}$1.5 \times 10^{53}$ \\
  		\hline\hline
  		\texttt{TF32} &
  			$\textcolor{error}{0}$ &
  			$\textcolor{error}{\infty}$ \\
  		\hline
   		\texttt{posit19} &
   			$\textcolor{error}{3.3881 \times 10^{-21}}$ &
   			$\textcolor{error}{3.0 \times 10^{20}}$ \\
   		\hline
  		\cellcolor{cellbg}\texttt{takum19} &
  			\cellcolor{cellbg}$1.10\textcolor{error}{70} \times 10^{-52}$ &
  			\cellcolor{cellbg}$1.5 \times 10^{53}$ \\
  		\hline\hline
		\texttt{float32} &
 			$\textcolor{error}{0}$ &
 			$\textcolor{error}{\infty}$ \\
		\hline
		\texttt{posit32} &
 			$\textcolor{error}{7.5232 \times 10^{-37}}$ &
 			$\textcolor{error}{1.3 \times 10^{36}}$ \\
 		\hline
 		\cellcolor{cellbg}\texttt{takum32} &
  			\cellcolor{cellbg}$1.1056 \times 10^{-52}$ &
  			\cellcolor{cellbg}$1.5 \times 10^{53}$ \\
  		\hline
	\end{tabular}
	\egroup
	\end{center}
	\label{tab:relative_error-large_constants}
\end{table}
\par
The second group (Table~\ref{tab:relative_error-large_constants}) further 
emphasises the dynamic range of each number format. Among the 8-bit formats, 
none can accurately express the magnitude of the mass of the universe, though \texttt{takum8} comes closest and even expresses the magnitude of the cosmological constant accurately. The disparity becomes more 
pronounced with 16 bits: whereas \texttt{float16} and \texttt{bfloat16} either 
underflow or overflow to $0$ and $\infty$ respectively, and \texttt{posit16} 
yields its minimum and maximum representable values, which deviate 
approximately 40 orders of magnitude from each constant, \texttt{takum16} 
successfully captures the magnitude of both constants and even two 
significant digits. This trend persists with 19 bits.
\par
At 32 bits, \texttt{float32} still fails to represent either constant 
accurately, while \texttt{posit32} remains around 20 orders of magnitude 
adrift. In stark contrast, \texttt{takum32} precisely represents both constants 
with all significant digits intact, particularly noteworthy as the cosmological 
constant involves five significant digits. This exemplifies the inadequacy of 
\texttt{bfloat16}, \texttt{TF32}, \texttt{float32}, and \texttt{posit32} in 
handling general-purpose arithmetic, especially given the likelihood of 
encountering such large numbers as intermediate computational results (e.g., 
when squaring the Boltzmann constant), underscoring the potential of 
\texttt{takum32} and takums in general as general purpose arithmetic
formats.
\par
Regarding the relative approximation error, an upper bound $\lambda$ for the 
relative approximation error of takum has been derived in 
Proposition~\ref{prop:takum-precision}, which varies depending on the number of 
available fraction bits. In contrast, IEEE 754 floating-point numbers and 
posits, both featuring linear significands rather than logarithmic ones like takum, 
are characterised by the well-known $\varepsilon$ (refer to 
(\ref{eq:ieee-epsilon})) as an upper bound for the relative approximation 
error. This bound remains constant for each type of IEEE 754 floating-point 
number, while it varies for posits based on the number of available fraction 
bits. Proposition~\ref{prop:takum-precision} has demonstrated that, for the 
same number of available fraction bits, $\lambda$ is smaller than
$\frac{2}{3} \varepsilon$, indicating that logarithmic significands offer better relative accuracy.
\par
Subsequently, we proceed to explore the upper bounds $\lambda$ and 
$\varepsilon$ for different bit lengths. In the ensuing plots, we depict the 
negative binary logarithm of $\lambda$ and $\varepsilon$. Consequently, the 
graphs are inverted, with higher values denoting greater precision.
\begin{figure}[htbp]
	\begin{center}
		\begin{tikzpicture}
			\begin{axis}[
				axis on top,
				scale only axis,
				width=\textwidth/1.1,
				height=\textwidth/2.5,
				xlabel={$\log_{10}(x)$},
				ylabel={$-\log_2(\lambda), -\log_2(\varepsilon)$},
				ymin=-0.5,
				ymax=5.5,
				xmin=-60,
				xmax=60,
				extra x ticks={
					-51.95640196996987,
					-33.17874397056745,
					-22.859916717321767,
					-18.795289606265012,
					8.476820702927926,
					9.963439863895655,
					23.77975090238511,
					53.176091259055674
				},
				extra x tick labels={
					$\Lambda$,
					$\mathrm{h}$,
					$\mathrm{k}$,
					$\mathrm{e}$,
					{$\mathrm{c},\!\Delta\mathrm{\nu}$}, 
					{},
					$\mathrm{N_A}$,
					$\mathrm{M}$
				},
				extra x tick style={%
					grid=major,
					ticklabel pos=top,
				},
				legend style={nodes={scale=0.75, transform shape}},
				legend style={at={(0.03,0.94)},anchor=north west},
			]
				\addplot [draw=yellow!60,fill=yellow!60, semitransparent,forget plot]
					coordinates {(-0.301030,-1) (-0.301030,5.5) (0.301030,5.5) (0.301030,-1)};
				\addplot[const plot,characteristic,thick] table [x=e, y=takum8, col 
				sep=comma]{out/error.csv};
				\addlegendentry{\texttt{takum8}};
				\addplot[const plot,sign] table [x=e, y=posit8, col sep=comma]{out/error.csv};
				\addlegendentry{\texttt{posit8}};
				\addplot[const plot,densely dashed] table [x=e, y=float8, col sep=comma]{out/error.csv};
				\addlegendentry{\texttt{float8}};
			\end{axis}
		\end{tikzpicture}
	\end{center}
	\caption{
		Comparison of negative logarithmic relative approximation error upper bounds within an 8-bit budget for encoded numbers $x$.
		Higher values indicate lower relative error and, consequently, higher precision.
	}
	\label{fig:relative_error-8-bit}
\end{figure}
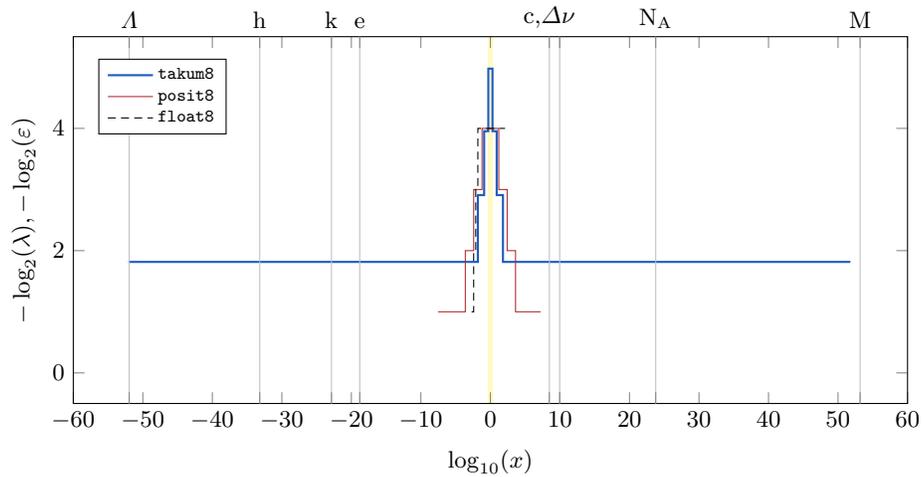
Starting with 8 bits, we conduct a comparative analysis among 
\texttt{takum8}, \texttt{posit8}, and \texttt{float8}, as illustrated in 
Figure~\ref{fig:relative_error-8-bit}. It becomes immediately apparent that 
\texttt{takum8} exhibits a markedly superior dynamic range in contrast to both 
\texttt{posit8} and \texttt{float8}.
Although the benchmark constants surpass the dynamic ranges of both 
\texttt{posit8} and \texttt{float8}, \texttt{takum8} encompasses all of these 
constants except $\mathrm{\Lambda}$ and $\mathrm{M}$, which are narrowly missed but still adequately represented. Particularly
noteworthy is the fact that, despite its expansive 
dynamic range, \texttt{takum8} demonstrates the lowest minimum relative 
approximation error, surpassing both \texttt{posit8} and \texttt{float8}. 
However, \texttt{posit8} offers a slightly wider range of numbers with more than 0 
fraction bits.
This plot introduces the concept of the \enquote{golden zone,} highlighted in 
yellow, 
which will be consistently referenced in all subsequent plots. The golden zone 
delineates the range of numbers wherein \texttt{takum} either matches or 
surpasses the performance of the reference IEEE 754 floating-point format.
In this instance, the golden zone appears relatively narrow, which is 
unsurprising considering that \texttt{float8} exhibits a very limited dynamic 
range coupled with relatively high overall precision. The presence of a 
discernible golden zone, despite the considerable dynamic range of 
\texttt{takum8}, underscores its commendable performance.
\par
\begin{figure}[htbp]
	\begin{center}
		\begin{tikzpicture}
			\begin{axis}[
				axis on top,
				scale only axis,
				width=\textwidth/1.1,
				height=\textwidth/2.5,
				xlabel={$\log_{10}(x)$},
				ylabel={$-\log_2(\lambda), -\log_2(\varepsilon)$},
				ymin=-1,
				ymax=14,
				xmin=-60,
				xmax=60,
				extra x ticks={
					-51.95640196996987,
					-33.17874397056745,
					-22.859916717321767,
					-18.795289606265012,
					8.476820702927926,
					9.963439863895655,
					23.77975090238511,
					53.176091259055674
				},
				extra x tick labels={
					$\Lambda$,
					$\mathrm{h}$,
					$\mathrm{k}$,
					$\mathrm{e}$,
					{$\mathrm{c},\!\Delta\mathrm{\nu}$}, 
					{},
					$\mathrm{N_A}$,
					$\mathrm{M}$
				},
				extra x tick style={%
					grid=major,
					ticklabel pos=top,
				},
				legend style={nodes={scale=0.75, transform shape}},
				legend style={at={(0.03,0.94)},anchor=north west},
			]
				\addplot [draw=yellow!60,fill=yellow!60, semitransparent,forget plot]
					coordinates {(-13.54635,-1) (-13.54635,14) (13.54635,14) (13.54635,-1)};
				\addplot[const plot,characteristic,thick] table [x=e, y=takum16, col 
				sep=comma]{out/error.csv};
				\addlegendentry{\texttt{takum16}};
				\addplot[const plot,sign] table [x=e, y=posit16, col sep=comma]{out/error.csv};
				\addlegendentry{\texttt{posit16}};
				\addplot[const plot,densely dashed] table [x=e, y=float16, col sep=comma]{out/error.csv};
				\addlegendentry{\texttt{float16}};
				\addplot[const plot,direction,semithick,densely dotted] table [x=e, y=bfloat16, col sep=comma]{out/error.csv};
				\addlegendentry{\texttt{bfloat16}};
			\end{axis}
		\end{tikzpicture}
	\end{center}
	\caption{
		Comparison of negative logarithmic relative approximation error upper bounds within a 16-bit budget for encoded numbers $x$. Higher values indicate lower relative error and, consequently, higher precision.
	}
	\label{fig:relative_error-16-bit}
\end{figure}
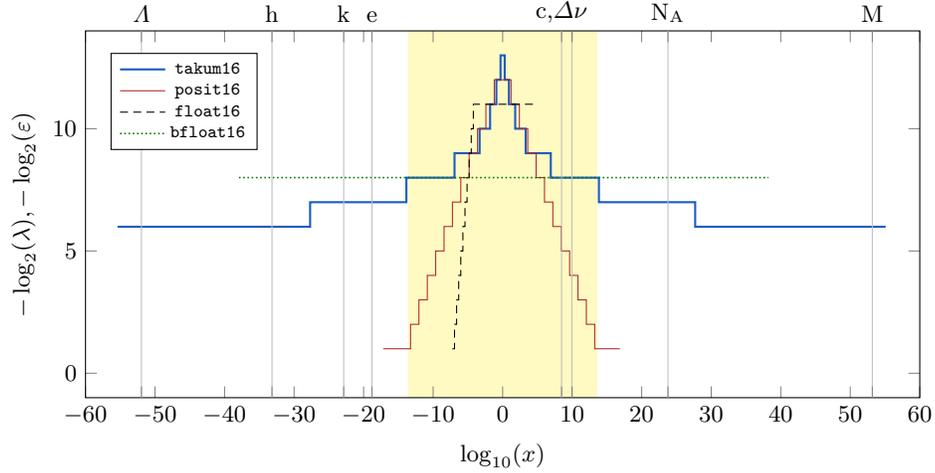
The comparison depicted in Figure~\ref{fig:relative_error-16-bit} introduces 
\texttt{bfloat16} as a fourth format alongside \texttt{takum16}, 
\texttt{posit16}, and \texttt{float16}. This addition is necessitated by the 
limited dynamic range of \texttt{float16}, which, among other reasons, served 
as the impetus for the development of \texttt{bfloat16} 
initially\cite{bfloat16}. To ensure a fair evaluation of 
16-bit floating-point formats, \texttt{bfloat16} is included as a benchmark 
against IEEE 754 floating-point numbers, notwithstanding its lack of 
standardisation. Subsequently, we adopt this approach throughout, including 
within the golden zone.
Observations reveal a noteworthy extension of the golden zone, which now 
encompasses one of the benchmark constants ($\mathrm{c}$). Additionally, the 
plot effectively illustrates the inadequacy of \texttt{bfloat16}'s dynamic 
range as it does not nearly encompass two of the benchmark constants
(see Table~\ref{tab:relative_error-large_constants}).
Conversely, \texttt{takum16} encompasses all benchmark constants.
\par
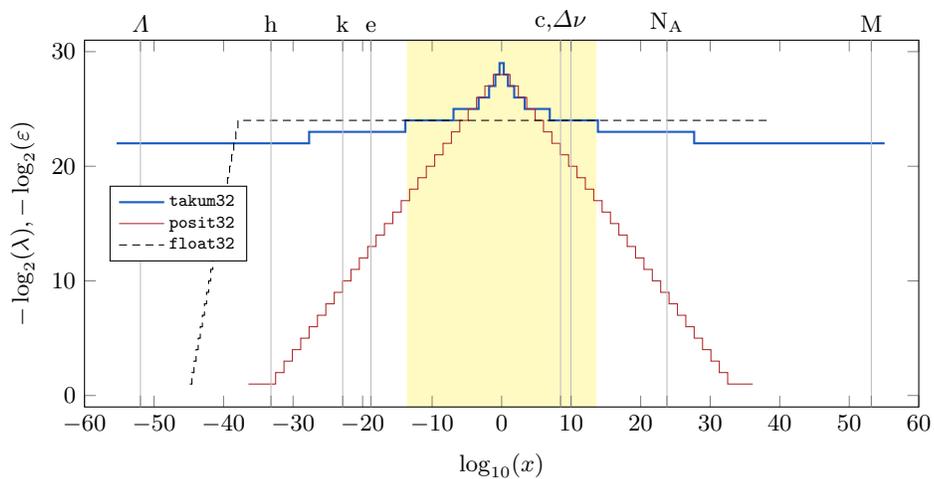
\begin{figure}[htbp]
	\begin{center}
		\begin{tikzpicture}
			\begin{axis}[
				axis on top,
				scale only axis,
				width=\textwidth/1.1,
				height=\textwidth/2.5,
				xlabel={$\log_{10}(x)$},
				ylabel={$-\log_2(\lambda), -\log_2(\varepsilon)$},
				ymin=-1,
				ymax=31,
				xmin=-60,
				xmax=60,
				extra x ticks={
					-51.95640196996987,
					-33.17874397056745,
					-22.859916717321767,
					-18.795289606265012,
					8.476820702927926,
					9.963439863895655,
					23.77975090238511,
					53.176091259055674
				},
				extra x tick labels={
					$\Lambda$,
					$\mathrm{h}$,
					$\mathrm{k}$,
					$\mathrm{e}$,
					{$\mathrm{c},\!\Delta\mathrm{\nu}$}, 
					{},
					$\mathrm{N_A}$,
					$\mathrm{M}$
				},
				extra x tick style={%
					grid=major,
					ticklabel pos=top,
				},
				legend style={nodes={scale=0.75, transform shape}},
				legend style={at={(0.03,0.5)},anchor=west},
			]
				\addplot [draw=yellow!60,fill=yellow!60, semitransparent,forget plot]
					coordinates {(-13.54635,-1) (-13.54635,31) (13.54635,31) 
					(13.54635,-1)};
				\addplot[const plot,characteristic,thick] table [x=e, y=takum32, col 
				sep=comma]{out/error.csv};
				\addlegendentry{\texttt{takum32}};
				\addplot[const plot,sign] table [x=e, y=posit32, col sep=comma]{out/error.csv};
				\addlegendentry{\texttt{posit32}};
				\addplot[const plot,densely dashed] table [x=e, y=float32, col sep=comma]{out/error.csv};
				\addlegendentry{\texttt{float32}};
			\end{axis}
		\end{tikzpicture}
	\end{center}
	\caption{
				Comparison of negative logarithmic relative approximation error upper bounds within a 32-bit budget for encoded numbers $x$.
				Higher values indicate lower relative error and, consequently, higher precision.
	}
	\label{fig:relative_error-32-bit}
\end{figure}
The 32-bit comparison (refer to Figure~\ref{fig:relative_error-32-bit}) 
exhibits similarities with the 16-bit comparison in the sense that the dynamic 
range of the IEEE 754 reference remains unchanged. The distinguishing factor 
lies in the inclusion of subnormals in \texttt{float32} and a reduced 
overall relative approximation error across all formats owing to the augmented 
bit count.
\par
\begin{figure}[htbp]
	\begin{center}
		\begin{tikzpicture}
			\begin{axis}[
				axis on top,
				scale only axis,
				width=\textwidth/1.1,
				height=\textwidth/2.5,
				xlabel={$\log_{10}(x)$},
				ylabel={$-\log_2(\lambda), -\log_2(\varepsilon)$},
				ymin=-1,
				ymax=66,
				xmin=-60,
				xmax=60,
				extra x ticks={
					-51.95640196996987,
					-33.17874397056745,
					-22.859916717321767,
					-18.795289606265012,
					8.476820702927926,
					9.963439863895655,
					23.77975090238511,
					53.176091259055674
				},
				extra x tick labels={
					$\Lambda$,
					$\mathrm{h}$,
					$\mathrm{k}$,
					$\mathrm{e}$,
					{$\mathrm{c},\!\Delta\mathrm{\nu}$}, 
					{},
					$\mathrm{N_A}$,
					$\mathrm{M}$
				},
				extra x tick style={%
					grid=major,
					ticklabel pos=top,
				},
				legend style={nodes={scale=0.75, transform shape}},
				legend style={at={(0.03,0.6)},anchor=west},
			]
				\addplot [draw=yellow!60,fill=yellow!60, semitransparent,forget plot]
					coordinates {(-55.39,-1) (-55.39,66) (55.08,66) (55.08,-1)};
				\addplot[const plot,characteristic,thick] table [x=e, y=takum64, col 
				sep=comma]{out/error.csv};
				\addlegendentry{\texttt{takum64}};
				\addplot[const plot,sign] table [x=e, y=posit64, col sep=comma]{out/error.csv};
				\addlegendentry{\texttt{posit64}};
				\addplot[const plot, densely dashed] table [x=e, y=float64, col sep=comma]{out/error.csv};
				\addlegendentry{\texttt{float64}};
			\end{axis}
		\end{tikzpicture}
	\end{center}
	\caption{
		Comparison of negative logarithmic relative approximation error upper
		bounds within a 64-bit budget for encoded numbers $x$. Higher values indicate lower relative error and, consequently, higher precision.
	}
	\label{fig:relative_error-64-bit}
\end{figure}
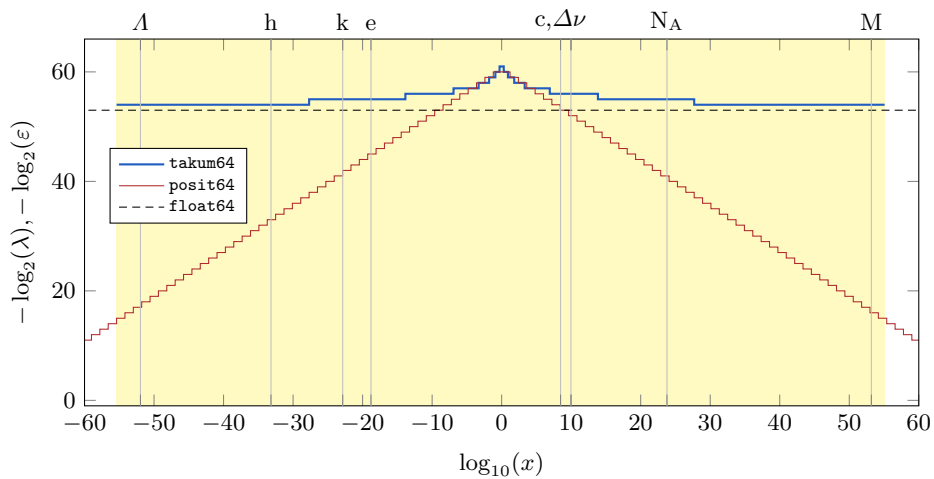
For the 64-bit plot (see Figure~\ref{fig:relative_error-64-bit}), it is evident 
that the golden zone spans the entire dynamic range of \texttt{takum64}, 
offering a significant advantage over and compatibility with \texttt{float64}. 
In contrast, \texttt{posit64} sacrifices precision across its dynamic range, 
whereas \texttt{takum64} maintains at least more than 53 binary orders of 
magnitude of precision consistently.
\par
Overall, takums offer a sufficiently broad dynamic range, and their relative accuracy exceeds that of posits near unity, nearly aligning with it for slightly larger exponents. Beyond this threshold, posits experience a rapid decline in precision, whereas takums maintain a high level of relative accuracy, typically deviating by at most 1-2 binary orders of magnitude from, or even surpassing at higher bit counts, the
respective reference IEEE 754 floating-point format.
%
\FloatBarrier
\subsection{Unary and Binary Arithmetic Closure}
\label{sec:evaluation-closure}
Besides the capability of a numerical format to represent a range of numbers, 
another crucial consideration is its efficacy in accurately representing the 
outcomes of arithmetic operations. The greater the number of arithmetic results 
that can be precisely represented, the lesser the accumulation of rounding 
errors across successive arithmetic operations. Considering $x$ and $y$ as 
numbers within a numerical system, $\round()$ as the associated rounding 
operation, and $\star$ as an arithmetic operation (such as addition, 
subtraction, multiplication, or division), we are concerned with both the 
absolute error
\begin{equation}
	e_{\text{abs}} := \round(x \star y) - (x \star y),
\end{equation}
and the relative error
\begin{equation}
	e_{\text{rel}} := \frac{\round(x \star y) - (x \star y)}{x \star y}.
\end{equation}
In order to condense the errors, which span a broad dynamic range, we introduce 
a clamped and logarithmically scaled \enquote{relative precision} in
binary magnitudes as
\begin{equation}
	\eta(e_{\text{rel}}) := \log_2(\max(1,-\log_2(|e_{\text{rel}}|))),
\end{equation}
which approaches $\infty$ when the relative error is zero and is bounded from 
below by zero, signifying a relative error greater than or equal to $10\%$.
\par
\begin{figure}[htbp]
	\begin{center}
		\subcaptionbox{\texttt{takum8}}{
			\begin{tikzpicture}
				\begin{axis}[
					scale only axis,
					width=\textwidth/3,
					height=\textwidth/3,
				]
					\addplot graphics[
						xmin=5.96046448e-8,
						xmax=1.67772216e+07,
						ymin=5.96046448e-8,
						ymax=1.67772216e+07,
					]{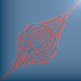};
				\end{axis}
			\end{tikzpicture}
		}
		\subcaptionbox{\texttt{posit8}}{
			\begin{tikzpicture}
				\begin{axis}[
					scale only axis,
					width=\textwidth/3,
					height=\textwidth/3,
				]
					\addplot graphics[
						xmin=5.96046448e-8,
						xmax=1.67772216e+07,
						ymin=5.96046448e-8,
						ymax=1.67772216e+07,
					]{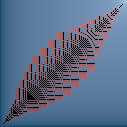};
				\end{axis}
			\end{tikzpicture}
		}
		\subcaptionbox{sorted relative precisions}{
			\begin{tikzpicture}
				\begin{axis}[
					scaled y ticks={base 10:-2},
					ytick={1e0,1e1,1e2},
					yticklabels={0,1,2},
					extra y ticks={500},
					extra y tick labels={$\infty$},
					grid=major,
					ymode=log,
					ylabel={$\eta(e_{\text{rel}})$},
					xmin=0,
					xmax=1,
					scale only axis,
					width=\textwidth/1.5,
					height=\textwidth/2.5,
					extra x ticks={0.0,0.07539683},
					extra x tick labels={{\quad\ \ \ $0.0\%$,$7.5\%$},},
			        extra x tick style={
						ticklabel pos=upper,
					},
					xtick={0,0.25,0.5,0.75,1.0},
	    			xticklabel={\pgfmathparse{\tick*100}\pgfmathprintnumber{\pgfmathresult}\%},
					point meta={x*100},
					legend style={nodes={scale=0.75, transform shape}},
					legend style={at={(0.92,0.82)},anchor=north east},
				]
					\addplot[characteristic,thick] table [col 
					sep=comma]{out/add-takum8.csv};
					\addlegendentry{\texttt{takum8}};
					\addplot[sign] table [col sep=comma]{out/add-posit8.csv};
					\addlegendentry{\texttt{posit8}};
				\end{axis}
			\end{tikzpicture}
		}
	\end{center}
	\caption{Addition closure analysis of \texttt{takum8} versus \texttt{posit8}
		within \texttt{posit8}'s positive dynamic range 
		$\left[2^{-24},2^{24}\right] \approx [\num{6.0e-8},\num{1.7e7}]$.}
	\label{fig:closure-8-bit-add}
\end{figure}
For small bit string lengths, such as $n = 8$, it is feasible to compute 
the relative error for every possible arithmetic operation and present it 
visually in a two-dimensional plot. In this plot, the x-axis corresponds to all 
values of $x$, while the y-axis corresponds to all values of $y$ (refer 
to Figure~\ref{fig:closure-8-bit-add} for an illustrative example). Each point 
$(x, y)$ on the plot is colour-coded based on the rounded outcome of the arithmetic 
operation $x \star y$: negative relative errors are shaded in blue, zero errors in black, and positive relative errors in red. Blue and red hues are scaled in a perceptually uniform manner with respect to lightness, signifying that an increase in error magnitude is represented by enhanced brightness, irrespective of the direction.
\par
To evaluate performance beyond mere visualisation, one can flatten the 
two-dimensional relative error \enquote{matrix} into an array, compute the 
relative 
precision, and then sort the array in descending order. By plotting this sorted 
array with the x-axis in percent units, the ratio of exact operations within 
the number system becomes immediately apparent (as depicted in 
Figure~\ref{fig:closure-8-bit-add}, where the exact operation ratios are 
$0.0\%$ and $7.5\%$ respectively). Both the concept of closure plots and 
the idea of sorting errors were adapted from 
\cite{posits-beating_floating-point-2017}, while the notion of relative 
precision is introduced here to enable the visualization of results in the 
plots, with higher values indicating better precision as opposed to worse.
\par
New in this work is the consideration of the dynamic range inherent in various 
number formats. Rather than indiscriminately encompassing all numbers within 
each format, we identify the smallest common denominator in dynamic range and 
confine all number formats to this scope, encompassing solely positive numbers. 
This standardisation ensures uniformity of dynamic range across all formats. 
The inclusion of negative numbers is deemed superfluous when conducting a 
comprehensive comparison of addition and subtraction, as well as multiplication 
and division.
\par
It is important to note that the approach of constraining the dynamic range 
places takums at a considerable disadvantage, given the format's expansive 
dynamic range. However, this method yields valuable insights into the 
feasibility of utilising takums within the same low-bit regimes for which 
posits have demonstrated efficacy, as evidenced in 
Section~\ref{sec:evaluation-relative_error}, where we have established the 
suitability of takums for general-purpose arithmetic based on their dynamic 
range and relative approximation error. For $n=8$, we adopt the dynamic range 
of \texttt{posit8} as $\left[ 2^{-24},2^{24} \right] \approx 
[\num{6.0e-8},\num{1.7e7}]$, and for $n=16$, we 
employ the dynamic range of \texttt{posit16} as $\left[ 2^{-56}, 2^{56} 
\right] \approx [\num{1.4e-17},\num{7.2e16}]$.
\par
An additional novel aspect, refining the approach presented in 
\cite{posits-beating_floating-point-2017}, is the extension of our analysis to 
$n=16$ bits, allowing for a direct comparison with \texttt{bfloat16}, a 
comparison not previously conducted.
\par
In terms of addition, the impact of the logarithmic significand is evident in 
Figure~\ref{fig:closure-8-bit-add}, where, for 8 bits, only $7.5\%$ of 
\texttt{posit8} additions yield exact results, compared to a mere $0.0\%$ for 
\texttt{takum8}. Despite differences in the frequency of exact operations, the 
sorted relative precisions between the two formats are strikingly similar.
\begin{figure}[htbp]
	\begin{center}
		\subcaptionbox{\texttt{takum16}}{
			\begin{tikzpicture}
				\begin{axis}[
					scale only axis,
					width=\textwidth/3,
					height=\textwidth/3,
				]
					\addplot graphics[
						xmin=1.387778780e-17,
						xmax=7.205759403e16,
						ymin=1.387778780e-17,
						ymax=7.205759403e16,
					]{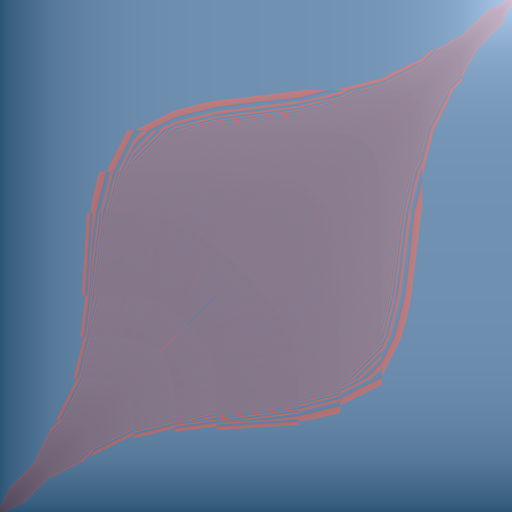};
				\end{axis}
			\end{tikzpicture}
		}
		\subcaptionbox{\texttt{posit16}}{
			\begin{tikzpicture}
				\begin{axis}[
					scale only axis,
					width=\textwidth/3,
					height=\textwidth/3,
				]
					\addplot graphics[
						xmin=1.387778780e-17,
						xmax=7.205759403e16,
						ymin=1.387778780e-17,
						ymax=7.205759403e16,
					]{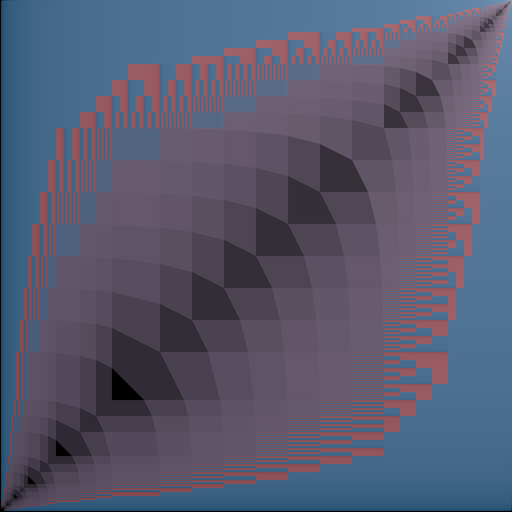};
				\end{axis}
			\end{tikzpicture}
		}
		\subcaptionbox{\texttt{bfloat16}}{
			\begin{tikzpicture}
				\begin{axis}[
					scale only axis,
					width=\textwidth/3,
					height=\textwidth/3,
				]
					\addplot graphics[
						xmin=1.387778780e-17,
						xmax=7.205759403e16,
						ymin=1.387778780e-17,
						ymax=7.205759403e16,
					]{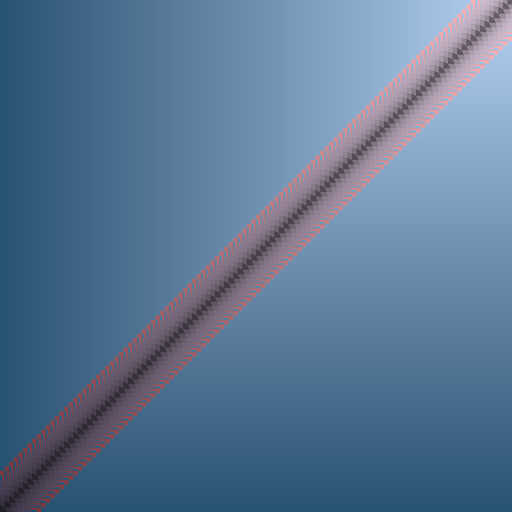};
				\end{axis}
			\end{tikzpicture}
		}
		\subcaptionbox{sorted relative precisions}{
			\begin{tikzpicture}
				\begin{axis}[
					scaled y ticks={base 10:-2},
					ytick={1e0,1e1,1e2},
					yticklabels={0,1,2},
					extra y ticks={500},
					extra y tick labels={$\infty$},
					grid=major,
					ymode=log,
					ylabel={$\eta(e_{\text{rel}})$},
					xmin=0,
					xmax=1,
					scale only axis,
					width=\textwidth/1.5,
					height=\textwidth/2.5,
					extra x ticks={0.0,0.01699995,0.08499998},
					extra x tick labels={{\quad\ \ \ 
						$0.0\%$,$1.7\%$,$8.5\%$},,},
			        extra x tick style={
						ticklabel pos=upper,
					},
					xtick={0,0.25,0.5,0.75,1.0},
	    			xticklabel={\pgfmathparse{\tick*100}\pgfmathprintnumber{\pgfmathresult}\%},
					point meta={x*100},
					legend style={nodes={scale=0.75, transform shape}},
					legend style={at={(0.92,0.82)},anchor=north east},
				]
					\addplot[characteristic,thick] table [col 
					sep=comma]{out/add-takum16.csv};
					\addlegendentry{\texttt{takum16}};
					\addplot[sign] table [col sep=comma]{out/add-posit16.csv};
					\addlegendentry{\texttt{posit16}};
					\addplot[direction, densely dashed] table [col sep=comma]{out/add-bfloat16.csv};
					\addlegendentry{\texttt{bfloat16}};
				\end{axis}
			\end{tikzpicture}
		}
	\end{center}
	\caption{Addition closure analysis of \texttt{takum16} versus \texttt{posit16}
		and \texttt{bfloat16}
		within \texttt{posit16}'s positive dynamic range 
		$\left[2^{-56},2^{56}\right] \approx [\num{1.4e-17},\num{7.2e16}]$.}
	\label{fig:closure-16-bit-add}
\end{figure}
This observation persists in the context of 16 bits, as depicted in 
Figure~\ref{fig:closure-16-bit-add}, where \texttt{bfloat16} has been 
incorporated into the comparison alongside \texttt{takum16} and 
\texttt{posit16}.
Notably, \texttt{posit16} exhibits the highest rate of exact 
additions at $8.5\%$, while \texttt{bfloat16} fares considerably poorer with 
only $1.7\%$ of additions resulting in exact outcomes. With a mere $0.0\%$ of 
additions being exact, \texttt{takum16} occupies the third position.
However, in non-exact cases, \texttt{takum16} outperforms \texttt{posit16}, 
displaying a higher relative precision, while both formats are behind
\texttt{bfloat16} in this regard.
\par
For further insights into the evaluations of subtraction, readers are directed 
to Figures~\ref{fig:closure-8-bit-sub} and \ref{fig:closure-16-bit-sub} in the 
appendix, which present evaluations for 8- and 16-bit scenarios, yielding 
comparable results.
\begin{figure}[htbp]
	\begin{center}
		\subcaptionbox{\texttt{takum8}}{
			\begin{tikzpicture}
				\begin{axis}[
					scale only axis,
					width=\textwidth/3,
					height=\textwidth/3,
				]
					\addplot graphics[
						xmin=5.96046448e-8,
						xmax=1.67772216e+07,
						ymin=5.96046448e-8,
						ymax=1.67772216e+07,
					]{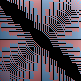};
				\end{axis}
			\end{tikzpicture}
		}
		\subcaptionbox{\texttt{posit8}}{
			\begin{tikzpicture}
				\begin{axis}[
					scale only axis,
					width=\textwidth/3,
					height=\textwidth/3,
				]
					\addplot graphics[
						xmin=5.96046448e-8,
						xmax=1.67772216e+07,
						ymin=5.96046448e-8,
						ymax=1.67772216e+07,
					]{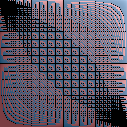};
				\end{axis}
			\end{tikzpicture}
		}
		\subcaptionbox{sorted relative precisions}{
			\begin{tikzpicture}
				\begin{axis}[
					scaled y ticks={base 10:-2},
					ytick={1e0,1e1},
					yticklabels={0,1},
					extra y ticks={30},
					extra y tick labels={$\infty$},
					grid=major,
					ymode=log,
					xmin=0,
					xmax=1,
					ylabel={$\eta(e_{\text{rel}})$},
					scale only axis,
					width=\textwidth/1.5,
					height=\textwidth/2.5,
					extra x ticks={0.2559524,0.4033537},
					extra x tick labels={$25.6\%$,$40.3\%$},
					extra x tick style={
						ticklabel pos=upper,
					},
					xtick={0,0.25,0.5,0.75,1.0},
					xticklabel={\pgfmathparse{\tick*100}\pgfmathprintnumber{\pgfmathresult}\%},
					point meta={x*100},
					legend style={nodes={scale=0.75, transform shape}},
					legend style={at={(0.92,0.82)},anchor=north east},
					]
					\addplot[characteristic,thick] table [col sep=comma]{out/mul-takum8.csv};
					\addlegendentry{\texttt{takum8}};
					\addplot[sign] table [col sep=comma]{out/mul-posit8.csv};
					\addlegendentry{\texttt{posit8}};
				\end{axis}
			\end{tikzpicture}
		}
	\end{center}
	\caption{
		Multiplication closure analysis of \texttt{takum8} versus \texttt{posit8}
		within \texttt{posit8}'s positive dynamic range 
		$\left[2^{-24},2^{24}\right] \approx [\num{6.0e-8},\num{1.7e7}]$.
	}
	\label{fig:closure-8-bit-mul}
\end{figure}
\par
In terms of multiplication, as illustrated for 8 bits in 
Figure~\ref{fig:closure-8-bit-mul}, it becomes evident that the logarithmic 
significand confers a distinct advantage. Specifically, \texttt{takum8} exhibits a 
noteworthy $40.3\%$ precision in exact multiplications, in stark contrast to 
\texttt{posit8}'s mere $25.6\%$. However, the precision for non-exact 
multiplications marginally diminishes for \texttt{takum8}, a consequence 
attributable to its larger dynamic range, thus resulting in a higher density 
of \texttt{posit8} when compared to \texttt{takum8} in the common dynamic range 
used in this analysis.
\begin{figure}[htbp]
	\begin{center}
		\subcaptionbox{\texttt{takum16}}{
			\begin{tikzpicture}
				\begin{axis}[
					scale only axis,
					width=\textwidth/3,
					height=\textwidth/3,
				]
					\addplot graphics[
						xmin=1.387778780e-17,
						xmax=7.205759403e16,
						ymin=1.387778780e-17,
						ymax=7.205759403e16,
					]{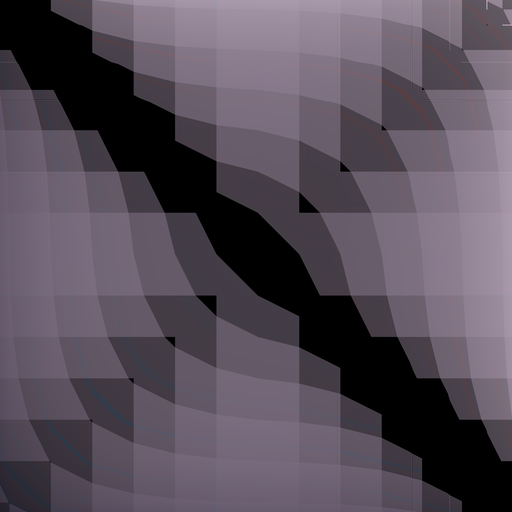};
				\end{axis}
			\end{tikzpicture}
		}
		\subcaptionbox{\texttt{posit16}}{
			\begin{tikzpicture}
				\begin{axis}[
					scale only axis,
					width=\textwidth/3,
					height=\textwidth/3,
				]
					\addplot graphics[
						xmin=1.387778780e-17,
						xmax=7.205759403e16,
						ymin=1.387778780e-17,
						ymax=7.205759403e16,
					]{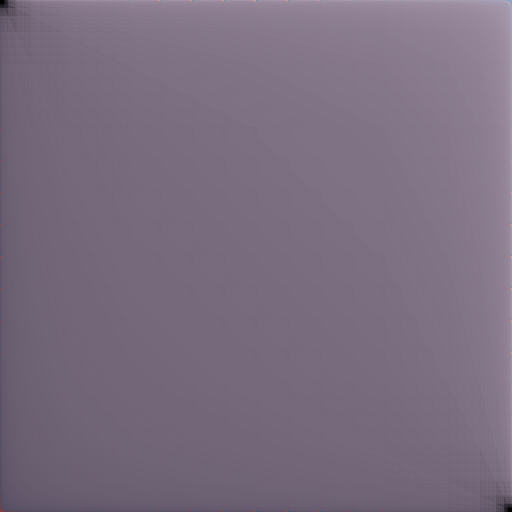};
				\end{axis}
			\end{tikzpicture}
		}
		\subcaptionbox{\texttt{bfloat16}}{
			\begin{tikzpicture}
				\begin{axis}[
					scale only axis,
					width=\textwidth/3,
					height=\textwidth/3,
				]
					\addplot graphics[
						xmin=1.387778780e-17,
						xmax=7.205759403e16,
						ymin=1.387778780e-17,
						ymax=7.205759403e16,
					]{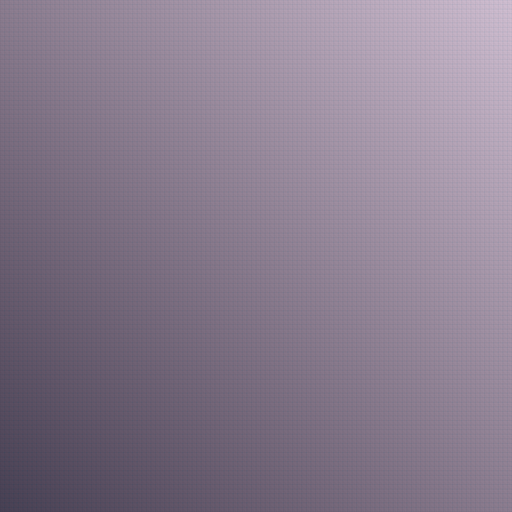};
				\end{axis}
			\end{tikzpicture}
		}
		\subcaptionbox{sorted relative precisions}{
			\begin{tikzpicture}
				\begin{axis}[
					scaled y ticks={base 10:-2},
					ytick={1e0,1e1,1e2},
					yticklabels={0,1,2},
					extra y ticks={500},
					extra y tick labels={$\infty$},
					grid=major,
					ymode=log,
					ylabel={$\eta(e_{\text{rel}})$},
					xmin=0,
					xmax=1,
					scale only axis,
					width=\textwidth/1.5,
					height=\textwidth/2.5,
					extra x ticks={0.02699993,0.03299999,0.3659997},
					extra x tick labels={{\,$2.7\%$,$3.3\%$},,$36.6\%$},
			        extra x tick style={
						ticklabel pos=upper,
					},
					xtick={0,0.25,0.5,0.75,1.0},
	    			xticklabel={\pgfmathparse{\tick*100}\pgfmathprintnumber{\pgfmathresult}\%},
					point meta={x*100},
					legend style={nodes={scale=0.75, transform shape}},
					legend style={at={(0.92,0.82)},anchor=north east},
				]
					\addplot[characteristic,thick] table [col 
					sep=comma]{out/mul-takum16.csv};
					\addlegendentry{\texttt{takum16}};
					\addplot[sign] table [col sep=comma]{out/mul-posit16.csv};
					\addlegendentry{\texttt{posit16}};
					\addplot[direction,densely dashed] table [col sep=comma]{out/mul-bfloat16.csv};
					\addlegendentry{\texttt{bfloat16}};
				\end{axis}
			\end{tikzpicture}
		}
	\end{center}
	\caption{Multiplication closure analysis of \texttt{takum16} versus \texttt{posit16}
		and \texttt{bfloat16}
		within \texttt{posit16}'s positive dynamic range 
		$\left[2^{-56},2^{56}\right] \approx [\num{1.4e-17},\num{7.2e16}]$.}
	\label{fig:closure-16-bit-mul}
\end{figure}
Multiplication in 16 bits (see Figure~\ref{fig:closure-16-bit-mul})
yields comparable results, but an even stronger dichotomy with
$36.6\%$ exact results for \texttt{takum16} versus only
$2.7\%$ and $3.3\%$ exact results for \texttt{bfloat16}
and \texttt{posit16} respectively.
Refer to Figures~\ref{fig:closure-8-bit-div} and
\ref{fig:closure-16-bit-div} in the appendix for 8-
and 16-bit evaluations of division, which yielded
similar results.
\begin{figure}[htbp]
	\begin{center}
		\begin{tikzpicture}
			\begin{axis}[
				scaled y ticks={base 10:-2},
				ytick={1e0,1e1,1e2},
				yticklabels={0,1,2},
				extra y ticks={500},
				extra y tick labels={$\infty$},
				grid=major,
				ymode=log,
				xmin=0,
				xmax=1,
				ylabel={$\eta(e_{\text{rel}})$},
				scale only axis,
				width=\textwidth/1.5,
				height=\textwidth/2.5,
				extra x ticks={0.3015873},
				extra x tick labels={$30.2\%$},
				extra x tick style={
					ticklabel pos=upper,
				},
				xtick={0,0.25,0.5,0.75,1.0},
				xticklabel={\pgfmathparse{\tick*100}\pgfmathprintnumber{\pgfmathresult}\%},
				point meta={x*100},
				legend style={nodes={scale=0.75, transform shape}},
				legend style={at={(0.92,0.82)},anchor=north east},
				]
				\addplot[characteristic,thick] table [col sep=comma]{out/inv-takum8.csv};
				\addlegendentry{\texttt{takum8}};
				\addplot[sign] table [col sep=comma]{out/inv-posit8.csv};
				\addlegendentry{\texttt{posit8}};
			\end{axis}
		\end{tikzpicture}
	\end{center}
	\caption{
		Inversion closure analysis of \texttt{takum8} versus \texttt{posit8}
		within \texttt{posit8}'s positive dynamic range 
		$\left[2^{-24},2^{24}\right] \approx [\num{6.0e-8},\num{1.7e7}]$.
	}
	\label{fig:closure-8-bit-inv}
\end{figure}
\par
As with binary operations, unary operations are also subject to analysis. 
Takums, by design, exhibit closure under inversion (refer to 
Proposition~\ref{prop:takum-inversion}), ensuring that 100\% of inversions 
yield exact results. This property distinguishes takums from posits, which lack 
such closure. For instance, in the case of $n=8$ (as depicted in 
Figure~\ref{fig:closure-8-bit-inv}), it is observed that only a still 
acceptable $30.2\%$ of \texttt{posit8} inversions yield exact results.
\begin{figure}[htbp]
	\begin{center}
		\begin{tikzpicture}
			\begin{axis}[
				scaled y ticks={base 10:-2},
				ytick={1e0,1e1,1e2},
				yticklabels={0,1,2},
				extra y ticks={500},
				extra y tick labels={$\infty$},
				grid=major,
				ymode=log,
				xmin=0,
				xmax=1,
				ylabel={$\eta(e_{\text{rel}})$},
				scale only axis,
				width=\textwidth/1.5,
				height=\textwidth/2.5,
				extra x ticks={0.002929866,0.007812500},
				extra x tick labels={,{$0.3\%$,$0.8\%$}},
				extra x tick style={
					ticklabel pos=upper,
				},
				xtick={0,0.25,0.5,0.75,1.0},
				xticklabel={\pgfmathparse{\tick*100}\pgfmathprintnumber{\pgfmathresult}\%},
				point meta={x*100},
				legend style={nodes={scale=0.75, transform shape}},
				legend style={at={(0.92,0.82)},anchor=north east},
				]
				\addplot[characteristic,thick] table [col sep=comma]{out/inv-takum16.csv};
				\addlegendentry{\texttt{takum16}};
				\addplot[sign] table [col sep=comma]{out/inv-posit16.csv};
				\addlegendentry{\texttt{posit16}};
				\addplot[direction,densely dashed] table [col sep=comma]{out/inv-bfloat16.csv};
				\addlegendentry{\texttt{bfloat16}};
			\end{axis}
		\end{tikzpicture}
	\end{center}
	\caption{
		Inversion closure analysis of \texttt{takum16} versus \texttt{posit16}
		and \texttt{bfloat16} within \texttt{posit16}'s positive dynamic range
		$\left[2^{-56},2^{56}\right] \approx [\num{1.4e-17},\num{7.2e16}]$.
	}
	\label{fig:closure-16-bit-inv}
\end{figure}
Here we observe that extending to $n=16$ bits (see 
Figure~\ref{fig:closure-16-bit-inv}) proves advantageous, averting erroneous 
conclusions, as in regard to the ratio of exact inversions \texttt{posit16} drops to a mere
$0.3\%$ at this juncture. Similar findings apply to \texttt{bfloat16}, with 
inversions exhibiting only $0.8\%$ precision, thereby demonstrating a 
pronounced advantage for takums, particularly evident for larger values of $n$.
\par
\begin{figure}[htbp]
	\begin{center}
		\begin{tikzpicture}
			\begin{axis}[
				scaled y ticks={base 10:-2},
				ytick={1e0,1e1},
				yticklabels={0,1},
				extra y ticks={30},
				extra y tick labels={$\infty$},
				grid=major,
				ymode=log,
				xmin=0,
				xmax=1,
				ylabel={$\eta(e_{\text{rel}})$},
				scale only axis,
				width=\textwidth/1.5,
				height=\textwidth/2.5,
				extra x ticks={0.2063492,0.5875000},
				extra x tick labels={$20.6\%$,$58.8\%$},
				extra x tick style={
					ticklabel pos=upper,
				},
				xtick={0,0.25,0.5,0.75,1.0},
				xticklabel={\pgfmathparse{\tick*100}\pgfmathprintnumber{\pgfmathresult}\%},
				point meta={x*100},
				legend style={nodes={scale=0.75, transform shape}},
				legend style={at={(0.08,0.18)},anchor=south west},
				]
				\addplot[characteristic,thick] table [col sep=comma]{out/sqrt-takum8.csv};
				\addlegendentry{\texttt{takum8}};
				\addplot[sign] table [col sep=comma]{out/sqrt-posit8.csv};
				\addlegendentry{\texttt{posit8}};
			\end{axis}
		\end{tikzpicture}
	\end{center}
	\caption{
		Square root closure analysis of \texttt{takum8} versus \texttt{posit8}
		within \texttt{posit8}'s positive dynamic range 
		$\left[2^{-24},2^{24}\right] \approx [\num{6.0e-8},\num{1.7e7}]$.
	}
	\label{fig:closure-8-bit-sqrt}
\end{figure}
Another ubiquitous unary operation is the calculation of the square root, as 
illustrated for $n=8$ in Figure~\ref{fig:closure-8-bit-sqrt}. With 
\texttt{takum8}, approximately $58.8\%$ of square roots yield exact results, 
contrasting sharply with the mere $20.6\%$ precision achieved by 
\texttt{posit8}.
\begin{figure}[htbp]
	\begin{center}
		\begin{tikzpicture}
			\begin{axis}[
				scaled y ticks={base 10:-2},
				ytick={1e0,1e1,1e2},
				yticklabels={0,1,2},
				extra y ticks={500},
				extra y tick labels={$\infty$},
				grid=major,
				ymode=log,
				xmin=0,
				xmax=1,
				ylabel={$\eta(e_{\text{rel}})$},
				scale only axis,
				width=\textwidth/1.5,
				height=\textwidth/2.5,
				extra x ticks={0.01269609,0.03125000,0.8409001},
				extra x tick labels={,{$1.3\%$,$3.1\%$\ },$84.1\%$},
				extra x tick style={
					ticklabel pos=upper,
				},
				xtick={0,0.25,0.5,0.75,1.0},
				xticklabel={\pgfmathparse{\tick*100}\pgfmathprintnumber{\pgfmathresult}\%},
				point meta={x*100},
				legend style={nodes={scale=0.75, transform shape}},
				legend style={at={(0.7,0.7)},anchor=north east},
				]
				\addplot[characteristic,thick] table [col sep=comma]{out/sqrt-takum16.csv};
				\addlegendentry{\texttt{takum16}};
				\addplot[sign] table [col sep=comma]{out/sqrt-posit16.csv};
				\addlegendentry{\texttt{posit16}};
				\addplot[direction,densely dashed] table [col 
				sep=comma]{out/sqrt-bfloat16.csv};
				\addlegendentry{\texttt{bfloat16}};
			\end{axis}
		\end{tikzpicture}
	\end{center}
	\caption{
		Square root closure analysis of \texttt{takum16} versus \texttt{posit16}
		and \texttt{bfloat16} within \texttt{posit16}'s positive dynamic range
		$\left[2^{-56},2^{56}\right] \approx [\num{1.4e-17},\num{7.2e16}]$.
	}
	\label{fig:closure-16-bit-sqrt}
\end{figure}
Even though this may appear to present a favourable ratio for \texttt{posit8}, 
upon scrutinising $n=16$ bits (refer to Figure~\ref{fig:closure-16-bit-sqrt}), 
it becomes evident that the exactness rate of \texttt{posit16} diminishes to 
$1.3\%$, with \texttt{bfloat16} faring only marginally better at $3.1\%$. In 
contrast, \texttt{takum16} achieves a square root exactness rate of $84.1\%$ 
without sacrificing precision in non-exact instances. These promising outcomes 
are further substantiated by the analysis of the squaring operation, as 
depicted in Figures~\ref{fig:closure-16-bit-pow2} and 
\ref{fig:closure-8-bit-pow2} in the appendix.
\par
Overall, it is imperative for the reader to bear in mind that takums offer an
almost $50\%$ broader dynamic range than \texttt{bfloat16} and an even more
substantially broader dynamic range than \texttt{posit8} and \texttt{posit16}.
Despite this, takums 
consistently outperform or match posits and \texttt{bfloat16} in all scenarios except 
for addition and subtraction. The assessment of the advantage afforded by 
takums over posits and \texttt{bfloat16} is thus heavily contingent upon the 
specific application. Notwithstanding the drawbacks in the realm of addition 
and subtraction, the complete closure under inversion and markedly enhanced 
exactness in multiplication, division, square root, and squaring tilt the 
favour towards takums for general-purpose arithmetic.
\FloatBarrier
\section{Conclusion}
Upon expanding the \textsc{Gustafson} criteria with two supplementary factors 
to encompass the desired dynamic range and establish a universally applicable 
one, we embarked on an exhaustive exploration of IEEE 754 floating-point 
numbers and posits, meticulously scrutinising their performance against these 
novel criteria.
\par
The proposed and formally verified takum number format constitutes a
logarithmic bounded-dynamic-range tapered precision number system that fulfils the \textsc{Gustafson} and 
newly proposed dynamic range criteria, while demonstrating significantly enhanced
coding efficiency for very high and very small magnitude numbers, and 
comparable numerical behaviour to posits. In comparison to posits, takums 
exhibit a constant dynamic range for all $n \ge 12$, with the characteristic length 
always bounded to $11$ bits, ensuring $n-12$ mantissa bits in all cases. These 
improvements are attributed to the absence of reliance on prefix codes, the 
bit-optimal leveraging of the unique properties of the number $255$, and 
efficient utilisation of implicit bits for the characteristic bits. In essence,
choosing a takum bit string length beyond $12$ bits is solely driven by the
pursuit of desired accuracy, rather than being a joint consideration of
desired accuracy and dynamic range, as is the case with all other existing 
number formats.
Based on the tapered format we have explored, we proposed the initiation of a potential new field termed \enquote{tapered precision numerical analysis}
as opposed to the established \enquote{uniform precision numerical analysis}.
\par
Regarding IEEE 754 floating-point numbers, the evaluation reveals that takums 
surpass the insufficient dynamic range of the $16$ and $32$ bit formats, 
offering significantly more precision for small exponents and only marginally 
reduced precision for larger exponents. Takums outperform \texttt{float64} by 
providing an overall lower relative approximation error and offering up to almost 8
binary orders of magnitude less relative approximation error.
A hardware implementation of takums benefits from the fact that the total characteristic bit count 
never exceeds 11 bits, allowing genuine flexibility in precision and sharing of 
components for all precisions. Additionally, the lossless convertibility from 
64-bit IEEE 754 floating-point numbers to \texttt{takum64} provides an 
efficient pathway for transitioning towards a true mixed-precision workflow in 
high-performance computing. The selection of base $\euler$, when coupled with the constrained dynamic range, unveils novel opportunities for hardware arithmetic implementations.
\par
In essence, we have beaten posits at their own game. While upholding all of 
\textsc{Gustafson}'s criteria, we have matched or even exceeded the original 
format in nearly all aspects, ultimately presenting a format that can be used 
both for low-precision applications and general-purpose computing.
\FloatBarrier
\newpage
\begin{appendix}
\section{Supplementary Closure Analysis}\label{sec:closure_appendix}
\begin{figure}[htbp]
	\begin{center}
		\subcaptionbox{\texttt{takum8}}{
			\begin{tikzpicture}
				\begin{axis}[
					scale only axis,
					width=\textwidth/3,
					height=\textwidth/3,
					]
					\addplot graphics[
					xmin=5.96046448e-8,
					xmax=1.67772216e+07,
					ymin=5.96046448e-8,
					ymax=1.67772216e+07,
					]{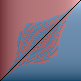};
				\end{axis}
			\end{tikzpicture}
		}
		\subcaptionbox{\texttt{posit8}}{
			\begin{tikzpicture}
				\begin{axis}[
					scale only axis,
					width=\textwidth/3,
					height=\textwidth/3,
					]
					\addplot graphics[
					xmin=5.96046448e-8,
					xmax=1.67772216e+07,
					ymin=5.96046448e-8,
					ymax=1.67772216e+07,
					]{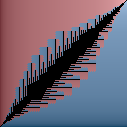};
				\end{axis}
			\end{tikzpicture}
		}
		\subcaptionbox{sorted relative precisions}{
			\begin{tikzpicture}
				\begin{axis}[
					scaled y ticks={base 10:-2},
					ytick={1e0,1e1,1e2},
					yticklabels={0,1,2},
					extra y ticks={500},
					extra y tick labels={$\infty$},
					grid=major,
					ymode=log,
					ylabel={$\eta(e_{\text{rel}})$},
					xmin=0,
					xmax=1,
					scale only axis,
					width=\textwidth/1.5,
					height=\textwidth/2.5,
					extra x ticks={0.0189024,0.1587302},
					extra x tick labels={$1.9\%$,$15.9\%$},
					extra x tick style={
						ticklabel pos=upper,
					},
					xtick={0,0.25,0.5,0.75,1.0},
					xticklabel={\pgfmathparse{\tick*100}\pgfmathprintnumber{\pgfmathresult}\%},
					point meta={x*100},
					legend style={nodes={scale=0.75, transform shape}},
					legend style={at={(0.92,0.82)},anchor=north east},
					]
					\addplot[characteristic,thick] table
						[col sep=comma]{out/sub-takum8.csv};
					\addlegendentry{\texttt{takum8}};
					\addplot[sign] table [col sep=comma]{out/sub-posit8.csv};
					\addlegendentry{\texttt{posit8}};
				\end{axis}
			\end{tikzpicture}
		}
	\end{center}
	\caption{
		Subtraction closure analysis of \texttt{takum8} versus \texttt{posit8}
		within \texttt{posit8}'s positive dynamic range 
		$\left[2^{-24},2^{24}\right] \approx [\num{6.0e-8},\num{1.7e7}]$.
	}
	\label{fig:closure-8-bit-sub}
\end{figure}
\begin{figure}[htbp]
	\begin{center}
		\subcaptionbox{\texttt{takum16}}{
			\begin{tikzpicture}
				\begin{axis}[
					scale only axis,
					width=\textwidth/3,
					height=\textwidth/3,
				]
					\addplot graphics[
						xmin=1.387778780e-17,
						xmax=7.205759403e16,
						ymin=1.387778780e-17,
						ymax=7.205759403e16,
					]{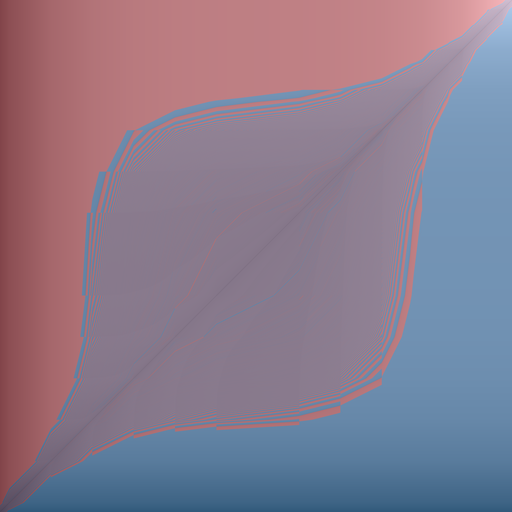};
				\end{axis}
			\end{tikzpicture}
		}
		\subcaptionbox{\texttt{posit16}}{
			\begin{tikzpicture}
				\begin{axis}[
					scale only axis,
					width=\textwidth/3,
					height=\textwidth/3,
				]
					\addplot graphics[
						xmin=1.387778780e-17,
						xmax=7.205759403e16,
						ymin=1.387778780e-17,
						ymax=7.205759403e16,
					]{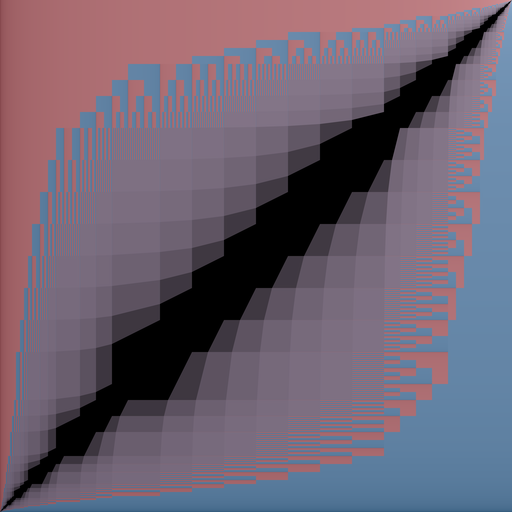};
				\end{axis}
			\end{tikzpicture}
		}
		\subcaptionbox{\texttt{bfloat16}}{
			\begin{tikzpicture}
				\begin{axis}[
					scale only axis,
					width=\textwidth/3,
					height=\textwidth/3,
				]
					\addplot graphics[
						xmin=1.387778780e-17,
						xmax=7.205759403e16,
						ymin=1.387778780e-17,
						ymax=7.205759403e16,
					]{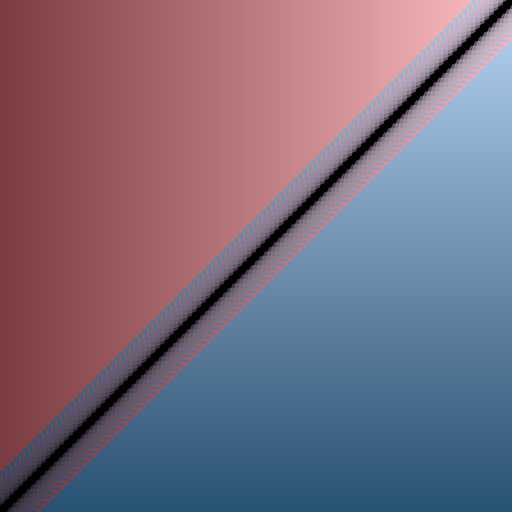};
				\end{axis}
			\end{tikzpicture}
		}
		\subcaptionbox{sorted relative precisions}{
			\begin{tikzpicture}
				\begin{axis}[
					scaled y ticks={base 10:-2},
					ytick={1e0,1e1,1e2},
					yticklabels={0,1,2},
					extra y ticks={500},
					extra y tick labels={$\infty$},
					grid=major,
					ymode=log,
					ylabel={$\eta(e_{\text{rel}})$},
					xmin=0,
					xmax=1,
					scale only axis,
					width=\textwidth/1.5,
					height=\textwidth/2.5,
					extra x ticks={0.0,0.03499990,0.1670000},
					extra x tick labels=
						{{\ \ \ $0.0\%$,$3.5\%$},,$16.7\%$},
			        extra x tick style={
						ticklabel pos=upper,
					},
					xtick={0,0.25,0.5,0.75,1.0},
	    			xticklabel={\pgfmathparse{\tick*100}\pgfmathprintnumber{\pgfmathresult}\%},
					point meta={x*100},
					legend style={nodes={scale=0.75, transform shape}},
					legend style={at={(0.92,0.82)},anchor=north east},
				]
					\addplot[characteristic,thick] table [col 
					sep=comma]{out/sub-takum16.csv};
					\addlegendentry{\texttt{takum16}};
					\addplot[sign] table [col sep=comma]{out/sub-posit16.csv};
					\addlegendentry{\texttt{posit16}};
					\addplot[direction,densely dashed] table [col sep=comma]{out/sub-bfloat16.csv};
					\addlegendentry{\texttt{bfloat16}};
				\end{axis}
			\end{tikzpicture}
		}
	\end{center}
	\caption{Subtraction closure analysis of \texttt{takum16} versus \texttt{posit16}
		and \texttt{bfloat16}
		within \texttt{posit16}'s positive dynamic range 
		$\left[2^{-56},2^{56}\right] \approx [\num{1.4e-17},\num{7.2e16}]$.}
	\label{fig:closure-16-bit-sub}
\end{figure}
\begin{figure}[htbp]
	\begin{center}
		\subcaptionbox{\texttt{takum8}}{
			\begin{tikzpicture}
				\begin{axis}[
					scale only axis,
					width=\textwidth/3,
					height=\textwidth/3,
				]
					\addplot graphics[
						xmin=5.96046448e-8,
						xmax=1.67772216e+07,
						ymin=5.96046448e-8,
						ymax=1.67772216e+07,
					]{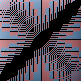};
				\end{axis}
			\end{tikzpicture}
		}
		\subcaptionbox{\texttt{posit8}}{
			\begin{tikzpicture}
				\begin{axis}[
					scale only axis,
					width=\textwidth/3,
					height=\textwidth/3,
				]
					\addplot graphics[
						xmin=5.96046448e-8,
						xmax=1.67772216e+07,
						ymin=5.96046448e-8,
						ymax=1.67772216e+07,
					]{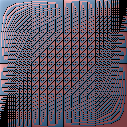};
				\end{axis}
			\end{tikzpicture}
		}
		\subcaptionbox{sorted relative precisions}{
			\begin{tikzpicture}
				\begin{axis}[
					scaled y ticks={base 10:-2},
					ytick={1e0,1e1},
					yticklabels={0,1},
					extra y ticks={30},
					extra y tick labels={$\infty$},
					grid=major,
					ymode=log,
					xmin=0,
					xmax=1,
					ylabel={$\eta(e_{\text{rel}})$},
					scale only axis,
					width=\textwidth/1.5,
					height=\textwidth/2.5,
					extra x ticks={0.2559524,0.4033537},
					extra x tick labels={$25.6\%$,$40.3\%$},
					extra x tick style={
						ticklabel pos=upper,
					},
					xtick={0,0.25,0.5,0.75,1.0},
					xticklabel={\pgfmathparse{\tick*100}\pgfmathprintnumber{\pgfmathresult}\%},
					point meta={x*100},
					legend style={nodes={scale=0.75, transform shape}},
					legend style={at={(0.92,0.82)},anchor=north east},
					]
					\addplot[characteristic,thick] table [col sep=comma]{out/div-takum8.csv};
					\addlegendentry{\texttt{takum8}};
					\addplot[sign] table [col sep=comma]{out/div-posit8.csv};
					\addlegendentry{\texttt{posit8}};
				\end{axis}
			\end{tikzpicture}
		}
	\end{center}
	\caption{
		Division closure analysis of \texttt{takum8} versus \texttt{posit8}
		within \texttt{posit8}'s positive dynamic range 
		$\left[2^{-24},2^{24}\right] \approx [\num{6.0e-8},\num{1.7e7}]$.
	}
	\label{fig:closure-8-bit-div}
\end{figure}
\begin{figure}[htbp]
	\begin{center}
		\subcaptionbox{\texttt{takum16}}{
			\begin{tikzpicture}
				\begin{axis}[
					scale only axis,
					width=\textwidth/3,
					height=\textwidth/3,
				]
					\addplot graphics[
						xmin=1.387778780e-17,
						xmax=7.205759403e16,
						ymin=1.387778780e-17,
						ymax=7.205759403e16,
					]{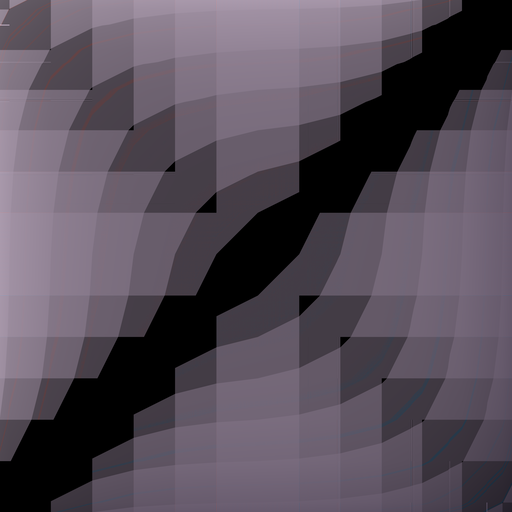};
				\end{axis}
			\end{tikzpicture}
		}
		\subcaptionbox{\texttt{posit16}}{
			\begin{tikzpicture}
				\begin{axis}[
					scale only axis,
					width=\textwidth/3,
					height=\textwidth/3,
				]
					\addplot graphics[
						xmin=1.387778780e-17,
						xmax=7.205759403e16,
						ymin=1.387778780e-17,
						ymax=7.205759403e16,
					]{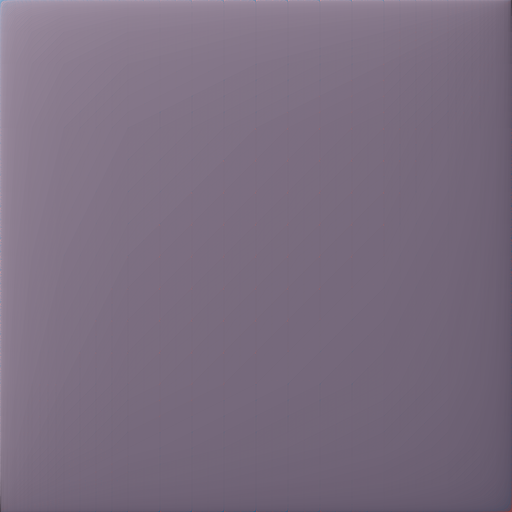};
				\end{axis}
			\end{tikzpicture}
k		}
		\subcaptionbox{\texttt{bfloat16}}{
			\begin{tikzpicture}
				\begin{axis}[
					scale only axis,
					width=\textwidth/3,
					height=\textwidth/3,
				]
					\addplot graphics[
						xmin=1.387778780e-17,
						xmax=7.205759403e16,
						ymin=1.387778780e-17,
						ymax=7.205759403e16,
					]{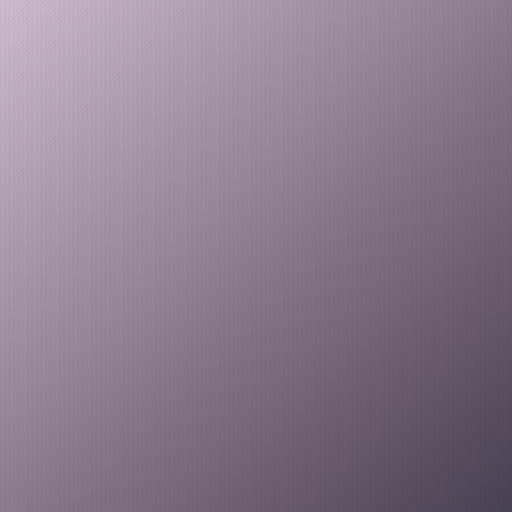};
				\end{axis}
			\end{tikzpicture}
		}
		\subcaptionbox{sorted relative precisions}{
			\begin{tikzpicture}
				\begin{axis}[
					scaled y ticks={base 10:-2},
					ytick={1e0,1e1,1e2},
					yticklabels={0,1,2},
					extra y ticks={500},
					extra y tick labels={$\infty$},
					grid=major,
					ymode=log,
					ylabel={$\eta(e_{\text{rel}})$},
					xmin=0,
					xmax=1,
					scale only axis,
					width=\textwidth/1.5,
					height=\textwidth/2.5,
					extra x ticks={0.02699993,0.03299999,0.3669997},
					extra x tick labels=
						{{\,$2.7\%$,$3.3\%$},,$36.7\%$},
			        extra x tick style={
						ticklabel pos=upper,
					},
					xtick={0,0.25,0.5,0.75,1.0},
	    			xticklabel={\pgfmathparse{\tick*100}\pgfmathprintnumber{\pgfmathresult}\%},
					point meta={x*100},
					legend style={nodes={scale=0.75, transform shape}},
					legend style={at={(0.92,0.82)},anchor=north east},
				]
					\addplot[characteristic,thick] table [col 
					sep=comma]{out/div-takum16.csv};
					\addlegendentry{\texttt{takum16}};
					\addplot[sign] table [col sep=comma]{out/div-posit16.csv};
					\addlegendentry{\texttt{posit16}};
					\addplot[direction,densely dashed] table [col sep=comma]{out/div-bfloat16.csv};
					\addlegendentry{\texttt{bfloat16}};
				\end{axis}
			\end{tikzpicture}
		}
	\end{center}
	\caption{Division closure analysis of \texttt{takum16} versus \texttt{posit16}
		and \texttt{bfloat16}
		within \texttt{posit16}'s positive dynamic range 
		$\left[2^{-56},2^{56}\right] \approx [\num{1.4e-17},\num{7.2e16}]$.}
	\label{fig:closure-16-bit-div}
\end{figure}
\begin{figure}[htbp]
	\begin{center}
		\begin{tikzpicture}
			\begin{axis}[
				scaled y ticks={base 10:-2},
				ytick={1e0,1e1,1e2},
				yticklabels={0,1,2},
				extra y ticks={500},
				extra y tick labels={$\infty$},
				grid=major,
				ymode=log,
				xmin=0,
				xmax=1,
				ylabel={$\eta(e_{\text{rel}})$},
				scale only axis,
				width=\textwidth/1.5,
				height=\textwidth/2.5,
				extra x ticks={0.2063492,0.6},
				extra x tick labels={$20.6\%$,$60.0\%$},
				extra x tick style={
					ticklabel pos=upper,
				},
				xtick={0,0.25,0.5,0.75,1.0},
				xticklabel={\pgfmathparse{\tick*100}\pgfmathprintnumber{\pgfmathresult}\%},
				point meta={x*100},
				legend style={nodes={scale=0.75, transform shape}},
				legend style={at={(0.08,0.18)},anchor=south west},
				]
				\addplot[characteristic,thick] table [col 
				sep=comma]{out/pow2-takum8.csv};
				\addlegendentry{\texttt{takum8}};
				\addplot[sign] table [col sep=comma]{out/pow2-posit8.csv};
				\addlegendentry{\texttt{posit8}};
			\end{axis}
		\end{tikzpicture}
	\end{center}
	\caption{
		Squaring closure analysis of \texttt{takum8} versus \texttt{posit8}
		within \texttt{posit8}'s positive dynamic range 
		$\left[2^{-24},2^{24}\right] \approx [\num{6.0e-8},\num{1.7e7}]$.
	}
	\label{fig:closure-8-bit-pow2}
\end{figure}
\begin{figure}[htbp]
	\begin{center}
		\begin{tikzpicture}
			\begin{axis}[
				scaled y ticks={base 10:-2},
				ytick={1e0,1e1,1e2},
				yticklabels={0,1,2},
				extra y ticks={500},
				extra y tick labels={$\infty$},
				grid=major,
				ymode=log,
				xmin=0,
				xmax=1,
				ylabel={$\eta(e_{\text{rel}})$},
				scale only axis,
				width=\textwidth/1.5,
				height=\textwidth/2.5,
				extra x ticks={0.01269609,0.06250000},
				extra x tick labels={{\ \ \ \ $1.3\%$,$6.3\%$},},
				extra x tick style={
					ticklabel pos=upper,
				},
				xtick={0,0.25,0.5,0.75,1.0},
				xticklabel={\pgfmathparse{\tick*100}\pgfmathprintnumber{\pgfmathresult}\%},
				point meta={x*100},
				legend style={nodes={scale=0.75, transform shape}},
				legend style={at={(0.92,0.82)},anchor=north east},
				]
				\addplot[characteristic,thick] table [col 
				sep=comma]{out/pow2-takum16.csv};
				\addlegendentry{\texttt{takum16}};
				\addplot[sign] table [col sep=comma]{out/pow2-posit16.csv};
				\addlegendentry{\texttt{posit16}};
				\addplot[direction,densely dashed] table [col sep=comma]{out/pow2-bfloat16.csv};
				\addlegendentry{\texttt{bfloat16}};
			\end{axis}
		\end{tikzpicture}
	\end{center}
	\caption{
		Squaring closure analysis of \texttt{takum16} versus \texttt{posit16}
		and \texttt{bfloat16} within \texttt{posit16}'s positive dynamic range
		$\left[2^{-56},2^{56}\right] \approx [\num{1.4e-17},\num{7.2e16}]$.
	}
	\label{fig:closure-16-bit-pow2}
\end{figure}
\FloatBarrier
\section{Proof of Proposition~\ref{prop:ieee-excess}: IEEE Waste Ratio}
\label{sec:proof-ieee-excess}
The number of wasted $\mathrm{NaN}$ bit strings is given by $2^{1+n_f}-3$. 
This is because we vary the sign bit and all fraction bits but subtract the two 
cases for $\pm\infty$ (where all fraction bits are zero) and a single 
$\mathrm{NaN}$ representation that are not redundant.
\par
It holds $\log_2\!\left(\euler^{\pm255}\right) \approx \pm 183.94$. To keep the following derivation
simple and not put IEEE 754 floating-point numbers at a disadvantage, we assume 
a simplified slightly 
larger desired 
dynamic range of $\pm\left(2^{-184},2^{184}\right)$. There are $2^{n_f+1}$ 
subnormal numbers, which are considered to be in excess when the format flushes 
subnormals to zero or when the subnormals are outside the dynamic range 
(i.e., when the subnormal exponent $-2^{n_e-1}+1$ is smaller than $-184$, 
which is only the case when $n_e \ge 9$).
\par
Left to quantify are the normal bit representations exceeding $2^{-184}$ 
and $2^{184}$, respectively. 
The smallest normal exponent is $-2^{n_e-1}+2$, 
meaning there are $\max(0,-184 - (-2^{n_e-1}+2)) = \max(0,2^{n_e-1} - 186)$ 
\enquote{excessive} exponents at the lower end of the dynamic range. 
Analogously, 
the largest exponent is $2^{n_e-1}-1$, and 
there are $\max(0,2^{n_e-1}-1-183)=\max(0,2^{n_e-1}-184)$ \enquote{excessive} 
exponents at the higher end of the 
dynamic range (we only go up to exponent $183$ as the significand contained in
$[1,2)$ asymptotically realizes the full quantity $2^{184}$). For each 
excessive exponent,
we vary the sign and fraction bits, 
adding up to $2^{1+n_f}$ combinations each. Additionally, the maxima can be 
omitted when $n_e \ge 9$. In total, we obtain the total number of excessive
normal 
bit strings exceeding the dynamic range as $(2^{n_e}-370) \cdot 2^{1+n_f}$ 
for 
$n_e \ge 9$. The complete expression for all excessive bit strings is:
\begin{equation}
	\left[1 + (n_e \ge 9) \cdot (2^{n_e} - 370) + (n_e \ge 9 \lor \text{no 
		subnormals}) \right] \cdot 2^{1+n_f}-3.
\end{equation}
For the ratio, we divide this expression by the number of bit strings,
$2^{1+n_e+n_f}$, and obtain
\begin{multline}
	\left[1 + (n_e \ge 9) \cdot (2^{n_e} - 370) + (n_e \ge 9 \lor \text{no 
		subnormals}) \right] \cdot 2^{-n_e}-\\
	3 \cdot 2^{-1-n_e-n_f},
\end{multline}
which was to be shown.\qed
\section{Proof of Proposition~\ref{prop:posit-excess}: Posit Waste Ratio}
\label{sec:proof-posit-excess}
It holds $\log_2\!\left(\euler^{\pm255}\right) = \pm 183.94$. To keep the following derivation
simple and not put posits at a disadvantage, we assume a simplified slightly 
larger desired dynamic range of $\pm\left(2^{-184},2^{184}\right)$. 
Without 
loss of generality, owing to the symmetry, we can easily determine the number 
of excessive bit representations for $\textcolor{p-sign}{S}=0$ and 
$\textcolor{p-regime}{R}_0 = 1$, denoting the quadrant with $p \ge 1$, and 
multiply it by four. The initial excessive $n$-bit posit value, $2^{184}$, with 
$r=46$ and thus $k=47$, results in the binary representation 
$\textcolor{p-sign}{0}\textcolor{p-regime}{1\dots1}\textcolor{p-regime-term}{0}\dots$
(the regime consists of 47 ones, and the regime termination bit is a ghost bit 
for $n=48$). Leveraging the monotonicity property, we only need to ascertain 
the number of representations between this bit string and 
$\textcolor{p-sign}{0}\textcolor{p-regime}{1\dots1}$, which corresponds to the 
degrees of freedom of $n-48$ bits. This amounts to $2^{n-48}$, multiplied by 
four to calculate the total number of excessive representations. The ratio is 
determined by dividing this by the total number of posits:
\begin{equation}
		\begin{cases}
			0 & n \le 47\\
			\frac{4 \cdot 2^{n-48}}{2^n} & n \ge 48
		\end{cases} =
		\begin{cases}
			0 & n \le 47\\
			2^{-46} & n \ge 48
		\end{cases} \approx
		\begin{cases}
			0 & n \le 47\\
			\num{1.42e-14} & n \ge 48.
		\end{cases}
\end{equation}
This way we have proven the proposition.\qed
\section{Proof of Proposition~\ref{prop:takum-uniqueness}: Takum Uniqueness}
\label{sec:proof-takum-uniqueness}
Let $B =: (\textcolor{sign}{S},\textcolor{direction}{D},\textcolor{regime}{R},
\textcolor{characteristic}{C},\textcolor{mantissa}{M})$ and
$\tilde{B} =: 
\left(\textcolor{sign}{\tilde{S}},\textcolor{direction}{\tilde{D}},
\textcolor{regime}{\tilde{R}},
\textcolor{characteristic}{\tilde{C}},\textcolor{mantissa}{\tilde{M}}\right)$.
We can see that the special cases 
$\takum((\textcolor{sign}{S},\textcolor{direction}{D},\textcolor{regime}{R},
\textcolor{characteristic}{C},\textcolor{mantissa}{M})) =
\takum\!\left(\left(\textcolor{sign}{\tilde{S}},\textcolor{direction}{\tilde{D}},
\textcolor{regime}{\tilde{R}},
\textcolor{characteristic}{\tilde{C}},\textcolor{mantissa}{\tilde{M}}\right)\right) = \mathrm{NaR}$
and $\takum((\textcolor{sign}{S},\textcolor{direction}{D},\textcolor{regime}{R},
\textcolor{characteristic}{C},\textcolor{mantissa}{M})) =
\takum\!\left(\left(\textcolor{sign}{\tilde{S}},\textcolor{direction}{\tilde{D}},
\textcolor{regime}{\tilde{R}},
\textcolor{characteristic}{\tilde{C}},\textcolor{mantissa}{\tilde{M}}\right)\right) = 0$ immediately
imply $(\textcolor{sign}{S},\textcolor{direction}{D},\textcolor{regime}{R},
\textcolor{characteristic}{C},\textcolor{mantissa}{M}) =
\left(\textcolor{sign}{\tilde{S}},\textcolor{direction}{\tilde{D}},
\textcolor{regime}{\tilde{R}},
\textcolor{characteristic}{\tilde{C}},\textcolor{mantissa}{\tilde{M}}\right)$ because
they are distinctively defined.
We distinguish all variables of
Definition~\ref{def:takum} in regard to
$\takum((\textcolor{sign}{\tilde{S}},\textcolor{direction}{\tilde{D}},
\textcolor{regime}{\tilde{R}},
\textcolor{characteristic}{\tilde{C}},\textcolor{mantissa}{\tilde{M}}))$ with a
tilde and investigate the last remaining case
\begin{equation}
	{(-1)}^{\textcolor{sign}{S}} \euler^\ell =
	\takum((\textcolor{sign}{S},\textcolor{direction}{D},\textcolor{regime}{R},
	\textcolor{characteristic}{C},\textcolor{mantissa}{M})) =
	\takum\!\left(\left(\textcolor{sign}{\tilde{S}},\textcolor{direction}{\tilde{D}},
	\textcolor{regime}{\tilde{R}},
	\textcolor{characteristic}{\tilde{C}},\textcolor{mantissa}{\tilde{M}}\right)\right) =
	{(-1)}^{\textcolor{sign}{\tilde{S}}} \euler^{\tilde{\ell}}.
\end{equation}
This immediately implies $\textcolor{sign}{S} = \textcolor{sign}{\tilde{S}}$
and subsequently $\ell = \tilde{\ell}$, which in turn,
using (\ref{eq:takum-logarithmic}) and $\textcolor{sign}{S} = \textcolor{sign}{\tilde{S}}$,
is equivalent to
\begin{equation}
	c + m = \tilde{c} + \tilde{m}.
\end{equation}
Given $c$ and $\tilde{c}$ are integers and
$m,\tilde{m} \in [0,1)$, it directly follows that
$c = \tilde{c}$ and $m=\tilde{m}$.
For $r,\tilde{r} \in \{0,7\}$ the characteristics $c,\tilde{c}$
are contained in the distinct sets
$\{-1\}$, $\{ -3,-2 \}$,
$\{ -7,\dots,-4 \}$, $\{ -15,\dots,-8 \}$, $\{ -31,\dots,-16 \}$,
$\{ -63,\dots,-32 \}$, $\{ -127,\dots,-64 \}$ and
$\{ -255,\dots,-128 \}$ for $\textcolor{direction}{D} = 0$.
Likewise the characteristics $c,\tilde{c}$ are contained in the
distinct sets $\{0\}$, $\{ 1,2 \}$,
$\{ 3,\dots,6 \}$, $\{ 7,\dots,14 \}$, $\{ 15,\dots,30 \}$,
$\{ 31,\dots,62 \}$, $\{ 63,\dots,126 \}$ and
$\{ 127,\dots,254 \}$ for $\textcolor{direction}{D} = 1$.
This means that for the equality
$c=\tilde{c}$ to be satisfied $r=\tilde{r}$ and
$\textcolor{direction}{D} = \textcolor{direction}{\tilde{D}}$
must hold, and thus $\textcolor{regime}{R} = \textcolor{regime}{\tilde{R}}$
and $\textcolor{characteristic}{C} = \textcolor{characteristic}{\tilde{C}}$.
With $r=\tilde{r}$ we can also conclude $p=\tilde{p}$ and
thus, with $m=\tilde{m}$, deduce the equality
$\textcolor{mantissa}{M} = \textcolor{mantissa}{\tilde{M}}$.\qed
\section{Proof of Proposition~\ref{prop:takum-monotonicity}: Takum Monotonicity}
\label{sec:proof-takum-monotonicity}
Given $B,\tilde{B} \neq (1,0,\dots,0)$ it follows that 
$\takum(B),\takum\!\left(\tilde{B}\right) \neq \mathrm{NaR}$, which in turn means
that $\takum(B),\takum\!\left(\tilde{B}\right) \in \mathbb{R}$. Given 
$(\mathbb{R},\le)$ is a partially ordered set, the proposition is 
well-defined.
\par
We prove the proposition by proving the equivalent statement
\begin{equation}
	\takum(B) < \takum(B+1)
\end{equation}
for $B,B+1 \neq (1,0,\dots,0)$. Let $B =:
(\textcolor{sign}{S},\textcolor{direction}{D},\textcolor{regime}{R},
\textcolor{characteristic}{C},\textcolor{mantissa}{M})$ and
$B + 1 =:  
\left(\textcolor{sign}{\tilde{S}},\textcolor{direction}{\tilde{D}},
\textcolor{regime}{\tilde{R}},
\textcolor{characteristic}{\tilde{C}},\textcolor{mantissa}{\tilde{M}}\right)$.
We add a tilde to all corresponding variables of Definition~\ref{def:takum}
in all subsequent cases.
A primary challenge encountered in this proof is the phenomenon whereby incrementing a bit string may precipitate carries that propagate across the full width of the string. Consequently, our strategy entails a detailed examination of all conceivable scenarios wherein segments towards the string's end become sequentially saturated. This approach enables us to identify the specific portions of the bit string that are impacted, thereby facilitating a comprehensive proof of monotonicity.
\par
We employ the notation $\bm{1}$ to denote a bit string comprised entirely of ones, a convention that is slightly extended to accommodate the inclusion of empty bit strings in the context of the characteristic and fraction bits $\textcolor{characteristic}{C}$ and $\textcolor{mantissa}{M}$. This allowance is justified on the grounds that an empty bit string effectively mirrors the external behaviour of a fully ones-populated string: An incrementation transforms an empty string of all ones to an equally empty string of all zeros, seamlessly propagating any incoming carries. Notably, such a configuration precludes the activation of case 5 by an empty $\textcolor{characteristic}{C}$, and similarly, case 6 remains untriggered by an empty $\textcolor{mantissa}{M}$, given that an empty bit string equates to an empty string of all ones.
\begin{description}
	\item[Case 1 ($\textcolor{sign}{S} = \textcolor{direction}{D} = 1,
	\textcolor{regime}{R} = \textcolor{characteristic}{C} =
	\textcolor{mantissa}{M} = \bm{1}$):]{
		It holds
		\begin{align}
			\takum\!\left((\textcolor{sign}{S},\textcolor{direction}{D},\textcolor{regime}{R},
			\textcolor{characteristic}{C},\textcolor{mantissa}{M})\right)
			&= \takum\!\left((\textcolor{sign}{1},
			\textcolor{direction}{1},
			\textcolor{regime}{\bm{1}},
			\textcolor{characteristic}{\bm{1}},
			\textcolor{mantissa}{\bm{1}})\right)= -\euler^\ell < 0,\\
			\takum\!\left(\left(\textcolor{sign}{\tilde{S}},\textcolor{direction}{\tilde{D}},
			\textcolor{regime}{\tilde{R}},
			\textcolor{characteristic}{\tilde{C}},\textcolor{mantissa}{\tilde{M}}\right)\right)
			&= \takum\!\left((\textcolor{sign}{0},
			\textcolor{direction}{0},
			\textcolor{regime}{\bm{0}},
			\textcolor{characteristic}{\bm{0}},
			\textcolor{mantissa}{\bm{0}})\right) = 0,
		\end{align}
		and thus 
		$\takum\!\left((\textcolor{sign}{S},\textcolor{direction}{D},\textcolor{regime}{R},
		\textcolor{characteristic}{C},\textcolor{mantissa}{M})\right) <
		\takum\!\left(\left(\textcolor{sign}{\tilde{S}},\textcolor{direction}{\tilde{D}},
		\textcolor{regime}{\tilde{R}},
		\textcolor{characteristic}{\tilde{C}},\textcolor{mantissa}{\tilde{M}}\right)\right)$.
	}
	\item[Case 2 ($\textcolor{sign}{S} \neq 1, \textcolor{direction}{D} = 
	1, 
	\textcolor{regime}{R} = \textcolor{characteristic}{C} =
	\textcolor{mantissa}{M} = \bm{1}$):]{
		$\textcolor{sign}{S} \neq 1$ directly implies
		$\textcolor{sign}{S} = 0$ and it holds
		\begin{align}
			(\textcolor{sign}{S},\textcolor{direction}{D},\textcolor{regime}{R},
			\textcolor{characteristic}{C},\textcolor{mantissa}{M})
			&= (\textcolor{sign}{0},
			\textcolor{direction}{1},
			\textcolor{regime}{\bm{1}},
			\textcolor{characteristic}{\bm{1}},
			\textcolor{mantissa}{\bm{1}}),\\
			\left(\textcolor{sign}{\tilde{S}},\textcolor{direction}{\tilde{D}},
			\textcolor{regime}{\tilde{R}},
			\textcolor{characteristic}{\tilde{C}},\textcolor{mantissa}{\tilde{M}}\right)
			&= (\textcolor{sign}{1},
			\textcolor{direction}{0},
			\textcolor{regime}{\bm{0}},
			\textcolor{characteristic}{\bm{0}},
			\textcolor{mantissa}{\bm{0}}),
		\end{align}
		however the latter equation is in violation of our precondition,
		as we assumed $B + 1 = 
		\left(\textcolor{sign}{\tilde{S}},\textcolor{direction}{\tilde{D}},
		\textcolor{regime}{\tilde{R}},
		\textcolor{characteristic}{\tilde{C}},\textcolor{mantissa}{\tilde{M}}\right)
		\neq (1,0,\dots,0)$. This means that we can ignore this case.
	}
\end{description}
In order to streamline the analysis of subsequent cases, all of which conform to the condition $\textcolor{sign}{S} = \textcolor{sign}{\tilde{S}}$, we embark on an examination of the requisite criteria to be met when $\textcolor{sign}{S} = \textcolor{sign}{\tilde{S}}$.
Under the additional assumption
$(\textcolor{sign}{S},\textcolor{direction}{D},\textcolor{regime}{R},
\textcolor{characteristic}{C},\textcolor{mantissa}{M})$,
$\left(\textcolor{sign}{\tilde{S}},\textcolor{direction}{\tilde{D}},
\textcolor{regime}{\tilde{R}},
\textcolor{characteristic}{\tilde{C}},\textcolor{mantissa}{\tilde{M}}\right) \neq 
\bm{0}$ it holds
\begin{align}
	\takum(B) < \takum(B+1)
	&\iff {(-1)}^{\textcolor{sign}{S}} \euler^\ell <
	{(-1)}^{\textcolor{sign}{S}} \euler^{\tilde{\ell}}\\
	&\iff \begin{cases}
		\euler^\ell < \euler^{\tilde{\ell}} & \textcolor{sign}{S} = 0\\
		-\euler^\ell < -\euler^{\tilde{\ell}} & \textcolor{sign}{S} = 1
	\end{cases}\\
	&\iff \begin{cases}
		\euler^\ell < \euler^{\tilde{\ell}} & \textcolor{sign}{S} = 0\\
		\euler^\ell > \euler^{\tilde{\ell}} & \textcolor{sign}{S} = 1
	\end{cases}
\end{align}
\begin{align}
	&\iff \begin{cases}
		\ell < \tilde{\ell} & \textcolor{sign}{S} = 0\\
		\ell > \tilde{\ell} & \textcolor{sign}{S} = 1
	\end{cases}\\
	&\iff \begin{cases}
		c+m < \tilde{c}+\tilde{m} & \textcolor{sign}{S} = 0\\
		-(c+m) > -\left( \tilde{c}+\tilde{m} \right)
		& \textcolor{sign}{S} = 1
	\end{cases}\\
	&\iff \begin{cases}
		c+m < \tilde{c}+\tilde{m} & \textcolor{sign}{S} = 0\\
		c+m < \tilde{c}+\tilde{m} & \textcolor{sign}{S} = 1
	\end{cases}\\
	&\iff (c - \tilde{c}) + (m - \tilde{m}) < 0
	\label{eq:takum-proof-monotonicity-goal}
\end{align}
\begin{description}
	\item[Case 3 ($\textcolor{direction}{D} \neq 1,
	\textcolor{regime}{R} = \textcolor{characteristic}{C} =
	\textcolor{mantissa}{M} = \bm{1}$):]{
		$\textcolor{direction}{D} \neq 1$ directly implies
		$\textcolor{direction}{D} = 0$ and it holds
		\begin{align}
			(\textcolor{sign}{S},\textcolor{direction}{D},\textcolor{regime}{R},
			\textcolor{characteristic}{C},\textcolor{mantissa}{M})
			&= (\textcolor{sign}{S},
			\textcolor{direction}{0},
			\textcolor{regime}{\bm{1}},
			\textcolor{characteristic}{\bm{1}},
			\textcolor{mantissa}{\bm{1}}),\\
			\left(\textcolor{sign}{\tilde{S}},\textcolor{direction}{\tilde{D}},
			\textcolor{regime}{\tilde{R}},
			\textcolor{characteristic}{\tilde{C}},\textcolor{mantissa}{\tilde{M}}\right)
			&= (\textcolor{sign}{S},
			\textcolor{direction}{1},
			\textcolor{regime}{\bm{0}},
			\textcolor{characteristic}{\bm{0}},
			\textcolor{mantissa}{\bm{0}}),
		\end{align}
		which yields $r=\tilde{r}=0$, $c=-1$, $\tilde{c}=0$, $m \in [0,1)$ and $\tilde{m}=0$. Inserting these results into
		(\ref{eq:takum-proof-monotonicity-goal}) we obtain
		\begin{equation}
			(c - \tilde{c}) + (m - \tilde{m}) < 0
			\iff
			-1 + m < 0
			\iff
			m < 1,
		\end{equation}
		which is satisfied as $f \in [0,1)$.
	}
	\item[Case 4 ($\textcolor{regime}{R} \neq \bm{1},
	\textcolor{characteristic}{C} = \textcolor{mantissa}{M} = \bm{1}$):]{
		It holds
		\begin{align}
			(\textcolor{sign}{S},\textcolor{direction}{D},\textcolor{regime}{R},
			\textcolor{characteristic}{C},\textcolor{mantissa}{M})
			&= (\textcolor{sign}{S},
			\textcolor{direction}{D},
			\textcolor{regime}{R},
			\textcolor{characteristic}{\bm{1}},
			\textcolor{mantissa}{\bm{1}}),\\
			\left(\textcolor{sign}{\tilde{S}},\textcolor{direction}{\tilde{D}},
			\textcolor{regime}{\tilde{R}},
			\textcolor{characteristic}{\tilde{C}},\textcolor{mantissa}{\tilde{M}}\right)
			&= (\textcolor{sign}{S},
			\textcolor{direction}{D},
			\textcolor{regime}{R} + 1,
			\textcolor{characteristic}{\bm{0}},
			\textcolor{mantissa}{\bm{0}}),
		\end{align}
		where it shall be noted that a change in the regime leads to a shift
		of the characteristic and mantissa bits. Nevertheless, 
		given that both segments remain at zero, this shift is rendered 
		inconsequential.
		\begin{description}
			\item[Case 4a ($\textcolor{direction}{D} = 0$):]{
				It follows $r \in \{1,\dots,7\}$ and $\tilde{r} = r - 1$
				given $\textcolor{regime}{R} \neq 1$. The characteristic bits
				$\textcolor{characteristic}{C}$ are saturated, which yields
				$c = -2^{r+1} + 1 + \left( 2^{r} - 1\right) = -2^r$. In 
				contrast, the characteristic bits 
				$\textcolor{characteristic}{\tilde{C}}$ are zero, which yields
				$\tilde{c} = -2^{\tilde{r}+1} +1 + 0 = -2^{(r-1)+1} + 1 + 0 = -2^{r} + 1$.
				The mantissas $m$ and $\tilde{m}$ follow as
				$m \in [0,1)$ and $\tilde{m} = 0$ due to the respectively 
				saturated and zero mantissa bits. It follows
				\begin{equation}
					c - \tilde{c} = -2^r + 2^{r} - 1 = -1,
				\end{equation}
				which is inserted into 
				(\ref{eq:takum-proof-monotonicity-goal}):
				\begin{equation}
					(c - \tilde{c}) + (m - \tilde{m}) < 0
					\iff
					-1 + m < 0
					\iff
					m < 1.
				\end{equation}
				The inequality is satisfied as $m \in [0,1)$.
			}
			\item[Case 4b ($\textcolor{direction}{D} = 1$):]{
				It follows $r \in \{0,\dots,6\}$ and $\tilde{r} = r + 1$
				given $\textcolor{regime}{R} \neq 1$. The characteristic bits
				$\textcolor{characteristic}{C}$ are saturated, which yields
				$c = 2^r - 1 + \left( 2^{r} - 1\right) = 2^{r+1}-2$. In
				contrast, the characteristic bits
				$\textcolor{characteristic}{\tilde{C}}$ are zero, which yields
				$\tilde{c} = 2^{\tilde{r}} - 1 + 0 = 2^{r+1} - 1 + 0 = 2^{r+1} - 1$.
				The mantissas $m$ and $\tilde{m}$ follow as
				$m \in [0,1)$ and $\tilde{m} = 0$ due to the respectively
				saturated and zero mantissa bits. It follows
				\begin{equation}
					c - \tilde{c} = 2^{r+1}-2 - 2^{r+1} + 1 = -1
				\end{equation}
				which is inserted into 
				(\ref{eq:takum-proof-monotonicity-goal}):
				\begin{equation}
					(c - \tilde{c}) + (m - \tilde{m}) < 0
					\iff
					-1 + m < 0
					\iff
					m < 1.
				\end{equation}
				The inequality is satisfied as $m \in [0,1)$.
			}
		\end{description}
		As both sub-cases fulfil (\ref{eq:takum-proof-monotonicity-goal}),
		this case also satisfies the desired inequality.
	}
	\item[Case 5 ($\textcolor{characteristic}{C} \neq \bm{1},
	\textcolor{mantissa}{M} = \bm{1}$):]{
		It holds
		\begin{align}
			(\textcolor{sign}{S},\textcolor{direction}{D},\textcolor{regime}{R},
			\textcolor{characteristic}{C},\textcolor{mantissa}{M})
			&= (\textcolor{sign}{S},
			\textcolor{direction}{D},
			\textcolor{regime}{R},
			\textcolor{characteristic}{C},
			\textcolor{mantissa}{\bm{1}}),\\
			\left(\textcolor{sign}{\tilde{S}},\textcolor{direction}{\tilde{D}},
			\textcolor{regime}{\tilde{R}},
			\textcolor{characteristic}{\tilde{C}},\textcolor{mantissa}{\tilde{M}}\right)
			&= (\textcolor{sign}{S},
			\textcolor{direction}{D},
			\textcolor{regime}{R},
			\textcolor{characteristic}{C} + 1,
			\textcolor{mantissa}{\bm{0}}),
		\end{align}
		which yields $r = \tilde{r}$ and $\tilde{c} = c + 1$,
		as $\textcolor{characteristic}{C} \neq \bm{1}$.
		The mantissas $m$ and $\tilde{m}$ follow as
		$m \in [0,1)$ and $\tilde{m} = 0$ due to the respectively
		saturated and zero mantissa bits. Inserting these results
		into (\ref{eq:takum-proof-monotonicity-goal}) it follows
		\begin{equation}
			(c - \tilde{c}) + (m - \tilde{m}) < 0
			\iff
			-1 + m < 0
			\iff
			m < 1,
		\end{equation}
		which is satisfied as $m \in [0,1)$.
	}
	\item[Case 6 ($\textcolor{mantissa}{M} \neq \bm{1}$):]{
		It holds
		\begin{align}
			(\textcolor{sign}{S},\textcolor{direction}{D},\textcolor{regime}{R},
			\textcolor{characteristic}{C},\textcolor{mantissa}{M})
			&= (\textcolor{sign}{S},
			\textcolor{direction}{D},
			\textcolor{regime}{R},
			\textcolor{characteristic}{C},
			\textcolor{mantissa}{M}),\\
			\left(\textcolor{sign}{\tilde{S}},\textcolor{direction}{\tilde{D}},
			\textcolor{regime}{\tilde{R}},
			\textcolor{characteristic}{\tilde{C}},\textcolor{mantissa}{\tilde{M}}\right)
			&= (\textcolor{sign}{S},
			\textcolor{direction}{D},
			\textcolor{regime}{R},
			\textcolor{characteristic}{C},
			\textcolor{mantissa}{M} + 1),
		\end{align}
		which yields $r = \tilde{r}$, $c = \tilde{c}$
		and $\tilde{m} > m$, as $\textcolor{mantissa}{M} \neq \bm{1}$.
		Inserting these results into 
		(\ref{eq:takum-proof-monotonicity-goal}) it follows
		\begin{equation}
			(c - \tilde{c}) + (m - \tilde{m}) < 0
			\iff
			0 + m - \tilde{m} < 0
			\iff
			m < \tilde{m},
		\end{equation}
		which is satisfied.
	}
\end{description}
Given that the cases encompass all possible bit strings, we have demonstrated what needed to be shown.\qed
\section{Proof of Proposition~\ref{prop:takum-negation}: Takum Negation}
\label{sec:proof-takum-negation}
The challenge inherent in this proof lies in articulating the inversion and incrementation of the input bit string $(\textcolor{sign}{S},\textcolor{direction}{D},\textcolor{regime}{R}, \textcolor{characteristic}{C},\textcolor{mantissa}{M})$ in a manner that allows for the analysis of its impacts on the underlying number representation. It can be observed that negating and incrementing preserve any consecutive sequence of zero bits at the trailing end, owing to the carry mechanism. Moreover, upon segmentation of the bit string, it becomes apparent that the carry is propagated to the first non-zero segment, which undergoes an increment of 1. As the segment is non-zero, its inversion does not yield a saturated integer, thereby ensuring that the incrementation does not influence any higher-order bit segments.
\par
For this reason, we will examine all potential scenarios involving consecutive zero-segments at the end. The initial non-zero segment is thereby identified as inverted and incremented (without overflow), while any subsequent segments are solely subjected to inversion. It is noteworthy to emphasise the notation, wherein $\bm{0}$ can represent an empty bit string in the context of $\textcolor{characteristic}{C}$ for $r=0$ and $\textcolor{mantissa}{M}$. This allowance is justified by the fact that an empty bit string remains unchanged under both negation and incrementation operations; the negation of an empty bit string remains an empty bit string, and incrementing an empty bit string also results in an empty bit string.
\begin{description}
\item[Case 1 ($\textcolor{sign}{S} = \textcolor{direction}{D} = 0,
	\textcolor{regime}{R} = \textcolor{characteristic}{C} =
	\textcolor{mantissa}{M} = \bm{0}$):]{
	It holds
	\begin{equation}
		\takum\!\left(\left(\overline{\textcolor{sign}{0}},\overline{\textcolor{direction}{0}},
		\overline{\textcolor{regime}{\bm{0}}},
		\overline{\textcolor{characteristic}{\bm{0}}},
		\overline{\textcolor{mantissa}{\bm{0}}}\right)+1\right) =
		\takum\!\left(
			(\textcolor{sign}{0},\textcolor{direction}{0},
			\textcolor{regime}{\bm{0}},\textcolor{characteristic}{\bm{0}},
			\textcolor{mantissa}{\bm{0}})
		\right) = 0 = -0 =
		-\takum\!\left(
			(\textcolor{sign}{0},\textcolor{direction}{0},
			\textcolor{regime}{\bm{0}},\textcolor{characteristic}{\bm{0}},
			\textcolor{mantissa}{\bm{0}})
		\right).
	\end{equation}
}
\item[Case 2 ($\textcolor{sign}{S} \neq 0, \textcolor{direction}{D} = 0, 
		\textcolor{regime}{R} = \textcolor{characteristic}{C} =
		\textcolor{mantissa}{M} = \bm{0}$):]{
	
	$\textcolor{sign}{S} \neq 0$ directly implies $\textcolor{sign}{S} = 1$ and
	it holds
	\begin{equation}
		\takum\!\left(\left(\overline{\textcolor{sign}{1}},\overline{\textcolor{direction}{0}},
		\overline{\textcolor{regime}{\bm{0}}},
		\overline{\textcolor{characteristic}{\bm{0}}},
		\overline{\textcolor{mantissa}{\bm{0}}}\right)+1\right) =
		\takum\!\left(
			(\textcolor{sign}{1},\textcolor{direction}{0},
			\textcolor{regime}{\bm{0}},\textcolor{characteristic}{\bm{0}},
			\textcolor{mantissa}{\bm{0}})
		\right) = \mathrm{NaR} =
		\takum\!\left(
			(\textcolor{sign}{1},\textcolor{direction}{0},
			\textcolor{regime}{\bm{0}},\textcolor{characteristic}{\bm{0}},
			\textcolor{mantissa}{\bm{0}})
		\right).
	\end{equation}
}
\end{description}
For ease of notation we define
$\left(\textcolor{sign}{\tilde{S}},\textcolor{direction}{\tilde{D}},
\textcolor{regime}{\tilde{R}},
\textcolor{characteristic}{\tilde{C}},\textcolor{mantissa}{\tilde{M}}\right)
:= \left(\overline{\textcolor{sign}{S}},\overline{\textcolor{direction}{D}},
\overline{\textcolor{regime}{R}},
\overline{\textcolor{characteristic}{C}},
\overline{\textcolor{mantissa}{M}}\right)+1$ and
add a tilde to all corresponding variables of Definition~\ref{def:takum}
in all subsequent cases.
As $0$ and $\mathrm{NaR}$ have already been covered in the previous two
cases, we can as follows assume (using Proposition~\ref{prop:takum-uniqueness})
that
\begin{equation}
	\takum((\textcolor{sign}{S},\textcolor{direction}{D},\textcolor{regime}{R},
	\textcolor{characteristic}{C},\textcolor{mantissa}{M})) =
	{(-1)}^{\textcolor{sign}{S}} \euler^{\ell}
\end{equation}
and
\begin{equation}
	\takum\!\left(\left(\textcolor{sign}{\tilde{S}},\textcolor{direction}{\tilde{D}},
	\textcolor{regime}{\tilde{R}},
	\textcolor{characteristic}{\tilde{C}},\textcolor{mantissa}{\tilde{M}}\right)\right) =
	{(-1)}^{\textcolor{sign}{\tilde{S}}} \euler^{\tilde{\ell}}.
\end{equation}
The goal to satisfy
$\takum\!\left(\left(\textcolor{sign}{\tilde{S}},\textcolor{direction}{\tilde{D}},
\textcolor{regime}{\tilde{R}},
\textcolor{characteristic}{\tilde{C}},\textcolor{mantissa}{\tilde{M}}\right)\right) =
-\takum((\textcolor{sign}{S},\textcolor{direction}{D},\textcolor{regime}{R},
\textcolor{characteristic}{C},\textcolor{mantissa}{M}))$ can be simplified to
\begin{equation}
	\textcolor{sign}{\tilde{S}} = \overline{\textcolor{sign}{S}} \land
	\ell = \tilde{\ell}.
\end{equation}
Using (\ref{eq:takum-logarithmic}) we can further reduce it to
\begin{equation}
	\textcolor{sign}{\tilde{S}} = \overline{\textcolor{sign}{S}} \land
	(c + m = -\tilde{c} - \tilde{m}),
\end{equation}
which is equivalent to
\begin{equation}
	\textcolor{sign}{\tilde{S}} = \overline{\textcolor{sign}{S}} \land
	((c + \tilde{c}) + (m + \tilde{m}) = 0).
	\label{eq:proof-takum-negation-goal}
\end{equation}
\begin{description}
\item[Case 3 ($\textcolor{direction}{D} \neq 0,
	\textcolor{regime}{R} = \textcolor{characteristic}{C} =
	\textcolor{mantissa}{M} = \bm{0}$):]{
	It holds as $\textcolor{direction}{D} \neq 0$ directly
	implies $\textcolor{direction}{D} = 1$
	\begin{align}
		(\textcolor{sign}{S},\textcolor{direction}{D},\textcolor{regime}{R},
			\textcolor{characteristic}{C},\textcolor{mantissa}{M}) &=
			(\textcolor{sign}{S},
			\textcolor{direction}{1},
			\textcolor{regime}{\bm{0}},
			\textcolor{characteristic}{\bm{0}},
			\textcolor{mantissa}{\bm{0}}),\\
		\left(\textcolor{sign}{\tilde{S}},\textcolor{direction}{\tilde{D}},
			\textcolor{regime}{\tilde{R}},
			\textcolor{characteristic}{\tilde{C}},\textcolor{mantissa}{\tilde{M}}\right) &=
			\left(
				\overline{\textcolor{sign}{S}},
				\overline{\textcolor{direction}{D}} + 1,
				\textcolor{regime}{R},
				\textcolor{characteristic}{C},
				\textcolor{mantissa}{M}
			\right) =
			(\overline{\textcolor{sign}{S}},
			\textcolor{direction}{1},
			\textcolor{regime}{\bm{0}},
			\textcolor{characteristic}{\bm{0}},
			\textcolor{mantissa}{\bm{0}}),
	\end{align}
	which yields $\textcolor{sign}{\tilde{S}} = \overline{\textcolor{sign}{S}}$,
	$c = \tilde{c} = 0$ and
	$m = \tilde{m} = 0$. This trivially satisfies (\ref{eq:proof-takum-negation-goal})
	with $0+0=0$.
}
\item[Case 4 ($\textcolor{regime}{R} \neq \bm{0},
	\textcolor{characteristic}{C} = \textcolor{mantissa}{M} = \bm{0}$):]{
	It holds
	\begin{align}
		(\textcolor{sign}{S},\textcolor{direction}{D},\textcolor{regime}{R},
			\textcolor{characteristic}{C},\textcolor{mantissa}{M}) &=
			(\textcolor{sign}{S},
			\textcolor{direction}{D},
			\textcolor{regime}{R},
			\textcolor{characteristic}{\bm{0}},
			\textcolor{mantissa}{\bm{0}}),\\
		\left(\textcolor{sign}{\tilde{S}},\textcolor{direction}{\tilde{D}},
			\textcolor{regime}{\tilde{R}},
			\textcolor{characteristic}{\tilde{C}},\textcolor{mantissa}{\tilde{M}}\right) &=
			\left(
				\overline{\textcolor{sign}{S}},
				\overline{\textcolor{direction}{D}},
				\overline{\textcolor{regime}{R}} + 1,
				\textcolor{characteristic}{\bm{0}},
				\textcolor{mantissa}{\bm{0}}
			\right),
	\end{align}
	which yields $\textcolor{sign}{\tilde{S}} = \overline{\textcolor{sign}{S}}$,
	$\textcolor{characteristic}{C} = \textcolor{characteristic}{\tilde{C}} = \bm{0}$ and
	$\textcolor{mantissa}{M} = \textcolor{mantissa}{\tilde{M}} = \bm{0}$,
	which implies $m = \tilde{m} = 0$.
	With $\textcolor{regime}{R} \neq \bm{0}$
	it also holds $\overline{\textcolor{regime}{R}} \neq \bm{1}$.
	\begin{description}
		\item[Case 4a ($\textcolor{direction}{D} = 0$):]{
			We know $\textcolor{direction}{\tilde{D}} = 1$,
			$r \in \{0,\dots,6\}$ and
			\begin{equation}
				\tilde{r} =
				\unsignedint(\overline{\textcolor{regime}{R}}+1) =
				\unsignedint(\overline{\textcolor{regime}{R}}) + 1 =
				(7 - \unsignedint(\textcolor{regime}{R})) + 1
				(7 - (7-r)) + 1 = r+1.
			\end{equation}
			Given $\textcolor{characteristic}{C} = \textcolor{characteristic}{\tilde{C}} = \bm{0}$
			it thus holds
			\begin{align}
				c + \tilde{c} &=
					\left( -2^{r+1} +1 + 0 \right) +
					\left( 2^{\tilde{r}} - 1 + 0 \right)\\
				&=
					\left( -2^{r+1} +1 + 0 \right) +
					\left( 2^{r+1} - 1 + 0 \right)\\
				&= 0.
			\end{align}
			This trivially satisfies (\ref{eq:proof-takum-negation-goal})
			with $0+0=0$.
		}
		\item[Case 4b ($\textcolor{direction}{D} = 1$):]{
			We know $\textcolor{direction}{\tilde{D}} = 0$,
			$r \in \{1,\dots,7\}$ and
			\begin{equation}
				\tilde{r} =
				7 - \unsignedint(\overline{\textcolor{regime}{R}}+1) =
				6 - \unsignedint(\overline{\textcolor{regime}{R}}) =
				6 - (7 - r) = r-1.
			\end{equation}
			Given $\textcolor{characteristic}{C} = \textcolor{characteristic}{\tilde{C}} = \bm{0}$
			it thus holds
			\begin{align}
				c + \tilde{c} &=
					\left( 2^r - 1 + 0 \right) +
					\left( -2^{\tilde{r}+1} + 1 + 0 \right)\\
				&=
					\left( 2^r - 1 + 0 \right) +
					\left( -2^{r} + 1 + 0 \right)\\
				&= 0.
			\end{align}
			This trivially satisfies (\ref{eq:proof-takum-negation-goal})
			with $0+0=0$.
		}
	\end{description}
	As both sub-cases fulfil (\ref{eq:proof-takum-negation-goal}),
	this case also satisfies the desired inequality.
}
\item[Case 5 ($\textcolor{characteristic}{C} \neq \bm{0},
	\textcolor{mantissa}{M} = \bm{0}$):]{
	We know that $\textcolor{characteristic}{C} \neq 0$ implies $r>0$, which means
	that the segment $\textcolor{characteristic}{C}$ has non-zero length, as
	it would otherwise neither stop the carry propagation nor be incremented. It holds
	\begin{align}
		(\textcolor{sign}{S},\textcolor{direction}{D},\textcolor{regime}{R},
			\textcolor{characteristic}{C},\textcolor{mantissa}{M}) &=
			(\textcolor{sign}{S},
			\textcolor{direction}{D},
			\textcolor{regime}{R},
			\textcolor{characteristic}{C},
			\textcolor{mantissa}{\bm{0}}),\\
		\left(\textcolor{sign}{\tilde{S}},\textcolor{direction}{\tilde{D}},
			\textcolor{regime}{\tilde{R}},
			\textcolor{characteristic}{\tilde{C}},\textcolor{mantissa}{\tilde{M}}\right) &=
			\left(
				\overline{\textcolor{sign}{S}},
				\overline{\textcolor{direction}{D}},
				\overline{\textcolor{regime}{R}},
				\overline{\textcolor{characteristic}{C}}+1,
				\textcolor{mantissa}{\bm{0}}
			\right),
	\end{align}
	which yields $\textcolor{sign}{\tilde{S}} = \overline{\textcolor{sign}{S}}$.
	With $\textcolor{direction}{\tilde{D}} = \overline{\textcolor{direction}{D}}$
	and $\textcolor{regime}{\tilde{R}} = \overline{\textcolor{regime}{R}}$ we
	observe a double inversion (see (\ref{eq:takum-regime}) and
	Lemma~\ref{lem:unsigned_integer_negation}):
	\begin{equation}
		\tilde{r} = 7 - (7 - r) = r.
	\end{equation}
	Using $\tilde{r}=r$ (operation 1) and $\textcolor{characteristic}{\tilde{C}} =
	\overline{\textcolor{characteristic}{C}} + 1$ and Lemma~\ref{lem:unsigned_integer_negation} (operation 2)
	we can deduce for the characteristics
	\begin{align}
		c + \tilde{c} &\overset{1}{=}
			(-2^{r+1}+1+2^r-1) + \sum_{i=0}^{r-1} {\textcolor{characteristic}{\tilde{C}}}_i 2^i +
			\sum_{i=0}^{r-1} {\textcolor{characteristic}{C}}_i 2^i\\
		&\overset{2}{=}
			-2^r + \left[(2^r - 1) - \left(\sum_{i=0}^{r-1}
			{\textcolor{characteristic}{C}}_i 2^i \right) + 1 \right] +
			\sum_{i=0}^{r-1} {\textcolor{characteristic}{C}}_i 2^i\\
		= 0.
	\end{align}
	Given $\textcolor{mantissa}{M} = \textcolor{mantissa}{\tilde{M}} = \bm{0}$
	it holds $m = \tilde{m} = 0$. This trivially satisfies (\ref{eq:proof-takum-negation-goal})
	with $0+0=0$.
}
\item[Case 6 ($\textcolor{mantissa}{M} \neq \bm{0}$):]{
	We know that $\textcolor{mantissa}{M} \neq 0$ implies $p>0$, which means
	that the segment $\textcolor{mantissa}{M}$ has non-zero length, as
	it would otherwise neither stop the carry propagation nor be incremented. It holds
	\begin{align}
		(\textcolor{sign}{S},\textcolor{direction}{D},\textcolor{regime}{R},
			\textcolor{characteristic}{C},\textcolor{mantissa}{M}) &=
			(\textcolor{sign}{S},
			\textcolor{direction}{D},
			\textcolor{regime}{R},
			\textcolor{characteristic}{C},
			\textcolor{mantissa}{M}),\\
		\left(\textcolor{sign}{\tilde{S}},\textcolor{direction}{\tilde{D}},
			\textcolor{regime}{\tilde{R}},
			\textcolor{characteristic}{\tilde{C}},\textcolor{mantissa}{\tilde{M}}\right) &=
			\left(
				\overline{\textcolor{sign}{S}},
				\overline{\textcolor{direction}{D}},
				\overline{\textcolor{regime}{R}},
				\overline{\textcolor{characteristic}{C}},
				\overline{\textcolor{mantissa}{M}} + 1
			\right),
	\end{align}
	which yields $\textcolor{sign}{\tilde{S}} = \overline{\textcolor{sign}{S}}$.
	With $\textcolor{direction}{\tilde{D}} = \overline{\textcolor{direction}{D}}$
	and $\textcolor{regime}{\tilde{R}} = \overline{\textcolor{regime}{R}}$ we
	observe a double inversion (see (\ref{eq:takum-regime}) and
	Lemma~\ref{lem:unsigned_integer_negation}):
	\begin{equation}
		\tilde{r} = 7 - (7 - r) = r.
	\end{equation}
	Using $\tilde{r}=r$ (operation 1) and $\textcolor{characteristic}{\tilde{C}} =\overline{\textcolor{characteristic}{C}}$ and Lemma~\ref{lem:unsigned_integer_negation} (operation 2)
	we can deduce for the characteristics
	\begin{align}
		c + \tilde{c} &\overset{1}{=}
			(-2^{r+1}+1+2^r-1) + \sum_{i=0}^{r-1} {\textcolor{characteristic}{\tilde{C}}}_i 2^i +
			\sum_{i=0}^{r-1} {\textcolor{characteristic}{C}}_i 2^i\\
		&\overset{2}{=}
			-2^r + \left[(2^r - 1) - \left(\sum_{i=0}^{r-1}
			{\textcolor{characteristic}{C}}_i 2^i \right) \right] +
			\sum_{i=0}^{r-1} {\textcolor{characteristic}{C}}_i 2^i\\
		&= -1.
	\end{align}
	Using $\tilde{p}=p$ (operation 1), which holds because $\tilde{r}=r$, and
	$\textcolor{mantissa}{\tilde{M}} =\overline{\textcolor{mantissa}{M}}+1$
	and Lemma~\ref{lem:unsigned_integer_negation} (operation 2) we can show
	for the mantissa
	\begin{align}
		\tilde{m} &\overset{1}{=} 2^{-p} \sum_{i=0}^{p-1} {\textcolor{mantissa}{\tilde{M}}}_i 2^i\\
		&\overset{2}{=} 2^{-p} \left[
				(2^p - 1) - \left(\sum_{i=0}^{p-1} {\textcolor{mantissa}{M}}_i 2^i \right) + 1
			\right]\\
		&= 1 - 2^{-p} \sum_{i=0}^{p-1} {\textcolor{mantissa}{M}}_i 2^i\\
		&= 1 - m,
	\end{align}
	which is equivalent to
	\begin{equation}
		m + \tilde{m} = 1.
	\end{equation}
	In total we obtain for (\ref{eq:proof-takum-negation-goal})
	\begin{equation}
		(a + \tilde{a}) + (m + \tilde{m}) = 0 \iff
		-1 + 1 = 0 \iff 0=0,
	\end{equation}
	which concludes that it is satisfied.
}
\end{description}
Given that the cases encompass all possible bit strings, we have proven what was intended to be shown.\qed
\section{Proof of Lemma~\ref{lem:takum-inversion-negation}: Takum 
	Inversion-Negation}
\label{sec:proof-takum-inversion-negation}
If $\takum((\textcolor{sign}{S},\textcolor{direction}{D},
\textcolor{regime}{R}, \textcolor{characteristic}{C},\textcolor{mantissa}{M})) = 0$
then the inversion of the sign bit directly yields $\mathrm{NaR}$ by
definition. Otherwise we know
\begin{align}
	\takum\!\left(\left(\overline{\textcolor{sign}{S}},\textcolor{direction}{D},
		\textcolor{regime}{R}, \textcolor{characteristic}{C},\textcolor{mantissa}{M}\right)\right) &=
		{(-1)}^{\overline{\textcolor{sign}{S}}} \euler^{{(-1)}^{\overline{\textcolor{sign}{S}}}(c+m)}\\
	&= -\frac{1}{{(-1)}^{\textcolor{sign}{S}} \euler^{{(-1)}^{\textcolor{sign}{S}} (c+m) }}\\
	&= -\frac{1}{{(-1)}^{\textcolor{sign}{S}} \euler^{\ell}}\\
	&= -\frac{1}{\takum((\textcolor{sign}{S},\textcolor{direction}{D},
		\textcolor{regime}{R}, \textcolor{characteristic}{C},\textcolor{mantissa}{M}))},
\end{align}
which was to be shown.\qed
\section{Proof of Proposition~\ref{prop:takum-inversion}: Takum Inversion}
\label{sec:proof-takum-inversion}
Let $\left(\textcolor{sign}{\tilde{S}},\textcolor{direction}{\tilde{D}},
\textcolor{regime}{\tilde{R}},
\textcolor{characteristic}{\tilde{C}},\textcolor{mantissa}{\tilde{M}}\right) :=
\left(\overline{\textcolor{sign}{S}},\textcolor{direction}{D},\textcolor{regime}{R},
\textcolor{characteristic}{C},\textcolor{mantissa}{M}\right)$. With Lemma~\ref{lem:takum-inversion-negation}
we know that
\begin{equation}
	\takum\!\left(\left(\textcolor{sign}{\tilde{S}},\textcolor{direction}{\tilde{D}},
		\textcolor{regime}{\tilde{R}},
		\textcolor{characteristic}{\tilde{C}},\textcolor{mantissa}{\tilde{M}}\right)\right) =
	\begin{cases}
		-\frac{1}{\takum((\textcolor{sign}{S},\textcolor{direction}{D},\textcolor{regime}{R},
					\textcolor{characteristic}{C},\textcolor{mantissa}{M}))} &
			\takum((\textcolor{sign}{S},\textcolor{direction}{D},
			\textcolor{regime}{R},
			\textcolor{characteristic}{C},\textcolor{mantissa}{M})) \neq
			0\\
		\mathrm{NaR} & \takum((\textcolor{sign}{S},\textcolor{direction}{D},
			\textcolor{regime}{R},
			\textcolor{characteristic}{C},\textcolor{mantissa}{M})) = 0.\label{eq:proof-takum-inversion-1}
	\end{cases}
\end{equation}
Applying Proposition~\ref{prop:takum-negation} and inserting (\ref{eq:proof-takum-inversion-1})
we can derive
\begin{align}
	\takum\!\left(\!\left(\overline{\textcolor{sign}{\tilde{S}}},\overline{\textcolor{direction}{\tilde{D}}},
	\overline{\textcolor{regime}{\tilde{R}}},
	\overline{\textcolor{characteristic}{\tilde{C}}},\overline{\textcolor{mantissa}{\tilde{M}}}\right)
	\!+\!1\!\right) &=
	\begin{cases}
		-\takum\!\left(\left(\textcolor{sign}{\tilde{S}},\textcolor{direction}{\tilde{D}},
			\textcolor{regime}{\tilde{R}},
			\textcolor{characteristic}{\tilde{C}},\textcolor{mantissa}{\tilde{M}}\right)\right) & \!
			\takum\!\left(\left(\textcolor{sign}{\tilde{S}},\textcolor{direction}{\tilde{D}},
			\textcolor{regime}{\tilde{R}},
			\textcolor{characteristic}{\tilde{C}},\textcolor{mantissa}{\tilde{M}}\right)\right) \!\neq\!
			\mathrm{NaR}\\
		\mathrm{NaR} & \!\takum\!\left(\left(\textcolor{sign}{\tilde{S}},\textcolor{direction}{\tilde{D}},
			\textcolor{regime}{\tilde{R}},
			\textcolor{characteristic}{\tilde{C}},\textcolor{mantissa}{\tilde{M}}\right)\right)\!=\!
			\mathrm{NaR}
	\end{cases}\\
	&= \begin{cases}
		-\takum\!\left(\left(\textcolor{sign}{\tilde{S}},\textcolor{direction}{\tilde{D}},
			\textcolor{regime}{\tilde{R}},
			\textcolor{characteristic}{\tilde{C}},\textcolor{mantissa}{\tilde{M}}\right)\right) &
			\takum((\textcolor{sign}{S},\textcolor{direction}{D},
			\textcolor{regime}{R},
			\textcolor{characteristic}{C},\textcolor{mantissa}{M})) \neq 0\\
		\mathrm{NaR} & \takum((\textcolor{sign}{S},\textcolor{direction}{D},
			\textcolor{regime}{R},
			\textcolor{characteristic}{C},\textcolor{mantissa}{M})) = 0
	\end{cases}\\
	&= \begin{cases}
		\frac{1}{\takum((\textcolor{sign}{S},\textcolor{direction}{D},\textcolor{regime}{R},
			\textcolor{characteristic}{C},\textcolor{mantissa}{M}))} &
			\takum((\textcolor{sign}{S},\textcolor{direction}{D},
			\textcolor{regime}{R},
			\textcolor{characteristic}{C},\textcolor{mantissa}{M})) \neq 0\\
		\mathrm{NaR} & \takum((\textcolor{sign}{S},\textcolor{direction}{D},
			\textcolor{regime}{R},
			\textcolor{characteristic}{C},\textcolor{mantissa}{M})) = 0.
	\end{cases}
\end{align}
Given by construction
\begin{equation}
	\left(\overline{\textcolor{sign}{\tilde{S}}},\overline{\textcolor{direction}{\tilde{D}}},
	\overline{\textcolor{regime}{\tilde{R}}},
	\overline{\textcolor{characteristic}{\tilde{C}}},\overline{\textcolor{mantissa}{\tilde{M}}}\right)+1 =
	\left(\overline{\overline{\textcolor{sign}{S}}},\overline{\textcolor{direction}{D}},
	\overline{\textcolor{regime}{R}},
	\overline{\textcolor{characteristic}{C}},\overline{\textcolor{mantissa}{M}}\right)+1 =
	\left(\textcolor{sign}{S},\overline{\textcolor{direction}{D}},
	\overline{\textcolor{regime}{R}}, \overline{\textcolor{characteristic}{C}},
	\overline{\textcolor{mantissa}{M}}\right) + 1
\end{equation}
holds we have proven what was to be shown.\qed
\section{Proof of Proposition~\ref{prop:takum-encoding}: Takum Floating-Point 
	Encoding}
\label{sec:proof-takum-encoding}
The following equality must hold:
\begin{equation}
	\takum((\textcolor{sign}{S},\textcolor{direction}{D},\textcolor{regime}{R},
		\textcolor{characteristic}{C},\textcolor{mantissa}{M})) = x,
\end{equation}
which is given in all cases, as $x$ is within $t$'s image by
construction. As $x \neq 0$ the equation is equivalent to
\begin{equation}
	{(-1)}^s\cdot \euler^{h \ln(2) + \ln(1+f)} =
	{(-1)}^s\cdot (1+f) \cdot 2^h = {(-1)}^{\textcolor{sign}{S}} \euler^{\ell}.
\end{equation}
With $\textcolor{sign}{S}=s$ and using Definition~\ref{def:takum} we
obtain
\begin{equation}
	2 \left( h \ln(2) + \ln(1+f) \right) = h \log_{\euler}(2) + \log_{\euler}(1+f) = \ell = {(-1)}^{\textcolor{sign}{S}} (c+m).
\end{equation}
This yields \ref{eq:takum-encoding-exponent-condition} with $\ell \in (-255,255)$ and
is equivalent to
\begin{equation}
	{(-1)}^{\textcolor{sign}{S}} \ell = c+m.
\end{equation}
Applying the floor function yields
\begin{equation}
	c = \left\lfloor {(-1)}^{\textcolor{sign}{S}} \ell \right\rfloor,
\end{equation}
and likewise
\begin{equation}
	m = {(-1)}^{\textcolor{sign}{S}} \ell -
		\left\lfloor {(-1)}^{\textcolor{sign}{S}} \ell \right\rfloor =
		{(-1)}^{\textcolor{sign}{S}} \ell - c.
\end{equation}
We know that $c$ is always negative and positive for
$\textcolor{direction}{D}=0$ and $\textcolor{direction}{D}=1$
respectively. Thus we can 
deduce $\textcolor{direction}{D} = c \ge 0$. The next step is to 
determine the regime $r$. We take $c$ and insert the definition
\begin{align}
	c &= \begin{cases}
		-2^{r+1} + 1 + \sum_{i=0}^{r-1} \textcolor{characteristic}{C}_i 2^{i}
		& \textcolor{direction}{D} = 0\\
		2^r - 1 + \sum_{i=0}^{r-1} \textcolor{characteristic}{C}_i 2^{i}
		& \textcolor{direction}{D} = 1\\
	\end{cases}\\
	&\in \begin{cases}
		\{ -2^{r+1}+1,\dots, -2^{r+1} + 1 + (2^r-1) \} & 
		\textcolor{direction}{D} = 0\\
		\{ 2^r-1,\dots, 2^r - 1 + (2^r-1) \} & \textcolor{direction}{D} = 1
	\end{cases}\\
	&= \begin{cases}
		\{ -2^{r+1}+1, \dots, -2^r \} & \textcolor{direction}{D} = 0\\
		\{ 2^r-1,\dots, 2^{r+1} - 2 \} & \textcolor{direction}{D} = 1.
	\end{cases}
\end{align}
In particular it holds for $\textcolor{direction}{D} = 0$
\begin{equation}
	-c \in \{ 2^r,\dots,2^{r+1}-1 \}
\end{equation}
and for $\textcolor{direction}{D} = 1$
\begin{equation}
	c+1 \in \{2^r,\dots,2^{r+1}-1\},
\end{equation}
which yields
\begin{equation}
	r = \begin{cases}
		\lfloor \log_2(-c) \rfloor & \textcolor{direction}{D} = 0\\
		\lfloor \log_2(c+1) \rfloor & \textcolor{direction}{D} = 1.
	\end{cases}
\end{equation}
The regime bits follow directly as
\begin{equation}
	\textcolor{regime}{R} = \begin{cases}
		7 -r & \textcolor{direction}{D} = 0\\
		r & \textcolor{direction}{D} = 1.
	\end{cases}
\end{equation}
The characteristic bits result directly from the definition
\begin{equation}
	\textcolor{characteristic}{C} = \begin{cases}
		c + 2^{r+1} - 1 & \textcolor{direction}{D} = 0\\
		c - 2^r + 1 & \textcolor{direction}{D} = 1.
	\end{cases}
\end{equation}
The mantissa bits $\textcolor{mantissa}{M}$ are obtained by
multiplying $m$ with $2^p$, where $p \in \mathbb{N}_0 \cup \{ \infty \}$
is sufficiently large such that $2^p m \in \mathbb{N}_0$,
which might yield infinitely many mantissa bits:
\begin{align}
	p &= \inf_{i \in \mathbb{N}_0} \!\left(2^i m \in \mathbb{N}_0\right) \in
	\mathbb{N}_0 \cup \{ \infty \},\\
	\textcolor{mantissa}{M} &= 2^p m \in {\{0,1\}}^p.
\end{align}
In total we have obtained all bits of the takum representing
the given floating-point value $x$.\qed
\section{Proof of Proposition~\ref{prop:takum-mantissa_bit_count}: Takum Mantissa Bit Count}
\label{sec:proof-takum-mantissa_bit_count}
Let $X =: (\textcolor{sign}{S},\textcolor{direction}{D},\textcolor{regime}{R},
\textcolor{characteristic}{C},\textcolor{mantissa}{M})$.
With $x \notin \{ 0, \mathrm{NaR} \}$ we know $x$ has the form
\begin{align}
	x = {(-1)}^{\textcolor{sign}{S}} \euler^\ell &\implies
		|x| = \euler^\ell\\
	&\implies \log_{\euler}(|x|) = 2\ln(|x|) = \ell = {(-1)}^{\textcolor{sign}{S}} (c+m)\\
	&\implies 2\ln(|x|) \sign(x) = c + m\\
	&\implies \lfloor 2\ln(|x|) \sign(x) \rfloor = c.
\end{align}
As $c+m$ is negative for $\textcolor{direction}{D} = 0$ and positive
for $\textcolor{direction}{D} = 1$ we know
\begin{equation}
	\textcolor{direction}{D} = (c+m) \ge 0 = 2\ln(|x|) \sign(x) \ge 0 =
	\ln(|x|) \sign(x) \ge 0.
\end{equation}
We can also derive
\begin{align}
	|c + \textcolor{direction}{D}| &= \left| \begin{cases}
		-2^{r+1} + 1 + \sum_{i=0}^{r-1}
			\textcolor{characteristic}{C}_i 2^i & \textcolor{direction}{D} = 0\\
		2^r -1 + \left(\sum_{i=0}^{r-1}
			\textcolor{characteristic}{C}_i 2^i\right) + 1 & \textcolor{direction}{D} = 1
	\end{cases} \right|\\
	&\in \left| \begin{cases}
		\{ -2^{r+1} + 1,\dots,-2^{r+1} + 1 + (2^r - 1) \} & \textcolor{direction}{D} = 0\\
		\{ 2^r,\dots,2^r + (2^r - 1) \} & \textcolor{direction}{D} = 1
	\end{cases} \right|\\
	&= \left| \begin{cases}
		\{ -2^{r+1} + 1,\dots,-2^{r} \} & \textcolor{direction}{D} = 0\\
		\{ 2^r,\dots,2^{r+1} - 1 \} & \textcolor{direction}{D} = 1
	\end{cases} \right|\\
	&= \{ 2^r,\dots,2^{r+1} - 1 \}.
\end{align}
It follows directly that
\begin{equation}
	r = \left\lfloor \log_2\!\left(\left|c + \textcolor{direction}{D}\right|\right) \right\rfloor,
\end{equation}
and thus
\begin{equation}
	r = \left\lfloor \log_2\!\left(\left|
		\lfloor 2\ln(|x|) \sign(x) \rfloor +
		(\ln(|x|) \sign(x) \ge 0)
	\right|\right) \right\rfloor.
\end{equation}
With $p = n - r - 5$ the proposition follows.\qed
\section{Proof of Proposition~\ref{prop:takum-mantissa_bit_count-lower_bound}:
	Takum Mantissa Bit Count Lower Bound}
\label{sec:proof-takum-mantissa_bit_count-lower_bound}
Given $x \in \pm(\euler^{-255},\euler^{255})$, we immediately proceed to the final clause in Algorithm~\ref{alg:rounding}. Let us denote $X := \truncate_n(\takuminv(x))$. Given $n \ge 12$, we ascertain that the truncation solely impacts the mantissa bits, leaving the sign, direction, regime bits, and characteristic bits unaffected. Consequently, the regime $r$ remains unaltered. Hence, the number of mantissa bits is
\begin{multline}
	p = n - 5 -
	\left\lfloor \log_2\!\left(\left|
		\lfloor 2\ln(|x|) \sign(x) \rfloor +
		(\ln(|x|) \sign(x) \ge 0)
	\right|\right) \right\rfloor
	\in\\ \{ n-12,\dots,n-5 \}
\end{multline}
as per Proposition~\ref{prop:takum-mantissa_bit_count}.
However, $X$ may be subject to incrementation during rounding. The resultant outcomes may vary depending on the markup of the distinct segments, which we will analyse individually on a case-by-case basis as delineated below.
Let us first define
\begin{align}
	X &=:
	(\textcolor{sign}{S},\textcolor{direction}{D},\textcolor{regime}{R},
	\textcolor{characteristic}{C},\textcolor{mantissa}{M}),\\
	X + 1 &=:
	\left(\textcolor{sign}{\tilde{S}},\textcolor{direction}{\tilde{D}},
	\textcolor{regime}{\tilde{R}},
	\textcolor{characteristic}{\tilde{C}},\textcolor{mantissa}{\tilde{M}}\right),
\end{align}
(applying tildes to the corresponding variables from
Definition~\ref{def:takum}) and then check $r$ and $\tilde{r}$ in all cases.
The notation $\bm{1}$ for a bit string composed entirely of ones may seem somewhat unconventional, particularly when considering the allowance for empty bit strings in the case of the characteristic and mantissa bits $\textcolor{characteristic}{C}$ and $\textcolor{mantissa}{M}$. However, this allowance is justifiable, as an empty bit string behaves equivalently to a bit string of all ones in certain contexts: An empty bit string, when interpreted as one of all ones, retains this equivalence upon incrementation, effectively propagating any carries passed to it. Specifically, case 5 would not be triggered by an empty $\textcolor{characteristic}{C}$, and case 6 would not be triggered by an empty $\textcolor{mantissa}{M}$, since an empty bit string is distinct from a bit string of all ones.
\begin{description}
	\item[Case 1 ($\textcolor{sign}{S} = \textcolor{direction}{D} = 1,
	\textcolor{regime}{R} = \textcolor{characteristic}{C} =
	\textcolor{mantissa}{M} = \bm{1}$):]{
		It holds
		\begin{align}
			(\textcolor{sign}{S},\textcolor{direction}{D},\textcolor{regime}{R},
			\textcolor{characteristic}{C},\textcolor{mantissa}{M})
			&= (\textcolor{sign}{1},
			\textcolor{direction}{1},
			\textcolor{regime}{\bm{1}},
			\textcolor{characteristic}{\bm{1}},
			\textcolor{mantissa}{\bm{1}}),\\
			\left(\textcolor{sign}{\tilde{S}},\textcolor{direction}{\tilde{D}},
			\textcolor{regime}{\tilde{R}},
			\textcolor{characteristic}{\tilde{C}},\textcolor{mantissa}{\tilde{M}}\right)
			&= (\textcolor{sign}{0},
			\textcolor{direction}{0},
			\textcolor{regime}{\bm{0}},
			\textcolor{characteristic}{\bm{0}},
			\textcolor{mantissa}{\bm{0}}).
		\end{align}
		This overflow to zero would be reset during the subsequent saturation step outlined in Algorithm~\ref{alg:rounding}. Consequently, in this case, the regime value remains unaffected.

	}
	\item[Case 2 ($\textcolor{sign}{S} \neq 1, \textcolor{direction}{D} = 
	1, 
	\textcolor{regime}{R} = \textcolor{characteristic}{C} =
	\textcolor{mantissa}{M} = \bm{1}$):]{
		$\textcolor{sign}{S} \neq 1$ directly implies
		$\textcolor{sign}{S} = 0$ and it holds
		\begin{align}
			(\textcolor{sign}{S},\textcolor{direction}{D},\textcolor{regime}{R},
			\textcolor{characteristic}{C},\textcolor{mantissa}{M})
			&= (\textcolor{sign}{0},
			\textcolor{direction}{1},
			\textcolor{regime}{\bm{1}},
			\textcolor{characteristic}{\bm{1}},
			\textcolor{mantissa}{\bm{1}}),\\
			\left(\textcolor{sign}{\tilde{S}},\textcolor{direction}{\tilde{D}},
			\textcolor{regime}{\tilde{R}},
			\textcolor{characteristic}{\tilde{C}},\textcolor{mantissa}{\tilde{M}}\right)
			&= (\textcolor{sign}{1},
			\textcolor{direction}{0},
			\textcolor{regime}{\bm{0}},
			\textcolor{characteristic}{\bm{0}},
			\textcolor{mantissa}{\bm{0}}).
		\end{align}
		This overflow to $\mathrm{NaR}$ would be reset during the subsequent saturation step outlined in Algorithm~\ref{alg:rounding}. Consequently, in this case, the regime value remains unaffected.
	}
	\item[Case 3 ($\textcolor{direction}{D} \neq 1,
	\textcolor{regime}{R} = \textcolor{characteristic}{C} =
	\textcolor{mantissa}{M} = \bm{1}$):]{
		$\textcolor{direction}{D} \neq 1$ directly implies
		$\textcolor{direction}{D} = 0$ and it holds
		\begin{align}
			(\textcolor{sign}{S},\textcolor{direction}{D},\textcolor{regime}{R},
			\textcolor{characteristic}{C},\textcolor{mantissa}{M})
			&= (\textcolor{sign}{S},
			\textcolor{direction}{0},
			\textcolor{regime}{\bm{1}},
			\textcolor{characteristic}{\bm{1}},
			\textcolor{mantissa}{\bm{1}}),\\
			\left(\textcolor{sign}{\tilde{S}},\textcolor{direction}{\tilde{D}},
			\textcolor{regime}{\tilde{R}},
			\textcolor{characteristic}{\tilde{C}},\textcolor{mantissa}{\tilde{M}}\right)
			&= (\textcolor{sign}{S},
			\textcolor{direction}{1},
			\textcolor{regime}{\bm{0}},
			\textcolor{characteristic}{\bm{0}},
			\textcolor{mantissa}{\bm{0}}),
		\end{align}
		which yields $r=\tilde{r}=0$.
	}
	\item[Case 4 ($\textcolor{regime}{R} \neq \bm{1},
	\textcolor{characteristic}{C} = \textcolor{mantissa}{M} = \bm{1}$):]{
		It holds
		\begin{align}
			(\textcolor{sign}{S},\textcolor{direction}{D},\textcolor{regime}{R},
			\textcolor{characteristic}{C},\textcolor{mantissa}{M})
			&= (\textcolor{sign}{S},
			\textcolor{direction}{D},
			\textcolor{regime}{R},
			\textcolor{characteristic}{\bm{1}},
			\textcolor{mantissa}{\bm{1}}),\\
			\left(\textcolor{sign}{\tilde{S}},\textcolor{direction}{\tilde{D}},
			\textcolor{regime}{\tilde{R}},
			\textcolor{characteristic}{\tilde{C}},\textcolor{mantissa}{\tilde{M}}\right)
			&= (\textcolor{sign}{S},
			\textcolor{direction}{D},
			\textcolor{regime}{R} + 1,
			\textcolor{characteristic}{\bm{0}},
			\textcolor{mantissa}{\bm{0}}),
		\end{align}
		where it shall be noted that a change in the regime leads to a shift
		of the characteristic and mantissa bits. However, as both segments are 
		zero, the shift is inconsequential.
		\begin{description}
			\item[Case 4a ($\textcolor{direction}{D} = 0$):]{
				It follows $r \in \{1,\dots,7\}$ and $\tilde{r} = r - 1$
				given $\textcolor{regime}{R} \neq 1$.
			}
			\item[Case 4b ($\textcolor{direction}{D} = 1$):]{
				It follows $r \in \{0,\dots,6\}$ and $\tilde{r} = r + 1$
				given $\textcolor{regime}{R} \neq 1$.
			}
		\end{description}
	}
	\item[Case 5 ($\textcolor{characteristic}{C} \neq \bm{1},
	\textcolor{mantissa}{M} = \bm{1}$):]{
		It holds
		\begin{align}
			(\textcolor{sign}{S},\textcolor{direction}{D},\textcolor{regime}{R},
			\textcolor{characteristic}{C},\textcolor{mantissa}{M})
			&= (\textcolor{sign}{S},
			\textcolor{direction}{D},
			\textcolor{regime}{R},
			\textcolor{characteristic}{C},
			\textcolor{mantissa}{\bm{1}}),\\
			\left(\textcolor{sign}{\tilde{S}},\textcolor{direction}{\tilde{D}},
			\textcolor{regime}{\tilde{R}},
			\textcolor{characteristic}{\tilde{C}},\textcolor{mantissa}{\tilde{M}}\right)
			&= (\textcolor{sign}{S},
			\textcolor{direction}{D},
			\textcolor{regime}{R},
			\textcolor{characteristic}{C} + 1,
			\textcolor{mantissa}{\bm{0}}),
		\end{align}
		which yields $r = \tilde{r}$.
	}
	\item[Case 6 ($\textcolor{mantissa}{M} \neq \bm{1}$):]{
		It holds
		\begin{align}
			(\textcolor{sign}{S},\textcolor{direction}{D},\textcolor{regime}{R},
			\textcolor{characteristic}{C},\textcolor{mantissa}{M})
			&= (\textcolor{sign}{S},
			\textcolor{direction}{D},
			\textcolor{regime}{R},
			\textcolor{characteristic}{C},
			\textcolor{mantissa}{M}),\\
			\left(\textcolor{sign}{\tilde{S}},\textcolor{direction}{\tilde{D}},
			\textcolor{regime}{\tilde{R}},
			\textcolor{characteristic}{\tilde{C}},\textcolor{mantissa}{\tilde{M}}\right)
			&= (\textcolor{sign}{S},
			\textcolor{direction}{D},
			\textcolor{regime}{R},
			\textcolor{characteristic}{C},
			\textcolor{mantissa}{M} + 1),
		\end{align}
		which yields $r = \tilde{r}$.
	}
\end{description}
As we can see the regime value is only changed in case~4 and
either incremented or decremented depending on $\textcolor{direction}{D}$.
It immediately follows
\begin{equation}
	\tilde{p} = \begin{cases}
		\begin{cases}
			p + 1 & \textcolor{direction}{D} = 0\\
			p - 1 & \textcolor{direction}{D} = 1
		\end{cases} & \textcolor{regime}{R} \neq \bm{1},
			\textcolor{characteristic}{C} = \textcolor{mantissa}{M} = \bm{1}\\
		p & \text{otherwise},
	\end{cases}
\end{equation}
in particular $\tilde{p} \in \{ p-1, p, p+1 \}$ and thus
$\tilde{p} \ge p-1$, which was to be shown.
\section{Proof of Proposition~\ref{prop:takum-precision}: Takum Machine 
	Precision}
\label{sec:proof-takum-precision}
Given $\round(0)=0$ and $x \neq \mathrm{NaR}$ we can assume
$x$ to have the form ${(-1)}^{\textcolor{sign}{S}} \euler^{\tilde{\ell}}$ with
$\tilde{\ell} \in (-255,255)$. The result of the rounding,
$\round(x)$, has the form ${(-1)}^{\textcolor{sign}{S}} \euler^\ell$.
The logarithmic value $\ell \in (-255,255)$ has $p$ mantissa bits and differs at
most by $\pm 2^{-p-1}$ from $\tilde{\ell}$ after rounding. It thus follows with
\begin{equation}
	\euler^{2^{-p-1}} - 1 \ge 1 - \euler^{-2^{-p-1}} \iff
	\euler^{-2^{-p-1}} {\left(\euler^{2^{-p-1}} - 1\right)}^2 \ge 0,
\end{equation}
which always holds, that
\begin{align}
	\left| \frac{x - \round(x)}{x} \right| &=
		\left| 1 - \frac{\round(x)}{x} \right|\\
	&= \left| 1 - \frac{{(-1)}^{\textcolor{sign}{S}} \euler^\ell}
		{{(-1)}^{\textcolor{sign}{S}} \euler^{\tilde{\ell}}} \right|\\
	&= \left| 1 - \euler^{\ell - \tilde{\ell}} \right|\\
	&\le \max\!\left(
		\left| 1 - \euler^{2^{-p-1}} \right|,
		\left| 1 - \euler^{-2^{-p-1}} \right|
		\right)\\
	&= \max\!\left(
			\euler^{2^{-p-1}} - 1,
			1 - \euler^{-2^{-p-1}}
		\right)\\
	&= \euler^{2^{-p-1}} - 1\\
	&= \lambda(p).
\end{align}
Next we show that $\lambda(p) < \frac{2}{3} \varepsilon(p)$. It holds
\begin{equation}\label{eq:takum-proof-prec-subtraction}
	\lambda(p) - \varepsilon(p) = \euler^{2^{-p-1}}- 2^{-p-1} - 1.
\end{equation}
We know that for $z \in (0,0.5)$ the function $\euler^z = \exp\!\left(\frac{z}{2}\right)$
has the \textsc{Taylor} polynomial with \textsc{Lagrange} remainder at $0$
\begin{equation}
	\exp\!\left(\frac{z}{2}\right) = 1 + \frac{z}{2} + \frac{\exp''(\xi)}{2!}
	{\left(\frac{z}{2}\right)}^2 =
		1 + \frac{z}{2} + \frac{\exp(\xi)}{2} {\left(\frac{z}{2}\right)}^2 <
		1 + \frac{z}{2} + \frac{\sqrt[4]{\mathrm{e}}}{2} {\left(\frac{z}{2}\right)}^2
\end{equation}
with $\xi \in \left(0,\frac{z}{2}\right) \subset (0,0.25)$, which is equivalent to
\begin{equation}
	\exp\!\left(\frac{z}{2}\right) - \frac{z}{2} - 1 < \frac{\sqrt[4]{\mathrm{e}}}{2}
	{\left(\frac{z}{2}\right)}^2 \iff
	\euler^z - z - 1 < \frac{\sqrt[4]{\mathrm{e}}}{2}
		{\left(\frac{z}{2}\right)}^2 - \frac{z}{2}.
\end{equation}
If we look back at (\ref{eq:takum-proof-prec-subtraction}) we can see that
with $z = 2^{-p-1} = \varepsilon(p) \in (0,0.5)$ and
$\varepsilon^2(p) < \varepsilon(p)$ it holds
\begin{align}
		\lambda(p) - \varepsilon(p) = \euler^{2^{-p-1}} - 2^{-p-1} - 1 &< \frac{\sqrt[4]{\mathrm{e}}}{2}
			{\left(\frac{\varepsilon(p)}{2}\right)}^2 -
			\frac{\varepsilon(p)}{2}\\
		&< \frac{\sqrt[4]{\mathrm{e}}}{8}
			\varepsilon(p) - \frac{\varepsilon(p)}{2}\\
		&= \left( \frac{\sqrt[4]{\mathrm{e}}}{8} - \frac{1}{2} \right)
			\varepsilon(p)\\
		&\approx -0.339 \cdot \varepsilon(p)\\
		&< -\frac{\varepsilon(p)}{3}.
\end{align}
This is equivalent to
\begin{equation}
	\lambda(p) < \frac{2}{3} \varepsilon(p),
\end{equation}
which was to be shown.\qed
\end{appendix}
\printbibliography
\end{document}